\documentclass[10pt]{amsart}

\usepackage{amsmath}
\usepackage{amssymb}
\usepackage{bm}
\usepackage{graphicx}
\usepackage{psfrag}
\usepackage{color}
\usepackage{hyperref}
\hypersetup{colorlinks=true, linkcolor=blue, citecolor=magenta, urlcolor=wine}
\usepackage{url}
\usepackage{algpseudocode}
\usepackage{fancyhdr}
\usepackage{mathtools}
\usepackage{tikz-cd}
\usepackage{xy}
\input xy
\xyoption{all}
\usepackage{stmaryrd}
\usepackage{calrsfs}

\newtheorem*{maintheorem*}{Main Theorem}
\newtheorem{theorem}{Theorem}[section]
\newtheorem*{theorem*}{Main Theorem}

\newtheorem{prop}[theorem]{Proposition}

\newtheorem{lemma}[theorem]{Lemma}
\newtheorem{cor}[theorem]{Corollary}

\theoremstyle{definition}
\newtheorem{definition}[theorem]{Definition}
\newtheorem{remark}[theorem]{Remark}
\newtheorem{example}[theorem]{Example}
\numberwithin{equation}{section}

\newcommand{\cc}{\mathbb{C}}
\newcommand{\ff}{\mathbb{F}}
\newcommand{\nn}{\mathbb{N}}
\newcommand{\pp}{\mathbb{P}}
\newcommand{\qq}{\mathbb{Q}}
\newcommand{\rr}{\mathbb{R}}
\newcommand{\zz}{\mathbb{Z}}

\newcommand{\aaa}{\mathbf{A}}

\providecommand\ldb{\llbracket}
\providecommand\rdb{\rrbracket}

\newcommand{\gp}{\text{gp}}

\newcommand{\ii}{\mathcal{Irr}}



\begin{document}
	
\mbox{}
\title{Bounded and finite factorization domains}

\author{David F. Anderson}
\address{Department of Mathematics\\University of Tennessee\\Knoxville, TN 37996}
\email{danders5@tennessee.edu}

\author{Felix Gotti}
\address{Department of Mathematics\\MIT\\Cambridge, MA 02139}
\email{fgotti@mit.edu}

\date{\today}
	
\begin{abstract}
	 An integral domain is atomic if every nonzero nonunit factors into irreducibles. Let $R$ be an integral domain. We say that $R$ is a bounded factorization domain if it is atomic and for every nonzero nonunit $x \in R$, there is a positive integer $N$ such that for any factorization $x = a_1 \cdots a_n$ of~$x$ into irreducibles $a_1, \dots, a_n$ in $R$, the inequality $n \le N$ holds. In addition, we say that $R$ is a finite factorization domain if it is atomic and every nonzero nonunit in $R$ factors into irreducibles in only finitely many ways (up to order and associates). The notions of bounded and finite factorization domains were introduced by D.~D. Anderson, D.~F. Anderson, and M. Zafrullah in their systematic study of factorization in atomic integral domains. In this chapter, we present some of the most relevant results on bounded and finite factorization domains.
\end{abstract}

\dedicatory{Dedicated to Daniel D. Anderson on his retirement}

\bigskip
\maketitle

\section{Introduction}
\label{sec:intro}

During the last three decades, the study of factorizations based on Diagram~\eqref{diag:AAZ's atomic chain} has earned significant attention among researchers in commutative algebra and semigroup theory. This diagram of classes of integral domains satisfying conditions weaker than unique factorization was introduced by D.~D. Anderson, D.~F. Anderson, and M. Zafrullah in~\cite{AAZ90}. We proceed to recall the definitions of the atomic classes in Diagram~\eqref{diag:AAZ's atomic chain}. Let~$R$ be an integral domain. Following P.~M. Cohn~\cite{pC68}, we say that~$R$ is \emph{atomic} if every nonzero nonunit of~$R$ can be factored into irreducibles. In addition,~$R$ satisfies the \emph{ascending chain condition on principal ideals} (or \emph{ACCP}) if every ascending chain of principal ideals of~$R$ eventually stabilizes. If an integral domain satisfies ACCP, then it is atomic; however, there are atomic domains that do not satisfy ACCP (the first example was constructed by A. Grams in~\cite{aG74}). On the other hand, $R$ is called a \emph{half-factorial domain} (or an \emph{HFD}) if $R$ is atomic and any two factorizations of the same nonzero nonunit of $R$ have the same number of irreducibles (counting repetitions). The term ``half-factorial domain" was coined by A. Zaks in~\cite{aZ76}. For a survey on half-factorial integral domains, see~\cite{CC00}.

\begin{equation} \label{diag:AAZ's atomic chain}
	\begin{tikzcd}[cramped]
		\textbf{ UFD } \ \arrow[r, Rightarrow]  \arrow[d, Rightarrow] & \ \textbf{ HFD } \arrow[d, Rightarrow] \\
		\textbf{ FFD } \ \arrow[r, Rightarrow] & \ \textbf{ BFD } \arrow[r, Rightarrow]  & \textbf{ ACCP domain}  \arrow[r, Rightarrow] & \textbf{ atomic domain}
	\end{tikzcd}
\end{equation}
\smallskip

\noindent We say that $R$ is a \emph{bounded factorization domain} (or a \emph{BFD}) if it is atomic and for every nonzero nonunit $x \in R$, there is a positive integer $N$ such that $x = a_1 \cdots a_n$ for irreducibles $a_1, \dots, a_n \in R$ implies that $n \le N$. In addition, we say that $R$ is a \emph{finite factorization domain} (or an \emph{FFD}) if it is atomic and every nonzero nonunit of $R$ factors into irreducibles in only finitely many ways (up to order and associates). The notions of a BFD and an FFD were introduced in~\cite{AAZ90} as part of Diagram~\eqref{diag:AAZ's atomic chain}. The purpose of this chapter is to survey some of the fundamental results related to bounded and finite factorization domains that have been established in the last three decades, indicating for the interested reader the sources where the most relevant results originally appeared. Although the rings we consider here have no nonzero zero-divisors, it is worth pointing out that the bounded and finite factorization properties have been extensively investigated in the context of commutative rings with zero-divisors by D.~D. Anderson and his students; see~\cite{AJ17} for more details and references.

This chapter is organized as follows. In Section~\ref{sec:prelim}, we recall some definitions and settle down the notation we will use throughout this paper. In Section~\ref{sec:BFMs and FFMs}, we give a few results about the bounded and finite factorization properties in the abstract context of monoids. Our treatment of monoids is brief as we only present results that will be useful later in the context of integral domains. Then in Section~\ref{sec:classes and examples of BFDs and FFDs}, we turn our attention to bounded and finite factorization domains, providing several characterizations and showing, among other results, that Noetherian domains and Krull domains are BFDs and FFDs, respectively. We also consider the popular $D+M$ construction. In Section~\ref{sec:extensions and localization}, we explore conditions under which the bounded and finite factorization properties are inherited by subrings or passed to ring extensions; we put particular emphasis on ring extensions by localization and pullback constructions. Directed unions are also considered. In Section~\ref{sec:polynomial-like rings}, we treat integral domains somehow related to rings of polynomials and rings of power series. We put special emphasis on the class of monoid domains. Finally, in Section~\ref{sec:generalized BFDs and FFDs}, we briefly explore an abstraction of the finite factorization property introduced by D. D. Anderson and the first author in~\cite{AA10}, where factorizations in an integral domain are identified up to a given arbitrary equivalence relation on the set of irreducibles (not necessarily that of being associates).

\bigskip
\section{Preliminary}
\label{sec:prelim}

In this section, we briefly review some notation and terminology we will use throughout this chapter. For undefined terms or a more comprehensive treatment of non-unique factorization theory, see~\cite{GH06} by A. Geroldinger and F. Halter-Koch.

\smallskip
\subsection{General Notation}

As is customary, $\zz$, $\zz/ \hspace{-1pt} n\zz$, $\qq$, $\rr$, and $\cc$ will denote the set integers, integers modulo $n$, rational numbers, real numbers, and complex numbers, respectively. We let $\nn$ and $\nn_0$ denote the set of positive and nonnegative integers, respectively. In addition, we let $\pp$ denote the set of primes. For $p \in \pp$ and $n \in \nn$, we let $\ff_{p^n}$ be the finite field of cardinality $p^n$. For $a,b \in \zz$ with $a \le b$, we let $\ldb a,b \rdb$ denote the set of integers between $a$ and $b$, i.e., $\ldb a,b \rdb = \{n \in \zz \mid a \le n \le b\}$. In addition, for $S \subseteq \rr$ and $r \in \rr$, we set $S_{\ge r} = \{s \in S \mid s \ge r\}$ and $S_{> r} = \{s \in S \mid s > r\}$.

\smallskip
\subsection{Factorizations}

Although a monoid is usually defined to be a semigroup with an identity element, here we will additionally assume that all monoids are cancellative and commutative. Let $M$ be a monoid. We say that $M$ is \emph{torsion-free} provided that for all $a,b \in M$, if $a^n = b^n$ for some $n \in \nn$, then $a=b$. The \emph{quotient group} $\gp(M)$ of a monoid $M$ is the set of quotients of elements in~$M$ (i.e., the unique abelian group $\gp(M)$ up to isomorphism satisfying that any abelian group containing a homomorphic image of $M$ will also contain a homomorphic image of $\gp(M)$). The group of invertible elements of $M$ is denoted by $U(M)$. The monoid $M$ is \emph{reduced} if $|U(M)| = 1$. An element $a \in M \! \setminus \! U(M)$ is an \emph{irreducible} (or an \emph{atom}) if whenever $a = uv$ for some $u,v \in M$, then either $u \in U(M)$ or $v \in U(M)$. The set of irreducibles of $M$ is denoted by $\ii(M)$. The monoid $M$ is \emph{atomic} if every non-invertible element factors into irreducibles. A subset $I$ of $M$ is an \emph{ideal} of~$M$ provided that $I \, M = I$ (or, equivalently, $I \, M \subseteq I$). The ideal $I$ is \emph{principal} if $I = bM$ for some $b \in M$. The monoid $M$ satisfies the \emph{ascending chain condition on principal ideals} (or \emph{ACCP}) if every ascending chain of principal ideals of $M$ eventually stabilizes.

It is clear that the monoid $M$ is atomic if and only if its quotient monoid $M_{\text{red}} = M/U(M)$ is atomic. Let $Z(M)$ denote the free (commutative) monoid on $\ii(M_{\text{red}})$, and let $\pi \colon Z(M) \to M_\text{red}$ be the unique monoid homomorphism fixing $a$ for every $a \in \ii(M_{\text{red}})$. If $z = a_1 \cdots a_\ell \in Z(M)$, where $a_1, \dots, a_\ell \in \ii(M_{\text{red}})$, then $\ell$ is the \emph{length} of $z$ and is denoted by $|z|$. For every $b \in M$, we set
\[
	Z(b) = Z_M(b) = \pi^{-1} (b U(M)) \quad \text{and} \quad L(b) = L_M(b) = \{ |z| \mid z \in Z(b) \}.
\]
If $M$ is atomic and $|Z(b)| < \infty$ for every $b \in M$, then we say that $M$ is a \emph{finite factorization monoid} (or an \emph{FFM}). On the other hand, if $M$ is atomic and $|L(b)| < \infty$ for every $b \in M$, then we say that~$M$ is a \emph{bounded factorization monoid} (or a \emph{BFM}). Clearly, every FFM is a BFM. The monoid~$M$ is a \emph{unique factorization monoid} (or a \emph{UFM}) if $Z(b)$ is a singleton for every $b \in M$, and $M$ is a \emph{half-factorial monoid} (or an \emph{HFM}) if $L(b)$ is a singleton for every $b \in M$. It is clear that every UFM is both an FFM and an HFM and that every HFM is a BFM.

Let $R$ be an integral domain. We let $R^\ast$ denote the multiplicative monoid of $R$, i.e., $R^\ast = R \setminus \{0\}$. We set $Z(R) = Z(R^\ast)$, and for every $x \in R^\ast$, we set $Z(x) = Z_{R^\ast}(x)$ and $L(x) = L_{R^\ast}(x)$. It is clear that $R$ is a BFD (resp., an FFD, an HFD, or a UFD) if and only if $R^\ast$ is a BFM (resp., an FFM, an HFM, or a UFM). As we did for monoids, we let $U(R)$ and $\ii(R)$ denote the group of units and the set of irreducibles of $R$, respectively. In addition, we let $\mathcal{P}(R)$ denote the set of primes of $R$. The quotient field of $R$ is denoted by $\text{qf}(R)$. An \emph{overring} of~$R$ is a subring of $\text{qf}(R)$ containing $R$. The abelian group $\text{qf}(R)^\ast/U(R)$, written additively, is the \emph{group of divisibility} of $R$ and is denoted by $G(R)$. The group $G(R)$ is partially ordered under the relation $x U(R) \le y U(R)$ if and only if $y \in xR$; we let $G(R)^+$ denote the monoid consisting of all the nonnegative elements of $G(R)$.

Even before we consider the bounded and finite factorization properties on monoid domains (in Subsection~\ref{subsec:monoid domains}), many of the examples that we construct here will involve such rings. For an integral domain $R$ and a monoid $M$, we let $R[X;M]$ denote the ring of polynomial expressions with coefficients in $R$ and exponents in $M$. Following R. Gilmer~\cite{rG84}, we will write $R[M]$ instead of $R[X;M]$. When~$M$ is torsion-free, $R[M]$ is an integral domain by \cite[Theorem~8.1]{rG84} and the group of units of $R[M]$ is $U(R[M]) = \{uX^m \mid u \in U(R) \ \text{and} \ m \in U(M)\}$ by \cite[Theorem~11.1]{rG84}. A detailed study of monoid rings is given by Gilmer in~\cite{rG84}.

\bigskip
\section{Bounded and Finite Factorization Monoids}
\label{sec:BFMs and FFMs}

In this section, we briefly present some basic results related to both the bounded and finite factorization properties in the abstract context of monoids. Diagram~\eqref{diag:AAZ's atomic chain} also holds for the more general class consisting of monoids (see Diagram~\eqref{diag:AAZ's atomic chain for monoids} below).

\begin{equation} \label{diag:AAZ's atomic chain for monoids}
	\begin{tikzcd}[cramped]
		\textbf{ UFM } \ \arrow[r, Rightarrow]  \arrow[d, Rightarrow] & \ \textbf{ HFM } \arrow[d, Rightarrow] \\
		\textbf{ FFM } \ \arrow[r, Rightarrow] & \ \textbf{ BFM } \arrow[r, Rightarrow]  & \textbf{ ACCP monoid}  \arrow[r, Rightarrow] & \textbf{ atomic monoid}
	\end{tikzcd}
\end{equation}
\smallskip

The last two implications in Diagram~\eqref{diag:AAZ's atomic chain for monoids} are the only ones that are not immediate from definitions. We argue these two implications in this section (Corollary~\ref{cor:BFMs are ACCP monoids} and Remark~\ref{rem:ACCP monoids are atomic}) and obtain, as a result, Diagram~\eqref{diag:AAZ's atomic chain for monoids}. As this survey focuses on integral domains, we will give a result in the context of monoids only if it is needed in Sections~\ref{sec:classes and examples of BFDs and FFDs}--\ref{sec:generalized BFDs and FFDs}.

\medskip
\subsection{The Bounded Factorization Property}

To begin with, we characterize BFMs in terms of the existence of certain ``length functions". Let $M$ be a monoid. A function $\ell \colon M \to \nn_0$ is called a \emph{length function} of $M$ if it satisfies the following two properties:
\begin{enumerate}
	\item[(i)] $\ell(u) = 0$ if and only if $u \in U(M)$;
	\smallskip
	
	\item[(ii)] $\ell(bc) \ge \ell(b) + \ell(c)$ for every $b,c \in M$.
\end{enumerate} 

The following characterization of a BFM will prove useful at several later points.

\begin{prop} \emph(\cite[Theorem~1]{fHK92}\emph) \label{prop:BFM characterization via length functions}
	A monoid $M$ is a BFM if and only if there is a length function $\ell \colon M \to \nn_0$.
\end{prop}

\begin{proof}
	Suppose first that $M$ is a BFM. Then define a function $\ell \colon M \to \nn_0$ by $\ell(b) = \max L(b)$. Condition~(i) in the definition of a length function follows immediately. In addition, it is clear that $\max L(bc) \ge \max L(b) + \max L(c)$ for every $b,c \in M$, from which we obtain condition~(ii). Conversely, suppose that $\ell \colon M \to \nn_0$ is a length function. Take $b \in M \setminus U(M)$ such that $b = a_1 \cdots a_m$ for some $a_1, \dots, a_m \in M \setminus U(M)$, and set  $b_j = a_1 \cdots a_j$ for every $j \in \ldb 1,m \rdb$. As $\ell(b_m) > \ell(b_{m-1}) > \cdots > \ell(b_1)$, the inequality $m \le \ell(b)$ holds. Now observe that if we take $m$ as large as it can possibly be, then the maximality of $m$ guarantees that $a_1, \dots, a_m \in \ii(M)$. Hence $M$ is atomic. Since $\sup L(b) \le \ell(b)$ for every $b \in M \setminus U(M)$, we conclude that $M$ is a BFM.
\end{proof}

As we mentioned in the introduction, every BFD satisfies ACCP. This actually holds in the more general context of monoids, as the next corollary indicates.

\begin{cor} \emph(\cite[Corollary~1]{fHK92}\emph) \label{cor:BFMs are ACCP monoids}
	If $M$ is a BFM, then $M$ satisfies ACCP.
\end{cor}

\begin{proof}
	By Proposition~\ref{prop:BFM characterization via length functions}, there is a length function $\ell \colon M \to \nn_0$. Suppose that $(b_n M)_{n \in \nn}$ is an ascending chain of principal ideals of $M$. For every $n \in \nn$, the inclusion $b_n \in b_{n+1}M$ ensures that $\ell(b_n) \ge \ell(b_{n+1})$. Hence there is an $n_0 \in \nn$ such that $\ell(b_n) = \ell(b_{n+1})$ for every $n \ge n_0$. This implies that $b_n \in b_{n+1}U(M)$ for every $n \ge n_0$, and so $(b_nM)_{n \in \nn}$ must stabilize. Thus, $M$ satisfies ACCP.
\end{proof}

The reverse implication of Corollary~\ref{cor:BFMs are ACCP monoids} does not hold in general. The following example, which is a fragment of \cite[Example~2.1]{AAZ90}, corroborates our observation.

\begin{example}\footnote{The factorization structure of additive submonoids of $\qq_{\ge 0}$ (known as Puiseux monoids) has been systematically studied in the last few years (see \cite{CGG20a} for a recent survey).}  \label{ex:ACCP monoid that is not BFM}
	 Let $M$ be the additive submonoid of $\qq_{\ge 0}$ generated by the set $\{ 1/p \mid p \in \pp \}$. It can be readily checked that $\mathcal{A}(M) = \{1/p \mid p \in \pp \}$. In addition, it is not hard to verify that for every $q \in M$, there is a unique $N(q) \in \nn_0$ and a unique sequence of nonnegative integers $(c_p(q))_{p \in \pp}$ such that $q = N(q) + \sum_{p \in \pp} c_p(q) \frac{1}{p}$. Set $S(q) = \sum_{p \in \pp} c_p(q)$. It is clear that if $q \in q' + M$ for some $q' \in M$, then $N(q') \le N(q)$. Also, if $q'$ is a proper divisor of $q$ in~$M$, then $N(q') = N(q)$ ensures that $S(q') < S(q)$. Thus, every sequence $(q_n)_{n \in \nn}$ in $M$ satisfying that $q_n \in q_{n+1} + M$ for every $n \in \nn$ must stabilize, and so $M$ must satisfy ACCP. Finally, we can see that $M$ is not a BFM because $\pp \subseteq L(1)$.
\end{example}

The bounded factorization property is inherited by those submonoids that preserve invertible elements.

\begin{prop} \emph(\cite[Theorem~3]{fHK92}\emph) \label{prop:BF inherited by inert submonoids}
	Let $M$ be a BFM. Then every submonoid $N$ of $M$ satisfying $U(N) = U(M) \cap N$ is also a BFM.
\end{prop}

\begin{proof}
	Let $N$ be a submonoid of $M$ such that $U(N) = U(M) \cap N$. Since $M$ is a BFM, there is a length function $\ell \colon M \to \nn_0$ by Proposition~\ref{prop:BFM characterization via length functions}. As $U(N) = U(M) \cap N$, the equality $\ell(u) = 0$ holds for $u \in N$ if and only if $u \in U(N)$. This, along with the fact that $\ell(bc) \ge \ell(b) + \ell(c)$ for every $b,c \in N$, guarantees that $\ell$ is still a length function when restricted to~$N$. Hence $N$ is a BFM.
\end{proof}

The reduced monoids in the following example will be useful later to construct monoid domains that are BFDs with further desired properties.

\begin{example} \label{ex:BF positive monoids}
	Let $M$ be an additive submonoid of $\qq_{\ge 0}$ such that $0$ is not a limit point of $M \setminus \{0\}$. Then it follows from \cite[Proposition~4.5]{fG19} that $M$ is a BFM.
\end{example}

\smallskip
\subsection{The Finite Factorization Property}

We now turn to give two characterizations of an FFM. To do so, we use Dickson's Lemma, a standard result in combinatorics stating that for every $k \in \nn$, a subset of $\nn_0^k$ contains only finitely many minimal elements under the usual product ordering.

\begin{prop} \emph(\cite[Theorem~2 and Corollary~2]{fHK92}\emph) \label{prop:FFM characterization via idf-monoids}
	Let $M$ be a monoid. Then the following statements are equivalent.
	\begin{enumerate}
		\item[(a)] $M$ is an FFM.
		\smallskip
		
		\item[(b)] Every element of $M$ has only finitely many non-associate divisors.
		\smallskip
		
		\item[(c)] $M$ is atomic and every element of $M$ is divisible by only finitely many non-associate irreducibles.
	\end{enumerate}
\end{prop}

\begin{proof}
	We assume, without loss of generality, that $M$ is reduced.
	\smallskip
	
	(a) $\Rightarrow$ (b):  Suppose that $M$ is an FFM, and fix $b \in M$. If $d$ is a divisor of $b$ in $M$, then every factorization of $d$ is a subfactorization of some factorization of $b$. This, together with the fact that $Z(b)$ is finite, implies that $b$ has only finitely many divisors in $M$.
	\smallskip
	
	(b) $\Rightarrow$ (c): Assume that every element of $M$ has only finitely many divisors. Note that $M$ must satisfy ACCP by Corollary~\ref{cor:BFMs are ACCP monoids}. Suppose, by way of contradiction, that $M$ is not atomic. Then the set $S$ consisting of all the elements of $M$ that do not factor into irreducibles is nonempty. Since $M$ satisfies ACCP and $S$ is nonempty, there is a $b \in S$ such that the ideal $bM$ is maximal among all principal ideals of $M$ generated by elements of $S$. Because $b \in S$, there are $b_1, b_2 \in M \setminus U(M)$ with $b = b_1 b_2$ such that $b_1 \in S$ or $b_2 \in S$. So $b M$ is strictly contained in either $b_1 M$ or $b_2 M$, which contradicts the maximality of $b M$. Thus, $M$ is atomic. The second part of the statement is clear.
	\smallskip
	
	(c) $\Rightarrow$ (a): Suppose that $M$ is atomic and every element of $M$ is divisible by only finitely many irreducibles. Take $b \in M$, and let $A_b$ be the set of irreducibles in $M$ dividing $b$. Because $Z(b)$ is a subset of the finite-rank free commutative monoid $F$ on $A_b$, it follows from Dickson's Lemma that $Z(b)$ has only finitely many minimal elements with respect to the order induced by division in $F$. This, along with the fact that any two factorizations in $Z(b)$ are incomparable as elements of $F$, implies that $|Z(b)| < \infty$. Thus, $M$ is an FFM.
\end{proof}

As a consequence of Proposition~\ref{prop:FFM characterization via idf-monoids}, finitely generated monoids are FFMs.

\begin{cor} \label{cor:finitely generated monoids are FFMs}
	Every finitely generated monoid is an FFM.
\end{cor}

\begin{proof}
	It suffices to prove the corollary for reduced monoids. Let $M$ be a finitely generated reduced monoid that is minimally generated by $a_1, \dots, a_m$. It readily follows that $\ii(M) = \{a_1, \dots, a_m\}$, and therefore, $M$ is atomic. Thus, $M$ is an FFM by Proposition~\ref{prop:FFM characterization via idf-monoids}.
\end{proof}

In the proof of Proposition~\ref{prop:FFM characterization via idf-monoids}, we have incidentally argued the following remark.

\begin{remark}  \label{rem:ACCP monoids are atomic}
	Every monoid satisfying ACCP is atomic.
\end{remark}

In contrast to what we have already seen for BFMs, an FFM $M$ can have a submonoid $N$ satisfying $U(N) = U(M) \cap N$ that is not an FFM.

\begin{example} \label{ex:an FFM monoid with a submonoid that is not an FFM}
	Let $M$ be the additive monoid $\zz \times \nn_0$. Then it is easy to verify that $M$ is atomic with $U(M) = \zz \times \{0\}$ and $\ii(M) = \zz \times \{1\}$. Since $M_{\text{red}} \cong \nn_0$, the monoid $M$ is a UFM and, in particular, an FFM. Now consider the submonoid $N = \{(0,0)\} \cup (\zz \times \nn)$ of $M$. Note that $N$ is reduced with $\ii(N) = \zz \times \{1\}$. As a result, $U(N) = U(M) \cap N$. However, $N$ is not an FFM as $(0,2) = (-n,1) + (n,1)$ for every $n \in \nn$.
\end{example}

We record the following proposition, whose proof is straightforward.

\begin{prop}\emph(\cite[Corollary~3]{fHK92}\emph)
	Every submonoid of a reduced FFM is an FFM.
\end{prop}

To conclude this section, we give some examples of FFMs that will be used later to construct monoid domains that are FFDs and have further algebraic properties.

\begin{example} \label{ex:FF positive monoids}
	Let $(q_n)_{n \in \nn}$ be an increasing sequence of positive rational numbers, and consider the additive submonoid $M = \langle q_n \mid n \in \nn \rangle$ of $\qq_{\ge 0}$. It is not hard to argue that $M$ is an FFM; indeed, it follows from \cite[Theorem~5.6]{fG19} that any additive submonoid of the nonnegative cone of an ordered field $F$ is an FFM provided that such a monoid can be generated by an increasing sequence of $F$.
\end{example}

\bigskip
\section{Bounded and Finite Factorization Domains}
\label{sec:classes and examples of BFDs and FFDs}

In this section, we provide characterizations and give various examples and classes of BFDs and FFDs.

\smallskip
\subsection{Characterizations of BFDs and (Strong) FFDs}

There are several other useful ways to rephrase what it means for an integral domain to be a BFD. The following proposition illustrates this observation.

\begin{prop} \emph(\cite[Theorem~2.4]{AAZ90}\emph) \label{prop:BFD characterizations}
	The following statements are equivalent for an integral domain~$R$.
	\begin{enumerate}
		\item[(a)] $R$ is a BFD.
		\smallskip
		
		\item[(b)] There is a length function $\ell \colon R^\ast \to \nn_0$.
		\smallskip
		
		\item[(c)] For every $x \in R^\ast$, there is a positive integer $n$ such that every (strictly) ascending chain of principal ideals starting at $xR$ has length at most $n$.
		\smallskip
		
		\item[(d)] For every $x \in G(R)^+$, there is a positive integer $n$ such that $x$ is the sum of at most $n$ (minimal) positive elements in $G(R)^+$.
	\end{enumerate}
\end{prop}

\begin{proof}
	(a) $\Leftrightarrow$ (b): This is a direct consequence of Proposition~\ref{prop:BFM characterization via length functions}.
	\smallskip
	
	(a) $\Leftrightarrow$ (d): It is clear that $G(R)^+ = \{xU(R) \mid x \in R^*\} = R^*_{\text{red}}$. As a result, for every $x \in R^* \setminus U(R)$, the set $L(x)$ has an upper bound $n \in \nn$ if and only if $x$ is the sum of at most $n$ positive elements in $G(R)^+$.
	\smallskip
	
	(b) $\Rightarrow$ (c): Let $\ell \colon R^* \to \nn_0$ be a length function. Take an $x \in R^*$ and set $n = \ell(x)$. Let $x_0R \subsetneq x_1R \subsetneq \cdots \subsetneq x_kR$ be a strictly ascending chain of principal ideals of $R$ such that $x_0 = x$. It is clear that $x_0, \dots, x_k$ are pairwise non-associates in $R$, and also that $x_{i-1} \in x_iR^*$ for every $i \in \ldb 1,k \rdb$. Therefore $n = \ell(x_0) > \ell(x_1) > \cdots > \ell(x_k)$. This implies that the length of $x_0R \subsetneq x_1R \subsetneq \cdots \subsetneq x_kR$ is at most~$n$.
	\smallskip
	
	(c) $\Rightarrow$ (b): Define the function $\ell \colon R^* \to \nn_0$ by taking $\ell(x)$ to be the smallest $n \in \nn_0$ such that every ascending chain of principal ideals of $R$ starting at $xR$ has length at most $n$. If $x \in U(R)$, then $xR = R$ and so $\ell(x) = 0$. In addition, if for $x_0,y_0 \in R$, we take two ascending chains of principal ideals $x_0R \subsetneq x_1R \subsetneq \cdots \subsetneq x_jR$ and $y_0R \subsetneq y_1R \subsetneq \cdots \subsetneq y_kR$, then the ascending chain of principal ideals $x_0 y_0R \subsetneq x_1 y_0R \subsetneq \cdots \subsetneq x_j y_0R \subseteq y_0R \subsetneq y_1R \subsetneq \cdots \subsetneq y_kR$ starts at $x_0y_0R$ and has length at least $j+k$. Hence $\ell(xy) \ge \ell(x) + \ell(y)$ for all $x,y \in R^*$. Thus, $\ell$ is a length function.
\end{proof}

The elasticity\footnote{Although R. Valenza coined the term elasticity and introduced it in the context of algebraic rings of integers, it is worth noting that J. L. Steffan \cite{jlS86} also studied elasticity about the same time in the more general context of Dedekind domains.}, introduced by R. Valenza~\cite{rV90} in the context of algebraic rings of integers, is an arithmetic invariant that allows us to measure how far an atomic integral domain is from being an HFD. Given an atomic integral domain $R$, its \emph{elasticity} is defined as follows:
\[
	\rho(R) = \sup \bigg\{ \frac{\sup L(x)}{\min L(x)} \ \bigg{|} \ x \in R^* \setminus U(R) \bigg\}
\]
when $R$ is not a field, and $\rho(R) = 1$ when $R$ is a field. Clearly, $1 \le \rho(R) \le \infty$ and $\rho(R) = 1$ if and only if $R$ is an HFD. For a survey on the elasticity of integral domains, see~\cite{dfA97}. Following \cite{AA92}, we say that~$R$ is a \emph{rational bounded factorization domain} (or an \emph{RBFD}) if $R$ is atomic and $\rho(R) < \infty$. Observe that HFD $\Rightarrow$ RBFD $\Rightarrow$ BFD. Moreover, none of these implications is reversible and being an FFD does not imply being an RBFD. The following example sheds some light upon these observations.

\begin{example}
		For every $r \in \rr_{\ge 1} \bigcup \{\infty\}$, \cite[Theorem~3.2]{AA92} guarantees the existence of a Dedekind domain (with torsion divisor class group) whose elasticity is $r$. Let $D_1$ be a Dedekind domain such that $\rho(D_1) = 3/2$. Since $\rho(D_1) > 1$, the domain $D_1$ is not an HFD. Thus, not every RBFD is an HFD. On the other hand, let $D_2$ be a Dedekind domain such that $\rho(D_2) = \infty$. As we will see in Corollary~\ref{cor:Dedekind domains are FFD}, every Dedekind domain is an FFD. As a result, $D_2$ is an FFD that is not an RBFD. Therefore not every BFD is an RBFD.
\end{example}

Following A. Grams and H. Warner~\cite{GW75}, we say that an integral domain~$R$ is an \emph{idf-domain} if every nonzero element of $R$ has at most finitely many non-associate irreducible divisors. We next give several useful characterizations of an FFD.

\begin{prop} \emph(\cite[Theorem~5.1]{AAZ90} and \cite[Theorem~1]{AM96}\emph) \label{prop:FFD characterizations}
	The following statements are equivalent for an integral domain $R$.
	\begin{enumerate}
		\item[(a)] $R$ is an FFD.
		\smallskip
		
		\item[(b)] $R$ is an atomic idf-domain.
		\smallskip
		
		\item[(c)] Every element of $R^*$ has only finitely many non-associate divisors.
		\smallskip
		
		\item[(d)] Every nonzero principal ideal of $R$ is contained in only finitely many principal ideals.
		\smallskip
		
		\item[(e)] For any infinite family $\{x_i R \mid i \in I\}$ of principal ideals, $\bigcap_{i \in I} x_i R = \{0\}$.
		\smallskip
		
		\item[(f)] For every $xU(R) \in G(R)^+$, the interval $[0,xU(R)]$ of the ordered monoid $G(R)^+$ is finite.
	\end{enumerate}
\end{prop}

\begin{proof}
	(a) $\Leftrightarrow$ (b) $\Leftrightarrow$ (c): It follows directly from Proposition~\ref{prop:FFM characterization via idf-monoids}.
	\smallskip
	
	(c) $\Leftrightarrow$ (d): It is clear as, for all $x,y \in R^*$, it follows that $x$ divides $y$ in $R^*$ if and only if $yR \subseteq xR$.
	\smallskip
	
	(d) $\Leftrightarrow$ (e): This is straightforward.
	\smallskip
	
	(c) $\Leftrightarrow$ (f): Since $G(R)^+ = R^*_{\text{red}}$, we only need to observe that, for every $y \in R^*$, the inclusion $yU(R) \in [0,xU(R)]$ holds if and only if $y$ divides $x$ in $R^*$.
\end{proof}

\begin{remark}
	A graph-theoretic characterization of an FFD has recently been provided by J. D. LaGrange in \cite[Theorem~13]{jL19}.
\end{remark}

Following D. D. Anderson and B. Mullins~\cite{AM96}, we say that an integral domain $R$ is a \emph{strong finite factorization domain} (or an \emph{SFFD}) if every nonzero element of $R$ has only finitely many divisors, and we say that $R$ is a \emph{strong idf-domain} if every nonzero element of $R$ has only finitely many divisors which are either units or irreducibles. We can characterize SFFDs as follows.

\begin{prop}  \emph(\cite[Theorem~5]{AM96}\emph) \label{prop:SFFD characterizations}
	The following statements are equivalent for an integral domain~$R$.
	\begin{enumerate}
		\item[(a)] $R$ is an SFFD.
		\smallskip
		
		\item[(b)] $R$ is an atomic strong idf-domain.
		\smallskip
		
		\item[(c)] $R$ is an FFD and $U(R)$ is finite.
	\end{enumerate}
\end{prop}

\begin{proof}
	(a) $\Rightarrow$ (b): Consider the map $\ell \colon R^* \to \nn_0$ defined by letting $\ell(x)$ be the number of nonunit divisors of $x$ in $R$. Clearly, $\ell(u) = 0$ for every $u \in U(R)$. If $x,y \in R^*$, then every nonunit divisor of~$x$ divides $xy$, and for every nonunit divisor $d$ of $y$, we see that $xd$ divides $xy$ but does not divide $x$; whence $\ell(xy) \ge \ell(x) + \ell(y)$. As a result, $\ell$ is a length function. Since $R$ is a BFD by Proposition~\ref{prop:BFD characterizations}, it must be atomic. In addition, it is clear that $R$ is a strong idf-domain.
	\smallskip
	
	(b) $\Rightarrow$ (c): That $R$ is an FFD follows from Proposition~\ref{prop:FFD characterizations}. In addition, $U(R)$ is the set of divisors of $1$, and therefore, it must be finite.
	\smallskip
	
	(c) $\Rightarrow$ (a): Since $R$ is an FFD, every element of $R^*$ has only finitely many non-associate divisors. In addition, every element of $R^*$ has only finitely many associates because $U(R)$ is finite. Hence every element of $R^*$ must admit only finitely many divisors. Thus, $R$ is an SFFD.
\end{proof}

Not every FFD is an SFFD; indeed there are integral domains with all its subrings being FFDs that are not SFFD. The following example was given in \cite[Remark~3]{AM96}.

\begin{example}
	For $p \in \pp$ and $m \in \nn$, let $\ff_{p^m}$ be the finite field of cardinality $p^m$. Since for every $n \in \nn$, the field $\ff_{p^{2^n}}$ contains a copy of $\ff_{p^{2^{n-1}}}$ as a subfield, we can consider the field $\ff = \bigcup_{n \in \nn_0} \ff_{p^{2^n}}$. Although $\ff$ is an infinite field, every proper subring of $\ff$ is a finite field and so an SFFD. However,~$\ff$ is not an SFFD because $|U(\ff)| = |\ff^*| = \infty$. Lastly, observe that every subring of~$\ff$ is a field, and therefore, an FFD.
\end{example}

\smallskip
\subsection{Some Relevant Classes of BFDs and FFDs}

In this subsection, we identify some relevant classes of BFDs and FFDs.
\smallskip

It is clear that every HFD is a BFD. Observe, however, that a BFD need not be an HFD; for instance, the BFD $\qq[X^2, X^3]$ is not an HFD because $(X^2)^3 = (X^3)^2$. Similarly, although every FFD is a BFD, there are BFDs that are not FFDs; indeed, the integral domain $\rr + X\cc[X]$ is a BFD (by Theorem~\ref{thm:Noetherian domains are BFDs}) that is not an FFD (see Example~\ref{ex:a Noetherian domain that is not an FFM}). As we illustrate in the next example, for every $q \in \qq_{> 0}$, the monoid domain $\qq[M_q]$, where $M_q = \{0\} \cup \qq_{\ge q}$, is a BFD that is neither an HFM nor an FFM. The monoid domain $\qq[M_1]$ seems to be used first by Gilmer in \cite[page 189]{rG84} as an example of an integral domain satisfying ACCP with a localization not satisfying ACCP. The same monoid domain was used in \cite[Example~2.7(a)]{AAZ90} as an example of a one-dimensional BFD with a localization that is not a BFD (cf. Example~\ref{ex:FFD with a non-atomic localization}). The fact that $\qq[M_1]$ is a BFD that is not an FFD was implicit in \cite[Example~4.1(b)]{AAZ90} and later observed in \cite[Example~3.26]{hK98}.

\begin{example} \label{ex:BFD that is neither an HFD nor an FFD}
	For $q \in \qq_{> 0}$, let $M_q$ denote the additive monoid $\{0\} \cup \qq_{\ge q}$. Note that $M_q$ is one of the monoids in Example~\ref{ex:BF positive monoids}, and so it is a BFM. By a simple degree consideration, one can verify that the monoid domain $\qq[M_q]$ is a BFD (cf. \cite[Theorem~4.3(2)]{fG20a}). It is clear that $\ii(M_q) = [q,2q) \cap \qq$. Then for every $n \in \nn$ with $n > 2/q$, both $q +\frac q2 + \frac 1n$ and $q + \frac q2 - \frac 1n$ are irreducibles in $M_q$ and $3q = \big(q +\frac q2 + \frac 1n \big) + \big( q + \frac q2 - \frac 1n \big)$. Since $|Z(3q)| = \infty$, the monoid $M_q$ is not an FFM. Therefore part~(1) of Proposition~\ref{prop:FFD monoid domains} guarantees that $\qq[M_q]$ is not an FFD. Finally, we check that $\qq[M_q]$ is not an HFD. To do this, take $q_1, q_2 \in \ii(M_q)$ such that $q_1 \neq q_2$,  and then write $q_1 = a_1/b_1$ and $q_2 = a_2/b_2$ for some $a_1, a_2, b_1, b_2 \in \nn$ such that $\gcd(a_1, b_1) = \gcd(a_2, b_2) = 1$. Observe that $X^{a_1 a_2} = (X^{q_1})^{a_2 b_1} = (X^{q_2})^{a_1 b_2}$. Since $a_2 b_1 \neq a_1 b_2$, we see that $(X^{q_1})^{a_2 b_1}$ and $(X^{q_2})^{a_1 b_2}$ are factorizations of $X^{a_1 a_2}$ with different lengths. Thus, $\qq[M_q]$ is not an HFD (for a more general result in this direction, see~\cite[Theorem~4.4]{fG20}).
\end{example} 

As a consequence of Corollary~\ref{cor:BFMs are ACCP monoids}, every BFD satisfies ACCP. The reverse implication of this observation does not hold in general, as we proceed to illustrate with an example of a monoid domain first given in~\cite[Example~2.1]{AAZ90}.

\begin{example} \label{ex:ACCP domain that is not a BFD}
	We have seen in Example~\ref{ex:ACCP monoid that is not BFM} that the additive monoid $M = \langle 1/p \mid p \in \pp \rangle$ satisfies ACCP but is not a BFM. In addition, we have seen that $\ii(M) = \{1/p \mid p \in \pp\}$. Now consider the monoid domain $\qq[M]$. From the fact that $M$ satisfies ACCP, we can easily argue that $\qq[M]$ also satisfies ACCP. However, $\qq[M]$ is not a BFD; indeed, for every $p \in \pp$, there is a length-$p$ factorization of~$X$, namely, $X = (X^{1/p})^p$.
\end{example}

Noetherian domains are among the most important examples of BFDs.

\begin{theorem}  \emph(\cite[Proposition~2.2]{AAZ90}\emph) \label{thm:Noetherian domains are BFDs}
	Every Noetherian domain is a BFD.
\end{theorem}

\begin{proof}
	Let $R$ be a Noetherian domain, and take $x \in R^\ast \setminus U(R)$. We know that there are only finitely many height-one prime ideals over $xR$ in $R$, namely, $P_1, \dots, P_n$. By the Krull Intersection Theorem, for every $i \in \ldb 1,n \rdb$, there is a $k_i \in \nn$ such that $x \notin P_i^{k_i}$. Set $k = \max \{k_i \mid i \in \ldb 1,n \rdb \}$. We claim that $\max L_R(x) < kn$. Suppose, by way of contradiction, that $x = x_1 \cdots x_m$ for some $m \ge kn$ and $x_1, \dots, x_m \in R \setminus U(R)$. Since $xR$ contains a power of $P_1 \cdots P_n$, for every $j \in \ldb 1,m \rdb$ the inclusion $xR \subseteq x_jR$ ensures that $x_j \in P$ for some $P \in \{P_1, \dots, P_n\}$. Therefore there is an $i \in \ldb 1,n \rdb$ such that $x \in P_i^k$. However, this contradicts that $x \notin P_i^{k_i}$. Thus, the set of lengths of every nonzero nonunit of~$R$ is bounded, and so $R$ is a BFD.
\end{proof}
\smallskip

We will see that integrally closed Noetherian domains are FFDs in Corollary~\ref{cor:Dedekind domains are FFD}, and we will characterize Noetherian FFDs in Proposition~\ref{prop:characterization of Noetherian FFDs}. For now, it is worth noting that not every Noetherian domain is an FFD.

\begin{example}(cf. Propositions \ref{prop:FFD and D+M construction} and~\ref{prop:FFD for power-series-like extensions}) \label{ex:a Noetherian domain that is not an FFM}
	Consider the integral domain $R = \rr + X\cc[X]$. It is not hard to verify that $R$ is a Noetherian domain, although it is a direct consequence of \cite[Theorem~4]{BR76}. For every $p \in \pp$, let $\zeta_p$ be a primitive $p$-root of unity. Since $U(R) = \rr^\ast$, it is clear that distinct primes yield non-associate primitive roots of unity. Then $\{(\zeta_p X)(\zeta_p^{-1}X) \mid p \in \pp\}$ is a set consisting of infinitely many factorizations of $X^2$ in $R$. Hence $R$ is not an FFD. As a final note, observe that $R$ is an HFD by \cite[Theorem~5.3]{AAZ91}.
\end{example}

For a nonempty set $I$, let $\{R_i \mid i \in I\}$ be a family of subrings of the same integral domain. The integral domain $R = \bigcap_{i \in I} R_i$ is said to be the \emph{locally finite} intersection of the $R_i$'s if for every $x \in R^*$, the set $\{i \in I \mid x \notin U(R_i)\}$ is finite. As the next proposition illustrates, we can produce BFDs by taking locally finite intersections of BFDs.

\begin{prop} \emph(\cite[page 17]{AAZ90}\emph) \label{prop:a locally finite intersection of BFDs is a BFD}
	For a nonempty set $I$, let $\{R_i \mid i \in I\}$ be a family of subrings of an integral domain. If $R_i$ is a BFD for every $i \in I$, then the locally finite intersection $\bigcap_{i \in I} R_i$ is a BFD.
\end{prop}

\begin{proof}
	Set $R = \bigcap_{i \in I} R_i$. By Proposition~\ref{prop:BFD characterizations}, for every $i \in I$, there is a length function $\ell_i \colon R^*_i \to \nn_0$. Since $R$ is a locally finite intersection, the function $\ell = \sum_{i \in I} \ell_i \colon R^* \to \nn_0$ is well defined. From the definition of $\ell$, it immediately follows that $\ell$ is a length function. As a result, Proposition~\ref{prop:BFD characterizations} guarantees that $R$ is a BFD.
\end{proof}

We proceed to identify some relevant classes of FFDs. It is clear that every UFD is an FFD, but it is not hard to verify that $\qq[X^2, X^3]$ (resp., $(\zz/ 2\zz)[X^2, X^3]$) is an FFD (resp., an SFFD) that is not even an HFD (in Example~\ref{ex:a Noetherian domain that is not an FFM}, we have seen an HFD that is not an FFD). A \emph{Cohen-Kaplansky domain} (or a \emph{CKD}) is an atomic domain with finitely many non-associate irreducibles. These integral domains were first investigated by I.~S. Cohen and I. Kaplansky in~\cite{CK46} and then by D.~D. Anderson and J.~L. Mott in~\cite{AM92}. It follows from Proposition~\ref{prop:FFD characterizations} that every CKD is an FFD.

By Theorem~\ref{thm:Noetherian domains are BFDs}, every Noetherian domain is a BFD. It turns out that every one-dimensional Noetherian domain is an FFD provided that each of its residue fields is finite.

\begin{prop} \emph(\cite[Example~1]{AM96}\emph) \label{prop:FNP domains are FFDs}
	Every one-dimensional Noetherian domain whose residue fields are finite is an FFD.
\end{prop}

\begin{proof}
	Let $R$ be a one-dimensional Noetherian domain whose residue fields are finite. It follows from~\cite[Theorem~2.7]{LM72} that $R/I$ is finite for every nonzero proper ideal $I$ of $R$. Fix $x \in R^* \setminus U(R)$. Clearly, two distinct principal ideals $yR$ and $y'R$ of $R$ containing the ideal $xR$ yield distinct subgroups $yR/xR$ and $y'R/xR$ of the additive group $R/xR$. Since $|R/xR| < \infty$, the principal ideal $xR$ can only be contained in finitely many principal ideals of $R$. Hence it follows from Proposition~\ref{prop:FFD characterizations} that $R$ is an FFD.
\end{proof}

Throughout this survey, an integral domain is said to be \emph{quasilocal} if it has exactly one maximal ideal, while it is said to be \emph{local} if it is Noetherian and quasilocal. The following corollary, which is a direct consequence of Proposition~\ref{prop:FNP domains are FFDs}, was first observed in~\cite{AM96}.

\begin{cor} \label{cor:local FNP domains are FFDs}
	Every one-dimensional local domain with finite residue field is an FFD.
\end{cor}

Let $R$ be a one-dimensional local domain with maximal ideal $M$. Since by Corollary~\ref{cor:local FNP domains are FFDs} we know that $R$ is an FFD provided that $R/M$ is finite, we may wonder what happens in the case where $R/M$ is infinite. Under the assumption that $R/M$ is infinite, it follows that $R$ is an FFD if and only if $R$ is integrally closed; this is \cite[Corollary~4]{AM96}.

As for BFDs, we can produce new FFDs by considering locally finite intersections of FFDs.

\begin{prop}  \emph(\cite[Theorem~2]{AM96}\emph) \label{prop:a locally finite intersection of FFDs is an FFD}
	For a nonempty set $I$, let $\{R_i \mid i \in I\}$ be a family of subrings of an integral domain. If $R_i$ is an FFD for every $i \in I$, then the locally finite intersection $\bigcap_{i \in I} R_i$ is an FFD.
\end{prop}

\begin{proof}
	Set $R = \bigcap_{i \in I} R_i$. Take a nonunit $x \in R^*$, and set $J = \{i \in I \mid x \notin U(R_i)\}$. It follows from Proposition~\ref{prop:FFD characterizations} that, for every $j \in J$, the ideal $xR_j$ is contained in only finitely many principal ideals of $R_j$. We claim that the ideal $xR$ is contained in only finitely many principal ideals of $R$. To verify this, take $y \in R$ such that $xR \subseteq yR$. It is clear that $xR_i \subseteq yR_i$ for every $i \in I$. As a result, $y \notin U(R_j)$ implies that $j \in J$.  Since $J$ is finite, $xR$ is contained in only finitely many principal ideals of $R$. Thus, using Proposition~\ref{prop:FFD characterizations} once again, we conclude that $R$ is an FFD.
\end{proof}

An important source of FFDs is the class of Krull domains. An integral domain $R$ has \emph{finite character} if there is a family $\{V_i \mid i \in I\}$ of valuation overrings of $R$ indexed by a nonempty set $I$ such that $R = \bigcap_{i \in I} V_i$ is a locally finite intersection.

\begin{theorem} \emph(\cite[Proposition~2.2]{AAZ90} and \cite[Proposition~1]{GW75}\emph) \label{thm:Krull domains are FFDs}
	Every Krull domain is an FFD, and thus also a BFD.
\end{theorem}

\begin{proof}
	Let $R$ be a Krull domain. Mimicking the proof of Theorem~\ref{thm:Noetherian domains are BFDs}, we can show that $R$ is a BFD, and therefore, an atomic domain. Now let $X$ be the set of all height-one prime ideals of~$R$. Then~$R$ has finite character with respect to the family of DVRs $\{R_P \mid P \in X\}$. As a result, it follows from~\cite[Proposition~1]{GW75} that $R$ is an idf-domain. Since $R$ is an atomic idf-domain, Proposition~\ref{prop:FFD characterizations} guarantees that $R$ is an FFD. Finally, we observe that Proposition~\ref{prop:a locally finite intersection of FFDs is an FFD} can be used to give an alternative proof.
\end{proof}

\begin{cor} \label{cor:Dedekind domains are FFD}
	Integrally closed Noetherian domains and, in particular, Dedekind domains and rings of algebraic integers are FFDs.
\end{cor}

\smallskip
\subsection{The $D+M$ Construction}

This subsection is devoted to the $D+M$ construction, which is a rich source of examples in commutative ring theory. Let $T$ be an integral domain, and let $K$ and~$M$ be a subfield of $T$ and a nonzero maximal ideal of $T$, respectively, such that $T = K + M$. For a subdomain $D$ of $K$, set $R = D + M$. This construction was introduced and studied by Gilmer \cite[Appendix II]{rG68} for valuation domains $T$, and then it was investigated by J. Brewer and E. A. Rutter~\cite{BR76} for arbitrary integral domains.

To begin with, we consider units and irreducibles in the $D+M$ construction. Recall that an integral domain is quasilocal if it contains a unique maximal ideal. When we work with the $D+M$ construction, we will often denote an element of $T$ by $\alpha + m$, tacitly assuming that $\alpha \in K$ and $m \in M$.

\begin{lemma} \label{lem:units of the D+M construction}
	Let $T$ be an integral domain, and let $K$ and $M$ be a subfield of $T$ and a nonzero maximal ideal of $T$, respectively, such that $T = K + M$. For a subdomain $D$ of $K$, set $R = D + M$. Then the following statements hold.
	\begin{enumerate}
		\item $U(R) = U(T) \cap R$ if and only if $D$ is a field.
		\smallskip
		
		\item If $T$ is quasilocal, then $U(T) = \{\alpha + m \mid \alpha \in K^*\}$ and $U(R) = \{\alpha + m \mid \alpha \in U(D)\}$.
	\end{enumerate}
\end{lemma}

\begin{proof}
	(1) For the direct implication, consider $\alpha \in D^*$. As $\alpha \in K$, it follows that $\alpha \in U(T) \cap R = U(R)$, and so $\alpha^{-1} \in K \cap R = D$. Hence $D$ is a field. Conversely, assume that $D$ is a field. It is clear that $U(R) \subseteq U(T) \cap R$. To argue the reverse inclusion, take $\alpha + m_1 \in U(T) \cap R$, and then let $\beta + m_2$ be the inverse of $\alpha + m_1$ in $T$. Since $D$ is a field and $(\alpha + m_1)(\beta + m_2) = 1$, we see that $\beta = \alpha^{-1} \in D$, and so $\beta + m_2 = \alpha^{-1} + m_2 \in R$. Thus, $\alpha + m_1 \in U(R)$.
	\smallskip
	
	(2) The first equality is an immediate consequence of the fact that in a quasilocal domain the units are precisely the elements outside the unique maximal ideal. To check the second equality, take $\alpha + m_1 \in U(R)$ and let $\beta + m_2 \in R$ be the inverse of $\alpha + m_1$ in $R$. As in the previous part, $\alpha^{-1} = \beta \in D$, and so $\alpha \in U(D)$. Conversely, any $\alpha + m_1 \in R$ with $\alpha \in U(D)$ is a unit in $T$, and its inverse $\beta + m_2$ is such that $\beta \in D$, whence $\alpha + m_1 \in U(R)$.
\end{proof}

\begin{lemma} \label{lem:irreducibles of the D+M construction}
	Let $T$ be an integral domain, and let $K$ and $M$ be a subfield of $T$ and a nonzero maximal ideal of $T$, respectively, such that $T = K + M$. For a subdomain $D$ of $K$, set $R = D + M$. Then $\ii(R) \subseteq U(T) \cup \ii(T)$. Moreover, the following statements hold.
	\begin{enumerate}
		\item If $D$ is a field, then $\ii(R) \subseteq \ii(T)$. 
		\smallskip
		
		\item If $T$ is quasilocal and $D$ is a field, then $\ii(R) = \ii(T) \subseteq M$.
	\end{enumerate}
\end{lemma}

\begin{proof}
	Let $a = d + m \in \ii(R)$, where $d \in D$ and $m \in M$. If $m = 0$, then $a \in D^* \subseteq U(T)$. So assume that $m \neq 0$. Take $x,y \in T$ such that $a = xy$. If $d = 0$, then either $x \in M$ or $y \in M$. Assume that $x \in M$ and write $a = (k^{-1}x)(ky)$ for some $k \in K^*$ such that $ky \in R$. Because $a$ is irreducible in $R$, either $k^{-1}x \in U(R)$ or $ky \in U(R)$. Since $x \in M$, it follows that $ky \in U(R) \subseteq U(T)$. Thus, $y \in U(T)$, and so $a \in \ii(T)$. If $d \neq 0$, then there are $k_1, k_2 \in K^*$ with $k_1 k_2 = d$ such that $x = k_1(1 + m_1)$ and $y = k_2(1 + m_2)$. As $a = d(1+m_1)(1+m_2)$, either $d(1+m_1) \in U(R) \subseteq U(T)$ or $1+m_2 \in U(R) \subseteq U(T)$. Hence either $x$ or $y$ (or both) belongs to $U(T)$, and so $a \in U(T) \cup \ii(T)$. As a consequence, $\ii(R) \subseteq U(T) \cup \ii(T)$.
	\smallskip
	
	(1)  If $D$ is a field, then it follows from part~(1) of Lemma~\ref{lem:units of the D+M construction} that $\ii(R) \cap U(T)$ is empty. This, along with $\ii(R) \subseteq U(T) \cup \ii(T)$, implies the desired inclusion.
	\smallskip
	
	(2) By part~(1), $\ii(R) \subseteq \ii(T)$, and by part~(2) of Lemma~\ref{lem:units of the D+M construction}, $\ii(T) \subseteq M$. As $\ii(T) \subseteq R$, if an irreducible $a$ in $T$ factors as $a = xy$ for $x,y \in R$, then it follows from part~(1) of Lemma~\ref{lem:units of the D+M construction} that either $x \in U(T) \cap R = U(R)$ or $y \in U(T) \cap R = U(R)$, and so $a \in \ii(R)$. Thus, $\ii(T) \subseteq \ii(R)$.
\end{proof}

\begin{remark}
	With the notation as in Lemma~\ref{lem:irreducibles of the D+M construction}, the inclusion $\ii(R) \subseteq U(T) \cup \ii(T)$ may be proper. For instance, taking $R = \zz + X \qq[X]$ and $T = \qq[X]$, we can see that $4 \in U(T) \setminus \ii(R)$ and $X + 2 \in \ii(T) \setminus \ii(R)$. Moreover, the inclusion  $\ii(R) \subseteq \ii(T)$ may be proper even when $D$ is a field. To see this, it suffices to take $R = \qq + X \rr[X]$ and $T = \rr[X]$, and then observe that $X + \pi \in \ii(T) \setminus \ii(R)$.
\end{remark}

We proceed to examine when the $D+M$ construction yields BFDs and FFDs.

\begin{prop} \emph(\cite[Proposition~2.8]{AAZ90}\emph) \label{prop:BFD and D+M construction}
	Let $T$ be an integral domain, and let $K$ and $M$ be a subfield of $T$ and a nonzero maximal ideal of $T$, respectively, such that $T = K + M$. For a subdomain $D$ of~$K$, set $R = D + M$. Then~$R$ is a BFD if and only if $T$ is a BFD and $D$ is a field.
\end{prop}

\begin{proof}
	For the direct implication, suppose that $R$ is a BFD. Assume, by way of contradiction, that~$D$ is not a field, and take $d \in D$ such that $d^{-1} \notin D$. In this case, $d \notin U(R)$, and for every $m \in M$ the decomposition $m = d(d^{-1}m)$ ensures that $m \notin \ii(R)$. Then no element of $M \setminus \{0\}$ factors into irreducibles, contradicting that $R$ is atomic. Thus, $D$ must be a field.
	
	We proceed to argue that $T$ is a BFD. Fix $x \in T^* \setminus U(T)$, and take $k \in K^*$ such that $xk^{-1} \in R$. As $R$ is atomic, $xk^{-1}$ factors into irreducibles in $R$, and so Lemma~\ref{lem:irreducibles of the D+M construction} ensures that $x$ factors into irreducibles in $T$. Therefore $T$ is atomic. We can readily check that for every $m \in M$, the element $m$ (resp., $1+m$) is irreducible in $R$ if and only if $m$ (resp., $1+m$) is irreducible in $T$. Suppose that $xk^{-1} = \prod_{i=1}^r m_i \prod_{j=1}^s (\alpha_j + m'_j)$ for irreducibles $m_1, \dots, m_r \in M$ and $\alpha_1 + m'_1, \dots, \alpha_s + m'_s \in K^* + M$ of $T$. Set $\alpha = \prod_{j=1}^s \alpha_j $.  If $r = 0$, then $\alpha \in R$ and so $xk^{-1}$ factors as a product of $r+s$ irreducibles in~$R$ as $xk^{-1} = \alpha (1 + \alpha_1^{-1}m'_1) \prod_{j=2}^s(1 + \alpha_j^{-1}m'_j)$. If $r > 0$, then $xk^{-1}$ still factors as a product of $r+s$ irreducibles in~$R$ as $xk^{-1} = (\alpha m_1) \prod_{i=2}^r m_i \prod_{j=1}^s(1 + \alpha_j^{-1}m'_j)$. Hence $L_T(x) = L_T(xk^{-1}) \subseteq L_R(xk^{-1})$. Thus, $T$ is also a BFD.
	
	For the reverse implication, suppose that $T$ is a BFD and $D$ is a field. Fix $x \in R^* \setminus U(R)$, and write $x = \prod_{i=1}^r m_i \prod_{j=1}^s (\alpha_j + m'_j)$ for irreducibles $m_1, \dots, m_r \in M$ and $\alpha_1 + m'_1, \dots, \alpha_s + m'_s \in K^* + M$ in $T$. As before, set $\alpha = \prod_{j=1}^s \alpha_j$. Observe that if $r=0$, then $\alpha \in R$, and so $x$ factors into irreducibles in $R$ as $x = \alpha \prod_{j=1}^s (1 + \alpha_j^{-1}m'_j)$. If $r > 0$, then $x$ still factors into irreducibles in $R$ as $x = (\alpha m_1) \prod_{i=2}^r m_i \prod_{j=1}^s (1 + \alpha_j^{-1}m'_j)$. Hence~$R$ is atomic. Finally, observe that since $D$ is a field, the inclusion $\ii(R) \subseteq \ii(T)$ holds by Lemma~\ref{lem:irreducibles of the D+M construction}. This guarantees that $L_R(x) \subseteq L_T(x)$. Thus, $R$ is also a BFD.
\end{proof}

\begin{remark} \label{rem:atomicity/ACCP in D+M construction}
	With the notation as in Proposition~\ref{prop:BFD and D+M construction}, we have also proved that $R$ is atomic if and only if $T$ is atomic and $D$ is a field. The same assertion holds if we replace being atomic by satisfying ACCP (see~\cite[Proposition~1.2]{AAZ90}).
\end{remark}

Two of the most important special cases of the $D+M$ construction are the following.

\begin{example} \label{ex:BFD F_1 + XF_2[X] and F_1 + XF_2[[X]]}
	Let $F_1 \subsetneq F_2$ be a proper field extension.
	\begin{enumerate}
		\item Consider the subring $R_1 = F_1 + XF_2[X]$ of the ring of polynomials $F_2[X]$. Since $F_2[X]$ is a UFD, it is also a BFD. As $XF_2[X]$ is a nonzero maximal ideal of $F_2[X]$, it follows from Proposition~\ref{prop:BFD and D+M construction} that~$R_1$ is a BFD. Observe that $R_1$ is not a UFD because, for instance, $X$ is an irreducible that is not prime.
		\smallskip
		
		\item On the other hand, consider the subring $R_2 = F_1 + XF_2[[X]]$ of the ring of power series $F_2[[X]]$. As in the previous case, it follows from Proposition~\ref{prop:BFD and D+M construction} that $R_2$ is a BFD that is not a UFD.
	\end{enumerate}
	Finally, it is worth noting that certain ring-theoretic properties of $R_1$ and $R_2$ only depend on the field extension $F_1 \subsetneq F_2$. For instance, both $R_1$ and $R_2$ are Noetherian if and only if $[F_2 : F_1] < \infty$ \cite[Theorem~4]{BR76}, while both $R_1$ and $R_2$ are integrally closed if and only if $F_1$ is algebraically closed in $F_2$ (cf. \cite[page 35]{BR76}).
\end{example}

Now we give a result for FFDs that is parallel to Proposition~\ref{prop:BFD and D+M construction}.

\begin{prop} \emph(\cite[Proposition~5.2]{AAZ90}\emph)\label{prop:FFD and D+M construction}
	Let $T$ be an integral domain, and let $K$ and $M$ be a subfield of $T$ and a nonzero maximal ideal of $T$, respectively, such that $T = K + M$. For a subdomain $D$ of~$K$, set $R = D + M$. Then~$R$ is an FFD if and only if $T$ is an FFD, $D$ is a field, and $K^\ast/D^\ast$ is finite.
\end{prop}

\begin{proof}
	For the direct implication, suppose that $R$ is an FFD. Since $R$ is in particular a BFD, $D$ must be a field by Proposition~\ref{prop:BFD and D+M construction}. We proceed to verify that $K^*/D^*$ is finite. Take $m \in M \setminus \{0\}$. Note that in any factorization of $m$ into irreducibles of $R$, one of the irreducibles must belong to $M$. After replacing $m$ by such an irreducible, we can assume that $m$ belongs to both $\ii(R)$ and $\ii(T)$. Observe that for $\alpha, \beta \in K^*$ the elements $\alpha m$ and $\beta m$ are irreducibles in both $R$ and~$T$, and they are associate elements in $R$ if and only if $\alpha$ and $\beta$ determine the same coset of $K^*/D^*$. On the other hand, the set $\{\alpha m \mid \alpha \in K^* \} \subseteq R$ has only finitely many non-associate elements because it consists of divisors of~$m^2$ in~$R$. As a result, $K^*/D^*$ is a finite group.
	
	By Proposition~\ref{prop:FFD characterizations}, proving that $T$ is an FFD amounts to verifying that every $x \in T$ has finitely many non-associate irreducible divisors. After replacing $x$ by $\alpha x$ for some $\alpha \in K^*$, we can assume that $x \in R$. Suppose that $x_1, \dots, x_n$ form a maximal set of non-associate irreducible divisors of $x$ in $R$. Let $x = \prod_{i=1}^r m_i \prod_{j=1}^s (\alpha_j + m'_j)$ be a factorization of $x$ into irreducibles of $T$, where $\alpha_1, \dots, \alpha_s \in K^*$ and $m_1, \dots, m_r, m'_1, \dots, m'_s \in M$. If $x \in M$, then $r > 0$, and therefore, the elements $m_1, \dots, m_r$ and $1 + \alpha_1^{-1} m'_1, \dots, 1 + \alpha_s^{-1}m'_s$ are irreducible divisors of $x$ in $R$. Hence they are associate to some of the elements $x_1, \dots, x_n$ in $T$. On the other hand, if $x \notin M$, then $r = 0$. Therefore we can write $x = \alpha \prod_{j=1}^s (1 + \alpha_j^{-1}m'_j)$, where $\alpha = \prod_{j=1}^s \alpha_j \in D^*$. In this case, the elements $1 + \alpha_1^{-1} m'_1, \dots, 1 + \alpha_s^{-1}m'_s$ are irreducible divisors of $x$ in $R$, and as a result, they are associate to some of the elements $x_1, \dots, x_n$ in $T$. So in any case, the irreducible factors on the right-hand side of $x = \prod_{i=1}^r m_i \prod_{j=1}^s (\alpha_j + m'_j)$ are associate to some of the elements $x_1, \dots, x_n$. Thus, $x$ has finitely many non-associate irreducible divisors.
	
	For the reverse implication, take $x \in R^*$. We will verify that $x$ has only finitely many non-associate irreducible divisors in $R$. Assume that $x_1, \dots, x_n \in R$ form a maximal set of non-associate irreducible divisors of $x$ in $T$, and let $\alpha_1D^*, \dots, \alpha_m D^*$ be the cosets of $D^*$ in $K^*$. If $d \in R$ is an irreducible divisor of $x$ in $R$, then $d \in \ii(T)$ because $D$ is a field, and therefore, $d \in \bigcup_{j=1}^n x_j K^* = \bigcup_{i=1}^m \bigcup_{j=1}^n \alpha_i x_jD^*$. Thus, every irreducible divisor of $x$ in $R$ is associate to some element in $\{\alpha_i x_j \mid (i,j) \in \ldb 1,m \rdb \times \ldb 1,n \rdb \}$. It now follows from Proposition~\ref{prop:BFD and D+M construction} and Proposition~\ref{prop:FFD characterizations} that $R$ is an FFD.
\end{proof}

\begin{remark}
	With the notation as in Proposition~\ref{prop:FFD and D+M construction}, when $D$ is a field it follows from Brandis' Theorem~\cite{aB65} that $K^*/D^*$ is finite if and only if $K$ is finite or $D = K$.
\end{remark}

We conclude this section revisiting Example~\ref{ex:BFD F_1 + XF_2[X] and F_1 + XF_2[[X]]}.

\begin{example}
	Let $F_1 \subsetneq F_2$ be a field extension, and set $R_1 = F_1 + XF_2[X]$ and $R_2 = F_1 + XF_2[[X]]$. As with the properties of being Noetherian and integrally closed, whether $R_1$ and $R_2$ are FFDs only depends on the field extension $F_1 \subsetneq F_2$. Indeed, because $F_2[X]$ and $F_2[[X]]$ are both FFDs, Proposition~\ref{prop:FFD and D+M construction} guarantees that $R_1$ and $R_2$ are FFDs if and only if $F_2^*/F_1^*$ is finite. Since $F_1 \neq F_2$, it follows from Brandis' Theorem~\cite{aB65} that $F_2^*/F_1^*$ is finite if and only if $F_2$ is finite. Finally, if $F_2$ is finite, then $|U(R_1)| = |F_1^*| < \infty$, and it follows from Proposition~\ref{prop:SFFD characterizations} that $R_1$ is, in fact, an SFFD.
\end{example}

\bigskip
\section{Subrings, Ring Extensions,  and Localizations}
\label{sec:extensions and localization}

In this section, we study when being a BFD (resp., an FFD) transfers from an integral domain to its subrings and overrings. We pay special attention to ring extensions by localization.

\smallskip
\subsection{Inert Extensions}

Let $A \subseteq B$ be a ring extension. Following Cohn~\cite{pC68}, we call $A \subseteq B$ an \emph{inert extension} if $xy \in A$ for $x,y \in B^\ast$ implies that $ux, u^{-1}y \in A$ for some $u \in U(B)$. Let $A \subseteq B$ be an inert extension of integral domains. Take $x,y \in B$ such that $xy = a \in \ii(A) \setminus U(B)$, and then write $a = (ux)(u^{-1}y)$ for some $u \in U(B)$ such that $ux, u^{-1}y \in A$. So either $ux$ or $u^{-1}y$ belongs to $U(A)$, and therefore, either $x$ or $y$ belongs to $U(B)$. Thus, $a \in \ii(B)$. As a result, $\ii(A) \subseteq U(B) \cup \ii(B)$. We record this last observation, which was first given in \cite[Lemma~1.1]{AAZ92}.

\begin{remark} \label{rem:irreducibles in inert extensions}
	If $A \subseteq B$ is an inert extension of integral domains, then  $\ii(A) \subseteq U(B) \cup \ii(B)$.
\end{remark}

As a result of the previous remark, one can easily check that if $A \subseteq B$ is an inert extension of integral domains with $U(A) = U(B) \cap A$, then $\ii(A) = \ii(B) \cap A$.

\begin{example}
	Let $R$ be an integral domain. It is clear that the extension $R \subseteq R[X]$ is inert. On the other hand, consider the extension $R[X^2] \subseteq R[X]$. 
	Clearly, $U(R[X^2]) = U(R)$. Observe, in addition, that $XX = X^2 \in R[X^2]$ even though $uX \notin R[X^2]$ for any $u \in U(R)$. Hence the extension $R[X^2] \subseteq R[X]$ is not inert.
\end{example}

The extensions $D \subseteq R = D + M$ and $R \subseteq T = K + M$ in the $D + M$ construction are both inert.

\begin{lemma}
	Let $T$ be an integral domain, and let $K$ and $M$ be a subfield of $T$ and a nonzero maximal ideal of $T$, respectively, such that $T = K + M$. For a subdomain $D$ of $K$, set $R = D + M$. Then the extensions $D \subseteq R$ and $R \subseteq T$ are both inert.
\end{lemma}

\begin{proof}
	First, we prove that $D \subseteq R$ is inert. Take $\alpha_1 + m_1$ and $\alpha_2 + m_2$ in $R$ with nonzero $\alpha_1$ and $\alpha_2$ such that $(\alpha_1 + m_1)(\alpha_2 + m_2) \in D$. Then $(\alpha_1 + m_1)(\alpha_2 + m_2) = \alpha_1 \alpha_2$, and therefore, $1 + \alpha_1^{-1}m_1$ and $1 + \alpha_2^{-1}m_2$ are units in $R$ that are inverses of each other. After setting $u = 1 + \alpha_2^{-1}m_2$, we obtain $u(\alpha_1 + m_1) = \alpha_1 \in D$ and $u^{-1}(\alpha_2 + m_2) = \alpha_2 \in D$. Hence the extension $D \subseteq R$ is inert.
	\smallskip
	
	To show that $R \subseteq T$ is also inert, suppose that $xy \in R$ for some $x,y \in T$, and write $x = \alpha_1 + m_1$ and $y = \alpha_2 + m_2$. If $\alpha_1 = \alpha_2 = 0$, then $ux \in R$ and $u^{-1}y \in R$ for $u = 1$. So we assume that $\alpha_1 \neq 0$ or $\alpha_2 \neq 0$. If $\alpha_1 = 0$, then $u x \in R$ and $u^{-1}y \in R$ for $u = \alpha_2$. Similarly, if $\alpha_2 = 0$, then $u x \in R$ and $u^{-1} y \in R$ for $u = \alpha_1^{-1}$. Finally, if $\alpha_1 \alpha_2 \neq 0$, then $xy \in R$ implies that $\alpha_1 \alpha_2 \in D$, and so $ux \in R$ and $u^{-1}y \in R$ for $u = \alpha_2$.
\end{proof}

The following example shows that extensions by localization are not necessarily inert.

\begin{example}
	Let $K$ be a field and consider the integral domain $R = K[X^2,X^3]$. First, we observe that the subset $S = \{1, X^2, X^3, \ldots\}$ of $R$ is a multiplicative set and $R_S = K[X, X^{-1}]$. In addition, $U(R_S) = \{\alpha X^n \mid \alpha \in K^* \text{ and } n \in \zz\}$. Because $(1-X)(1+X) = 1- X^2 \in R$, in order for the extension $R \subseteq R_S$ to be inert, there must be an integer $n$ such that $X^n(1-X) \in R$ and $X^{-n}(1+X) \in R$, which is not possible. Hence the extension $R \subseteq R_S$ is not inert.
\end{example}

However, localizing at certain special multiplicative sets yields inert extensions. Following~\cite{AAZ92}, we say that a saturated (i.e., divisor-closed) multiplicative set $S$ of an integral domain $R$ is a \emph{splitting multiplicative set} if we can write every $x \in R$ as $x = rs$ for some $r \in R$ and $s \in S$ with $rR \cap s'R = rs' R$ for every $s' \in S$.

\begin{lemma} \emph(\cite[Proposition~1.5]{AAZ92}\emph) \label{lem:localizations at splitting MS are inert}
	Let $R$ be an integral domain, and let $S$ be a splitting multiplicative set of $R$. Then $R \subseteq R_S$ is an inert extension.
\end{lemma}

\begin{proof}
	Take $x,y \in R_S$ such that $xy = r \in R$. As $S$ is a splitting multiplicative set, there are $a,b \in R$ and $s,t,u,v \in S$ with $x = ast^{-1}$ and $y = buv^{-1}$ such that $aR \cap s'R = as' R$ and $bR \cap s'R = bs' R$ for every $s' \in S$. Since $absu = rtv \in bR \cap vtR = bvtR$, there is a $c \in R$ satisfying $asu = cvt$. Taking $w = u/v \in U(R_S)$, we see that $wx = c$ and $w^{-1}y = b$, which are both in $R$. Hence the extension $R \subseteq R_S$ is inert.
\end{proof}

Multiplicative sets generated by primes are always saturated. The next proposition characterizes the multiplicative sets generated by primes that are splitting multiplicative sets.

\begin{lemma} \emph(\cite[Proposition~1.6]{AAZ92}\emph) \label{lem:when multiplicative sets generated by primes are SMS}
	Let $R$ be an integral domain, and let $S$ be a multiplicative set of~$R$ generated by primes. Then the following statements are equivalent.
	\begin{enumerate}
		\item[(a)] $S$ is a splitting multiplicative set.
		\smallskip
		
		\item[(b)] $\bigcap_{n \in \nn} p^n R = \bigcap_{n \in \nn} p_n R = \{0\}$ for every prime $p \in S$ and every sequence $(p_n)_{n \ge 1}$ of non-associate primes in $S$.
		\smallskip
	
		\item[(c)] For every nonunit $x \in R^\ast$, there is an $n_x \in \nn$ such that $x \in p_1 \cdots p_n R$ for $p_1, \dots, p_n \in S$ implies that $n \le n_x$.
	\end{enumerate}
\end{lemma}

\begin{proof}
	(b) $\Leftrightarrow$ (c): It follows easily.
	\smallskip
	
	(a) $\Rightarrow$ (c): Take a nonunit $x \in R^*$. As $S$ is generated by primes, we can write $x = r p'_1 \cdots p'_{n_x}$ for some $n_x \in \nn$ and (possibly repeated) primes $p'_1, \dots, p'_{n_x} \in S$ such that $rR \cap s'R = rs'R$ for every $s' \in S$. Observe that none of the primes $p$ in $S$ divides $r$ as, otherwise, $rpR = rR \cap pR = rR$, which would imply that $p$ is a unit. As a result, if $x \in p_1 \cdots p_n R$ for some $p_1, \dots, p_n \in S$, then $n \le n_x$.
	\smallskip
	
	(c) $\Rightarrow$ (a): Fix a nonunit $x \in R^*$, and then take the smallest $n_x \in \nn$ among those satisfying (c). So there are (possibly repeated) primes $p_1, \dots, p_{n_x} \in S$ such that $s = p_1 \cdots p_{n_x} \in S$ divides $x$. Take $a \in R$ such that $x = as$. It is clear that no prime in $S$ can divide $a$. Now if $y = ar \in aR \cap s'R$ for some $r \in R$ and $s' \in S$, then the fact that $S$ is generated by primes (none of them dividing $a$) guarantees that $s'$ divides $r$, and so $y \in as'R$. Thus, $aR \cap s'R = as'R$ for every $s' \in S$, and we conclude that $S$ is a splitting multiplicative set.
\end{proof}

\begin{cor} \emph(\cite[Corollary~1.7]{AAZ92}\emph) \label{cor:MS generated by primes in atomic monoids are SMS}
	Let $R$ be an atomic domain. Then every multiplicative set of $R$ generated by primes is a splitting multiplicative set.
\end{cor}

\begin{proof}
	Let $S$ be a multiplicative set of $R$ generated by primes. Suppose that $x$ is a nonunit in $R^*$. Because $R$ is atomic, $x = a_1 \cdots a_n$ for some $a_1, \dots, a_n \in \ii(R)$. If $x \in p_1 \cdots p_m R$ for some primes $p_1, \dots, p_m \in S$, then there exists a permutation $\sigma \colon \ldb 1,n \rdb \to \ldb 1,n \rdb$ such that $p_i$ and $a_{\sigma(i)}$ are associates in $R$ for every $i \in \ldb 1,m \rdb$. As a result, $m \le n$. It then follows from Lemma~\ref{lem:when multiplicative sets generated by primes are SMS} that $S$ is a splitting multiplicative set.
\end{proof}

In general, multiplicative sets generated by primes are not always splitting multiplicative sets. On the other hand, there are splitting multiplicative sets that are not generated by primes. The following examples, which can be found in \cite[Example~1.8]{AAZ92} confirm these observations.

\newpage
\begin{example} \hfill
	\begin{enumerate}
		\item Let $R$ be a two-dimensional valuation domain with height-one prime ideal $P$ and principal maximal ideal $M = pR$. In addition, let $S$ be the multiplicative set of $R$ generated by $p$. Since $\bigcap_{n \in \nn} p^nR = P \neq \{0\}$, it follows from Lemma~\ref{lem:when multiplicative sets generated by primes are SMS} that $S$ is not a splitting multiplicative set. Finally, note that $p$ can be chosen so that $V[1/p]$ is a DVR.
		\smallskip
		
		\item Consider the integral domain $R = \zz + X\qq[[X]]$, and set $S = \zz^\ast$. It is clear that $S$ is a multiplicative set generated by primes and $R_S = \qq[[X]]$. However, Lemma~\ref{lem:when multiplicative sets generated by primes are SMS} guarantees that~$S$ is not a splitting multiplicative set because $\bigcap_{n \in \nn} p^nR = X\qq[[X]] \neq \{0\}$ for every prime $p$ in~$S$ and $\bigcap_{n \in \nn} p_nR = X\qq[[X]] \neq \{0\}$ for every sequence $(p_n)_{n \in \nn}$ of non-associate primes in~$S$.
		\smallskip
		
		\item Let $R$ be a GCD-domain that is not a UFD (for instance, a non-discrete one-dimensional valuation domain), and consider the integral domain $R[X]$. It is clear that $S = R^\ast$ is a multiplicative set of $R[X]$. Since $R$ is a GCD-domain, every $p(X) \in R[X]^\ast$ can be written as $c(p)p'(X)$, where $c(p) \in S$ is the content of $p(X)$ and $p'(X) \in R[X]$ has content~$1$. For $s \in S$, take $p'(X)q(X) \in p'(X)R[X] \cap sR[X]$, and note that $c(q) \in sR$ by Gauss' Lemma. This implies that $p'(X)q(X) \in sp'(X)R[X]$, and so $p'(X)R[X] \cap sR[X] = sp'(X)R[X]$. As a result, $S$ is a splitting multiplicative set. Finally, observe that $S$ cannot be generated by primes because $R$ is not a UFD.
	\end{enumerate}
\end{example}

As for the case of splitting multiplicative sets, multiplicative sets generated by primes yield inert extensions.

\begin{lemma} \emph(\cite[Proposition~1.9]{AAZ92}\emph) \label{lem:localization at prime-generated MS are inert}
	Let $R$ be an integral domain, and let $S$ be a multiplicative set of~$R$ generated by primes. Then $R \subseteq R_S$ is an inert extension.
\end{lemma}

\begin{proof}
	Take $x,y \in R_S^*$ such that $xy \in R$. Now write $x = a(p_1 \cdots p_m)^{-1}$ and $y = b(q_1 \cdots q_n)^{-1}$ for elements $a,b \in R$ and primes $p_1, \dots, p_m, q_1, \dots, q_n$ in $S$ such that $a \notin p_i R$ for any $i \in \ldb 1,m \rdb$ and $b \notin q_j R$ for any $j \in \ldb 1,n \rdb$. Because $xy \in R$, it follows that $b \in p_1 \cdots p_m R$ and $a \in q_1 \cdots q_n R$. Then after setting $u = p_1 \cdots p_m (q_1 \cdots q_n)^{-1} \in U(R_S)$, we see that $ux = a (q_1 \cdots q_n)^{-1} \in R$ and $u^{-1}y = b (p_1 \cdots p_m)^{-1} \in R$. Thus, $R \subseteq R_S$ is an inert extension.
\end{proof}

Some of the results we will discuss in the next two subsections have been generalized by D.~D. Anderson and J.~R. Juett in \cite{AJ15} to the context of inert extensions $A \subseteq B$ of integral domains that satisfy $U(A) = U(B) \cap A$ and $B = A U(B)$.

\smallskip
\subsection{Subrings}

In general, the properties of being a BFD or an FFD are not inherited by subrings.

\begin{example}
	Let $A = \zz + X\qq[X] \subseteq B = \qq[X]$. Since~$B$ is a UFD, it is in particular a BFD and an FFD. However, as $\zz$ is not a field, it follows from Proposition~\ref{prop:BFD and D+M construction} that~$A$ is neither a BFD nor an FFD.
\end{example}

The following proposition is an immediate consequence of Proposition~\ref{prop:BF inherited by inert submonoids}.

\begin{prop} \emph(\cite[page~9]{AAZ90}\emph) \label{prop:BFD underrings}
	Let $A \subseteq B$ be an extension of integral domains with $U(A) = U(B) \cap A$. If~$B$ is a BFD, then $A$ is also a BFD. In particular, if an integral extension of an integral domain $A$ is a BFD, then $A$ is also a BFD.
\end{prop}

As the following example shows, the converse of Proposition~\ref{prop:BFD underrings} does not hold.

\begin{example}
	Consider the extension of integral domains $A = \qq[X] \subseteq B = \qq[M]$, where~$M$ is the additive submonoid of $\qq_{\ge 0}$ generated by the set $\{ 1/p \mid p \in \pp \}$. Since $M$ is a reduced monoid, $U(A) = \qq^\ast = U(B) = U(B) \cap A$. It is clear that $A$ is a BFD. However, we have seen in Example~\ref{ex:ACCP domain that is not a BFD} that $B$ is not a BFD.
\end{example}

By strengthening the hypothesis of Proposition~\ref{prop:BFD underrings}, we can obtain a version for FFDs.

\begin{prop}  \emph(\cite[Theorem~3]{AM96}\emph) \label{prop:FFD underrings}
	Let $A \subseteq B$ be an extension of integral domains. Suppose that $(U(B) \cap \emph{qf}(A)^\ast)/U(A)$ is finite. If $B$ is an FFD, then $A$ is also an FFD.
\end{prop}

\begin{proof}
	Take $x \in A^\ast$, and let $D \subseteq A^\ast$ be the set of divisors of $x$ in $A^\ast$. Since $B$ is an FFD, Proposition~\ref{prop:FFD characterizations} ensures that $x$ has only finitely many non-associate divisors in $B$, namely, $x_1, \dots, x_m$. Clearly, $D \subseteq \bigcup_{i=1}^m x_i U(B)$. Because $(U(B) \cap \text{qf}(A)^\ast)/U(A)$ is finite, only finitely many cosets $gU(A)$ of $U(B) \cap \text{qf}(A)^\ast$ intersect $D$. Let them be $g_1 U(A), \dots, g_n U(A)$. For each pair $(i,j) \in \ldb 1,m \rdb \times \ldb 1,n \rdb$, take $x_{i,j} \in D$ such that $x_{i,j} \in x_i U(B) \bigcap g_j U(A)$. Now take $y \in A^*$ with $x \in y A$. Since $y \in D$, there must be a pair $(i,j) \in \ldb 1,m \rdb \times \ldb 1,n \rdb$ such that $y \in x_i U(B) \cap g_j U(A) \subseteq \text{qf}(A)$, and so $y x^{-1}_{i,j} \in U(B) \cap \text{qf}(A) \cap U(A) = U(A)$. Hence every divisor of $x$ in $A^\ast$ is an associate of one of the elements $x_{i,j}$. Thus, $A$ is an FFD by Proposition~\ref{prop:FFD characterizations}.
\end{proof}

The converse of Proposition~\ref{prop:FFD underrings} does not hold, as the next example illustrates.

\begin{example}
	Consider the extension of integral domains $A = \qq[X] \subseteq B = \qq[M]$, where~$M$ is the additive monoid $\{0\} \cup \qq_{\ge 1}$. Since $M$ is reduced, $U(A) = \qq^\ast = U(B) = U(B) \cap \text{qf}(A)^*$. Also,~$A$ is an FFD because it is a UFD. However, we have already verified in Example~\ref{ex:BFD that is neither an HFD nor an FFD} that the monoid domain~$B$ is not an FFD.
\end{example}

Let $R$ be an integral domain, and let $S$ be a multiplicative set of $R$. The fact that $R_S$ satisfies either the bounded or the finite factorization property does not imply that $R$ does. For instance, the quotient field of every (non-BFD) integral domain is trivially an FFD. The next ``Nagata-type" theorem provides a scenario where the bounded and finite factorization properties are inherited by an integral domain from some special localization.

\begin{theorem} \emph(\cite[Theorems~2.1 and~3.1]{AAZ92}\emph) \label{thm:underring localization BFD/FFD}
	Let $R$ be an integral domain, and let $S$ be a splitting multiplicative set of $R$ generated by primes. Then the following statements hold.
	\begin{enumerate}
		\item  If $R_S$ is a BFD, then $R$ is a BFD.
		\smallskip
		
		\item  If $R_S$ is an FFD, then $R$ is an FFD.
	\end{enumerate}
\end{theorem}

\begin{proof}
	Assume first that $R_S$ is atomic. Fix a nonunit $x \in R^\ast$. As $S$ is a splitting multiplicative set generated by primes, we can write $x = rs$ for some $r \in R$ and $s \in S$ such that $s$ is a product of primes and no prime in $S$ divides~$r$ in $R$. Take $a_1, \dots, a_n \in \ii(R_S)$ such that $r = a_1 \cdots a_n$. Since no prime in $S$ divides~$r$, we can assume that $a_1, \dots, a_n \in \ii(R)$. Hence $x$ can be written in $R$ as a product of irreducibles. Thus, $R$ is atomic.
	\smallskip
	
	(1) Now assume that $R_S$ is a BFD. By the conclusion of the previous paragraph, $x = a_1 \cdots a_n$ for some $a_1, \dots, a_n \in \ii(R)$. Assume, without loss of generality, that there is a $j \in \ldb 0, n \rdb$ such that $a_{j+1}, \dots, a_n$ are the elements among $a_1, \dots, a_n$ that belong to $S$ and, therefore, are primes. Set $a = a_1 \cdots a_j$ and $s = a_{j+1} \cdots a_n$. Because $R_S$ is a BFD, there is an $\ell \in \nn$ such that each factorization of $a$ in $R_S$ has at most $\ell$ irreducible factors. As each irreducible in $\ii(R) \setminus S$ dividing $a$ remains irreducible in $R_S$, the set $L_R(a)$ is bounded by $\ell$. Suppose now that $x = b_1 \cdots b_m$ is another factorization of $x$ in $R$. Then there are exactly $n-j$ irreducibles (counting repetitions) in $b_1, \dots, b_m$ that are primes in $S$; let them be $b_{m-n+j+1}, \dots, b_m$. Set $b = b_1 \cdots b_{m-n+j}$ and $t = b_{m-n+j+1} \cdots b_m$. It is clear that $tR = sR$, and so $bR = aR$. In particular, $\max L_R(b) = \max L_R(a) \le \ell$, and therefore, $\max L_R(x) \le \ell + n - j$. Thus, $R$ is a BFD.
	\smallskip
	
	(2) Finally, assume that $R_S$ is an FFD. Take a nonunit $x \in R^\ast$, and write $x = rp_1 \cdots p_n$ for primes $p_1, \dots, p_n \in S$ so that no prime in $S$ divides $r$. Since $R_S$ is an idf-domain by Proposition~\ref{prop:FFD characterizations}, $r$ has only finitely many irreducible divisors in $R_S$, namely, $a_1, \dots, a_m$. As we did in the first paragraph, we can assume that $a_1, \dots, a_m \in \ii(R)$. Now suppose that $y \in \ii(R)$ divides $x$ in $R$. Then either $y$ is an associate of some of the primes in $S$ or $y \in \ii(R_S)$. In the latter case, $sy = ta_i$ for some $i \in \ldb 1,m \rdb$ and $s,t \in S$. Then $y$ and $a_i$ are associates if $y$ is not an associate of some prime in $S$. As a result, $a_1, \dots, a_m, p_1, \dots, p_n$ account, up to associates, for all irreducible divisors of $x$ in $R$. Hence $R$ is an atomic idf-domain, and thus an FFD by Proposition~\ref{prop:FFD characterizations}.
\end{proof}

\begin{remark}
	Theorem~\ref{thm:underring localization BFD/FFD} still holds if we replace BFD by either ACCP or UFD.
\end{remark}

\smallskip
\subsection{Ring Extensions and Overrings}

Let $A \subseteq B$ be an extension of integral domains. Often $B$ is not a BFD (resp., an FFD) even when $A$ is a BFD (resp., an FFD) and  $U(A) = U(B) \cap \text{qf}(A)$.

\begin{example} \label{ex:non-BF/FF extension}
	 Consider the extension of integral domains $A = \rr[X] \subseteq B =  \rr[\qq_{\ge 0}]$. Because $A$ is a UFD, it is, in particular, a BFD. On the other hand, $B$ is not even atomic because the additive monoid $\qq_{\ge 0}$ is not atomic. Finally, we observe that $U(A) = \rr^\ast = U(B)$, from which the equality $U(A) = U(B) \cap \text{qf}(A)$ follows.
\end{example}

However, if we require the ideal $[A :_A B] = \{r \in A \mid rB \subseteq A\}$ to be nonzero, then the property of being an FFD passes from $A$ to $B$.

\begin{prop}  \emph(\cite[Theorem~4]{AM96}\emph) \label{prop:FFD to ring extension}
	Let $A \subseteq B$ be an extension of integral domains, and suppose that $[A :_A B]$ is nonzero. If $A$ is an FFD, then the group $U(B)/U(A)$ is finite and $B$ is an FFD.
\end{prop}

\begin{proof}
	Let $x$ be a nonzero nonunit in $[A :_A B]$. Observe that for every $u \in U(B)$, the fact that $ux, u^{-1}x \in A$ implies that $x^2 = (ux)(u^{-1}x) \in uxA$. As $A$ is an FFD, it follows from Proposition~\ref{prop:FFD characterizations} that the set $\{uxA \mid u \in U(B)\}$ is finite, and therefore, we can take $u_1, \dots, u_n \in U(B)$ such that for every $u \in U(B)$, the equality $ux A = u_ix A$ holds for some $i \in \ldb 1,n \rdb$. Then for every $u \in U(B)$, we can take $i \in \ldb 1,n \rdb$ and $v \in U(A)$ such that $u = u_iv$, whence $uU(A) = u_ivU(A) = u_iU(A)$. As a result, the group $U(B)/U(A)$ is finite.
	
	We proceed to argue that $B$ is an FFD. As before, let $0 \neq x \in [A :_A B]$. Let $b \in B$, and suppose that $y$ is a divisor of $b$ in $B$. Then $(xy)(xy') = x^2b$ for some $y' \in B$, and both $xy$ and $xy'$ belong to $A$. Therefore $x^2b A \subseteq xy A$, and so $xb A \subseteq y A$. As $A$ is an FFD and $xb \in A$, Proposition~\ref{prop:FFD characterizations} guarantees that the set $\{yA \mid y \text{ divides } b \text{ in } B\}$ is finite, and thus, the set $\{yB \mid y \text{ divides } b \text{ in } B\}$ is also finite. As a consequence, $B$ is an FFD.
\end{proof}

Next we characterize Noetherian FFDs.

\begin{prop}  \emph(\cite[Theorem~6]{AM96}\emph) \label{prop:characterization of Noetherian FFDs}
	The following statements are equivalent for a Noetherian domain~$R$.
	\begin{enumerate}
		\item[(a)] $R$ is an FFD.
		\smallskip
		
		\item[(b)] If $S$ is an overring of $R$ that is a finitely generated $R$-module, then $U(S)/U(R)$ is finite.
		\smallskip
		
		\item[(c)] There is an FFD overring $T$ of $R$ that is integral over $R$ such that if $S$ is an intermediate domain of the extension $R \subseteq T$ that is a finitely generated $R$-module, then $U(S)/U(R)$ is finite. 
	\end{enumerate}
\end{prop}

\begin{proof}
	(a) $\Rightarrow$ (b): It follows from Proposition~\ref{prop:FFD to ring extension}.
	\smallskip
	
	(b) $\Rightarrow$ (c): Take $T$ to be the integral closure of $R$. Since $R$ is a Noetherian domain, it follows from the Mori-Nagata Theorem that $T$ is a Krull domain. As a consequence, it follows from Theorem~\ref{thm:Krull domains are FFDs} that~$T$ is an FFD. 
	\smallskip
	
	(c) $\Rightarrow$ (a): Let $T$ be an overring of $R$ satisfying the conditions in~(c). Suppose towards a contradiction that $R$ is not an FFD, and take a nonunit $r \in R^*$ with infinitely many non-associate divisors. Since every divisor of $r$ in $R$ is also a divisor of $r$ in $T$, the fact that $T$ is an FFD guarantees the existence of a sequence $(r_n)_{n \in \nn}$ of non-associate divisors of $r$ in $R$ such that $r_1 T = r_n T$ for every $n \in \nn$. Let $I$ be the ideal generated by the terms of the sequence $(r_n)_{n \in \nn}$. Since $R$ is Noetherian,~$I$ is generated by $r_1, \dots, r_m$ for some $m \in \nn$. Consider the overring $S = R[\{r_j r_1^{-1} \mid j \in \ldb 2, m \rdb \}]$ of~$R$. For every $n \in \nn$, the equality $r_1 T = r_n T$ implies that $r_n r_1^{-1} \in U(T)$. Therefore $S$ is an intermediate domain of the extension $R \subseteq T$. Because $S$ is a finitely generated $R$-module, the group $U(S)/U(R)$ is finite by~(c). This, together with the fact that $r_n r_1^{-1} \in S \cap U(T) = U(S)$ for every $n \in \nn$, ensures the existence of $i,j \in \nn_{\ge 2}$ such that $r_i r_1^{-1} U(R) = r_j r_1^{-1} U(R)$. However, this implies that $r_i U(R) = r_j U(R)$, which is a contradiction.
\end{proof}

\begin{cor} \emph{(}\cite[Theorem~7]{fHK92}\emph{)}
	Let $R$ be a Noetherian domain whose integral closure $T$ is a finitely generated $R$-module. Then $R$ is an FFD if and only if the group $U(T)/U(R)$ is finite.
\end{cor}

Now we consider whether the bounded and finite factorization properties transfer from an integral domain to its extensions by localization. As the following example illustrates, such transfers do not happen in general.

\begin{example} \label{ex:FFD with a non-atomic localization}
	Let $R = \rr[M]$ be the monoid domain of $M = \langle 1 - 1/n \mid n \in \nn \rangle$ over $\rr$, and let~$S$ be the multiplicative set $\{X^q \mid q \in M\}$. The generating sequence $(1 - 1/n)_{n \in \nn}$ is increasing, so~$M$ is one of the FFMs in Example~\ref{ex:FF positive monoids}. In addition, as $M$ can be generated by an increasing sequence, it is not hard to argue that $\rr[M]$ is an FFD (this is similar to the proof of \cite[Theorem~4.3.3]{fG20a}). In particular, $\rr[M]$ is a BFD. On the other hand, it follows from \cite[Proposition~3.1]{GGT19} that $\gp(M) = \qq$, and therefore,  $R_S = \rr[\qq_{\ge 0}]$ is not even atomic. Hence $R_S$ is not a BFD (or an FFD).
\end{example}

If an extension of an integral domain by localization is inert, then both the bounded and finite factorization properties transfer.

\begin{theorem}  \emph(\cite[Theorems~2.1 and~3.1]{AAZ92}\emph) \label{thm:overring localization BFD/FFD}
	Let $R$ be an integral domain, and let $S$ be a multiplicative set of $R$ such that $R \subseteq R_S$ is an inert extension. Then the following statements hold.
	\begin{enumerate}
		\item  If $R$ is a BFD, then $R_S$ is a BFD.
		\smallskip
		
		\item If $R$ is an FFD, then $R_S$ is an FFD.
	\end{enumerate}
\end{theorem}

\begin{proof}
	Suppose first that $R$ is atomic. Take a nonzero nonunit $x$ in $R_S$ and write it as $x = r/s$, where $r \in R$ and $s \in S$. Since $R$ is atomic, there are $a_1, \dots, a_n \in \ii(R)$ such that $r = a_1 \cdots a_n$. As the extension $R \subseteq R_S$ is inert, in light of Remark~\ref{rem:irreducibles in inert extensions} we can assume that $a_1, \dots, a_j \in \ii(R_S)$ and $a_{j+1}, \dots, a_n \in U(R_S)$ for some $j \in \ldb 0, n \rdb$, and so $a_1 \cdots a_j \in Z_{R_S}(x)$. Thus, $R_S$ is an atomic domain.
	\smallskip
	
	(1) Assume now that $R$ is a BFD. To argue that $R_S$ is indeed a BFD, suppose that the principal ideal $xR_S$ of $R_S$ is strictly contained in the principal ideal $yR_S$. Assume, without loss of generality, that $x,y \in R$. Take $r \in R$ and $s \in S$ such that $x = y(r/s)$. Since the extension $R \subseteq R_S$ is inert, there is a $u \in U(R_S)$ such that $uy$ and $u^{-1}(r/s)$ are both elements of $R$. Setting $y' = uy$, we see that $xR = (uy)(u^{-1}(r/s))R \subsetneq uyR$, where the inclusion is strict because $r/s \notin U(R_S)$. Hence $xR$ is properly contained in $uyR$, and $uyR_S = yR_S$. Since $R$ is a BFD, it follows from Proposition~\ref{prop:BFD characterizations} that $R_S$ is also a BFD.
	\smallskip
	
	(2) Finally, assume that $R$ is an FFD. Take a nonzero nonunit $r/s \in R_S$, and let $r_1, \dots, r_n$ form a maximal set of non-associate divisors of $r$ in $R$. Let $y \in R_S$ be a divisor of $r$ in $R_S$, and write $r = y y'$ for some $y' \in R_S$. As the extension $R \subseteq R_S$ is inert, there is a $u \in U(R_S)$ such that $uy$ and $u^{-1}y'$ belong to $R$. Then $y = u^{-1}vr_i$ for some $i \in \ldb 1,n \rdb$ and $v \in U(R)$. As a result, $r_1, \dots, r_n$ form a maximal set of non-associate divisors of $r/s$ in $R_S$. Hence Proposition~\ref{prop:FFD characterizations} guarantees that $R_S$ is an FFD.
\end{proof}

Combining Theorem~\ref{thm:overring localization BFD/FFD} and Lemmas~\ref{lem:localizations at splitting MS are inert} and~\ref{lem:localization at prime-generated MS are inert}, we obtain the following corollary.

\begin{cor} \emph(\cite[Corollary~2.2]{AAZ92}\emph) \label{cor:overring localization BFD/FFD}
	Let $R$ be an integral domain, and let $S$ be a multiplicative set of~$R$ such that~$S$ is either generated by primes or a splitting multiplicative set. If $R$ is a BFD (resp., an FFD), then $R_S$ is a BFD (resp., an FFD).
\end{cor}

\begin{remark}
	Theorems~\ref{thm:underring localization BFD/FFD} and~\ref{thm:overring localization BFD/FFD} hold if we replace being a BFD by being an atomic domain, satisfying ACCP, or being a UFD (see~\cite[Theorems~2.1 and~3.1]{AAZ92}).
\end{remark}

The algebraic closures of an integral domain are among the most useful and investigated overrings. We conclude this section illustrating that neither the bounded nor the finite factorization property ascend to the seminormal, integral, or complete integral closures.

\begin{example}
	Consider the submonoid $M = \langle (3/2)^n \mid n \in \nn_0 \rangle$ of $\qq_{\ge 0}$. By \cite[Propostion~3.1]{GGT19}, the seminormal closure of $M$ is $M' = \zz[1/2] \cap \qq_{\ge 0}$, where $\zz[1/2]$ denotes the localization of $\zz$ at the multiplicative set $\{2^n \mid n \in \nn_0\}$. Now consider the monoid domain $R = \qq[M]$. It follows from \cite[Theorem~4.3]{fG20a} that $R$ is an FFD (cf. Example~\ref{ex:FFD with a non-atomic localization}), while it follows from \cite[Theorem~5.3]{fG20a} that $R' = \widetilde{R} = \widehat{R} = \qq[M']$, where $R'$, $\widetilde{R}$, and $\widehat{R}$ are the seminormal, root, and complete integral closures of $R$, respectively. Since $M$ is not finitely generated, it follows from \cite[Proposition~3.1]{GGT19} that $M'$ is antimatter (i.e., contains no irreducibles). Therefore $X$ cannot be written as a product of irreducibles in~$R'$. Then although $R$ is an FFD (and so a BFD), $R'$ ($\widetilde{R}$ or $\widehat{R}$) is not even atomic.
\end{example}
\smallskip

We conclude this subsection with a few words about directed unions of integral domains in connection to the bounded and finite factorization properties. Recall that a partial order $\Gamma$ is a directed set if for all $\alpha, \beta \in \Gamma$, there is a $\theta \in \Gamma$ such that $\alpha \le \theta$ and $\beta \le \theta$. A family $(R_\gamma)_{\gamma \in \Gamma}$ of integral domains indexed by a nonempty directed set $\Gamma$ is called a \emph{directed family} of integral domains if $R_\alpha$ is a subring of $R_\beta$ whenever $\alpha \le \beta$. In this case, $\bigcup_{\gamma \in \Gamma} R_\gamma$ is called the \emph{directed union} of $(R_\gamma)_{\gamma \in \Gamma}$.

As the next theorem states, the property of being a BFD (or an FFD) passes from the members of a directed family of integral domains to its directed union provided that every extension in the directed family is inert.

\begin{lemma} \label{lem:inert directed family yield inert directed unions}
	Let $(R_\gamma)_{\gamma \in \Gamma}$ be a directed family of integral domains such that every extension $R_\alpha \subseteq R_\beta$ is inert. If $R$ is the directed union of $(R_\gamma)_{\gamma \in \Gamma}$, then the extension $R_\gamma \subseteq R$ is inert for every $\gamma \in \Gamma$.
\end{lemma}

\begin{proof}
	Fix $\gamma \in \Gamma$, and consider the extension $R_\gamma \subseteq R$. Take $x,y \in R^*$ such that $xy \in R_\gamma$. Then $x \in R_\alpha$ and $y \in R_\beta$ for some $\alpha, \beta \in \Gamma$. Since $(R_\gamma)_{\gamma \in \Gamma}$ is a directed family, there is a $\theta \in \Gamma$ such that $x,y \in R_\theta$ and $R_\gamma \subseteq R_\theta$. As the extension $R_\gamma \subseteq R_\theta$ is inert, $ux, u^{-1}y \in R_\gamma$ for some $u \in U(R_\theta) \subseteq U(R)$. Thus, the extension $R_\gamma \subseteq R$ is also inert.
\end{proof}

\begin{theorem} \emph(\cite[Theorem~5.2]{AAZ92}\emph) \label{thm:BF/FF for directed unions}
	Let $(R_\gamma)_{\gamma \in \Gamma}$ be a directed family of integral domains such that every extension $R_\alpha \subseteq R_\beta$ is inert. Then the following statements hold.
	\begin{enumerate}
		\item If $R_\gamma$ is a BFD for every $\gamma \in \Gamma$, then the directed union $\bigcup_{\gamma \in \Gamma} R_\gamma$ is a BFD.
		\smallskip
		
		\item  If $R_\gamma$ is an FFD for every $\gamma \in \Gamma$, then the directed union $\bigcup_{\gamma \in \Gamma} R_\gamma$ is an FFD.
	\end{enumerate}
\end{theorem}

\begin{proof}
	(1) Suppose that $R_\gamma$ is a BFD for every $\gamma \in \Gamma$, and set $R = \bigcup_{\gamma \in \Gamma} R_\gamma$. Take a nonunit $x \in R^*$. Since $x \in R_\gamma$ for some $\gamma \in \Gamma$, and $R_\gamma$ is atomic, there are $a_1, \dots, a_n \in \ii(R_\gamma)$ such that $x = a_1 \cdots a_n$. As the extension $R_\gamma \subseteq R$ is inert by Lemma~\ref{lem:inert directed family yield inert directed unions}, it follows from Remark~\ref{rem:irreducibles in inert extensions} that $a_1, \dots, a_n$ are either irreducibles or units in $R$. Hence $R$ is an atomic domain.
	
	Now take $x,y \in R$ such that $xR \subsetneq yR$, and write $x = yr$ for some $r \in R^*\setminus U(R)$. Since $x \in R_\alpha$ for some $\alpha \in \Gamma$ and the extension  $R_\alpha \subseteq R$ is inert, there is a $u \in U(R)$ with $uy, u^{-1}r \in R_\alpha$. Because $r \notin U(R)$, we see that $u^{-1}r \notin U(R_\alpha)$. So $yR = yuR$ and $xR_\alpha = yrR_\alpha = yu (u^{-1}r)R_\alpha \subsetneq yuR_\alpha$. Hence for any ascending chain of principal ideals of $R$ starting at $xR$, we can construct an ascending chain of principal ideals of $R_\alpha$ starting at $xR_\alpha$ and having the same length. Since $R_\alpha$ is a BFD, it follows from Proposition~\ref{prop:BFD characterizations} that $R$ is also a BFD.
	\smallskip
	
	(2) Now suppose that $R_\gamma$ is an FFD for every $\gamma \in \Gamma$. Fix $x \in R$, and take $\alpha \in \Gamma$ such that $x \in R_\alpha$. By Proposition~\ref{prop:FFD characterizations}, there is a largest (finite) list $x_1, \dots, x_m$ of non-associate divisors of $x$ in $R_\alpha$. Let $y$ be a divisor of $x$ in $R$ and write $x = yr$ for some $r \in R$. Because the family $(R_\gamma)_{\gamma \in \Gamma}$ is directed, there is a $\beta \in \Gamma$ such that $R_\alpha \subseteq R_\beta$ and $y,r \in R_\beta$. Since $R_\alpha \subseteq R_\beta$ is an inert extension, $yu, ru^{-1} \in R_\alpha$ for some $u \in U(R_\beta)$. As $yu$ divides $x$ in $R_\alpha$, there exists $v \in U(R_\alpha) \subseteq U(R)$ such that $yu = x_j v$ for some $j \in \ldb 1,m \rdb$. Hence $y \in x_j U(R)$. Therefore every divisor of $x$ in $R$ must be associate to some of the elements $x_1, \dots, x_m$ in $R$. Thus, $R$ is an FFD by Proposition~\ref{prop:FFD characterizations}.
\end{proof}

\begin{remark}
	A similar version of Theorem~\ref{thm:BF/FF for directed unions} holds if one replaces being a BFD (or an FFD) by satisfying ACCP, being an HFD, or being a UFD; see \cite[Theorem~5.2]{AAZ92}.
\end{remark}

\smallskip
\subsection{Pullback Constructions}

We conclude this section by studying the bounded and finite factorization properties for integral domains given by certain pullbacks that generalize the $D+M$ construction. To formalize this, consider an integral domain $T$ with a nonzero maximal ideal $M$, and let $\varphi \colon T \to K$ be the natural projection on the residue field $K = T/M$. For a subring $D$ of $K$, we call $R = \varphi^{-1}(D)$ the \emph{pullback} of $D$ by~$\varphi$. Observe that the $D+M$ construction is a special case of a pullback: indeed, if $k$ is a subfield of $T$ such that $T = k+M$, then $K = T/M$ can be identified with~$k$ canonically, and so any subring $D$ of $k$ can be thought of as an actual subring of $K$.

When $T$ is quasilocal, the results that we have already established for the $D+M$ construction extend to pullbacks, as we will see in Propositions~\ref{prop:pullback quasilocal BF} and \ref{prop:pullback quasilocal FF}. First, we prove the following lemmas.

\begin{lemma} \emph(\cite[Lemma~6.1]{AeA99}\emph) \label{lem:units in extensions with a common prime ideal}
	Let $R \subseteq T$ be an extension of integral domains, and let $I$ be a nonzero ideal of both $R$ and $T$. If $R$ is atomic and $I$ is a prime ideal of $R$, then $U(R) = U(T) \cap R$.
\end{lemma}

\begin{proof}
	Since $U(R)$ is contained in $U(T) \cap R$, it suffices to show that every element of $U(T) \cap R$ is a unit of $R$. Take $x \in U(T) \cap R$. Since $I$ is a nonzero prime ideal of the atomic domain $R$, there must be an irreducible $a$ of $R$ contained in $I$. Because $I$ is also an ideal of $T$, it follows that $x^{-1}a \in I \subseteq R \setminus U(R)$, and so the equality $a = x(x^{-1}a)$ ensures that $x \in U(R)$. Hence $U(R) = U(T) \cap R$.
\end{proof}

\begin{lemma} \emph(\cite[Lemma~6.2]{AeA99}\emph) \label{lem:units in pullback}
	Let $T$ be an integral domain with a nonzero maximal ideal $M$, and let $\varphi \colon T \to K$ be the natural projection on $K = T/M$. In addition, let $D$ be a subring of $K$, and set $R = \varphi^{-1}(D)$. Then the following statements hold.
	\begin{enumerate}
		\item $U(R) = U(T) \cap \varphi^{-1}(U(D))$, and so $U(R) = U(T) \cap R$ when $D$ is a field.
		\smallskip
		
		\item If $T$ is quasilocal, then $U(R) = U(T) \cap R$ if and only if $D$ is a field.
	\end{enumerate}
\end{lemma}

\begin{proof}
	(1) It is clear that $U(R) \subseteq U(T) \cap \varphi^{-1}(U(D))$. In order to argue the reverse inclusion, take $x \in \varphi^{-1}(U(D))$ with $x^{-1} \in T$. Since $\varphi(x) \in U(D)$, it follows that $\varphi(x^{-1}) = \varphi(x)^{-1} \in U(D)$. As a result, $x^{-1} \in \varphi^{-1}(U(D)) \subseteq R$, and so $x \in U(R)$. Hence $U(T) \cap \varphi^{-1}(U(D)) \subseteq U(R)$. The second statement is an immediate consequence of the first.
	\smallskip
	
	(2) Proving this part amounts to noting that when $T$ is quasilocal, restricting $\varphi$ to $U(T)$ yields a surjective group homomorphism $U(T) \to K^*$.
\end{proof}

Note that in the pullback construction, $R$ is quasilocal with maximal ideal $M$ if and only if $D$ is a field.

\begin{lemma} \label{lem:atoms in pullback}
	Let $T$ be a quasilocal integral domain with nonzero maximal ideal $M$, and let $\varphi \colon T \to K$ be the natural projection on $K = T/M$. In addition, let $D$ be a subring of $K$, and set $R = \varphi^{-1}(D)$. If $D$ is a field, then $\ii(R) = \ii(T) \subseteq M$.
\end{lemma}

\begin{proof}
	To argue that $\ii(R) \subseteq \ii(T)$, take $m \in \ii(R)$, and suppose that $m = xy$ for some $x,y \in T$. Since $M \subseteq R$ and $m \in \ii(R)$, the elements $x$ and $y$ cannot be contained in $M$ simultaneously. Therefore either $x \in T \setminus M = U(T)$ or $y \in T \setminus M = U(T)$, and so $m \in \ii(T)$. To argue the reverse inclusion, take $m \in \ii(T)$, and suppose that $m = xy$ for some $x,y \in R$. Since $m \in \ii(T)$, either $x \in U(T)$ or $y \in U(T)$. Because $T$ is quasilocal and $D$ is a field, it follows from part~(2) of Lemma~\ref{lem:units in pullback} that $U(R) = U(T) \cap R$. Therefore $x \in U(R)$ or $y \in U(R)$, and so $m \in \ii(R)$. Thus, $\ii(R) = \ii(T)$. Finally, the fact that $T$ is quasilocal ensures that $\ii(T) \subseteq M$.
\end{proof}

\begin{prop}  \emph(\cite[Proposition~6.3]{AeA99}\emph) \label{prop:pullback quasilocal BF}
	Let $T$ be a quasilocal integral domain with nonzero maximal ideal $M$, and let $\varphi \colon T \to K$ be the natural projection on $K = T/M$. In addition, let $D$ be a subring of $K$, and set $R = \varphi^{-1}(D)$. Then $R$ is a BFD if and only if $T$ is a BFD and $D$ is a field.
\end{prop}

\begin{proof}
	Suppose first that $R$ is an atomic domain. Since $M$ is a maximal ideal of $T$ contained in~$R$, it follows that $M$ is a nonzero prime ideal of $R$. This, along with Lemma~\ref{lem:units in extensions with a common prime ideal}, guarantees that $U(R) = U(T) \cap R$. Because $T$ is quasilocal, $D$ is a field by part~(2) of Lemma~\ref{lem:units in pullback}, and so it follows from Lemma~\ref{lem:atoms in pullback} that $\ii(R) = \ii(T)$. Therefore every element in $M$ factors into irreducibles in~$R$ if and only if it factors into irreducibles in $T$. Hence $T$ must be atomic. On the other hand, assume that $T$ is atomic and $D$ is a field. As $D$ is a field, $U(R) = U(T) \cap R$ by part~(1) of Lemma~\ref{lem:units in pullback}, which  implies that $U(R) = R \setminus M$. Therefore every nonzero nonunit of $R$ can be written as a product of elements in $\ii(T)$ because $T$ is atomic. As $\ii(T) = \ii(R)$ by Lemma~\ref{lem:atoms in pullback}, the atomicity of $R$ follows.
	
	Assuming that $D$ is a field, it is not hard to verify that for every nonzero $x,y \in M$ the inclusion $xT \subsetneq yT$ holds if and only if the inclusion $xR \subsetneq yR$ holds. As a result, it follows from Proposition~\ref{prop:BFD characterizations} that $R$ is a BFD if and only if $T$ is a BFD.
\end{proof}

Parallel to Proposition~\ref{prop:pullback quasilocal BF}, we proceed to give a result for the finite factorization property in pullback constructions.

\begin{prop}  \emph(\cite[Propositions~6.3 and~6.7]{AeA99}\emph) \label{prop:pullback quasilocal FF}
	Let $T$ be an integral domain with a nonzero maximal ideal $M$, and let $\varphi \colon T \to K$ be the natural projection on $K = T/M$. In addition, let $D$ be a subring of $K$, and set $R = \varphi^{-1}(D)$. Then the following statements hold.
	\begin{enumerate}
		\item $R$ is an FFD if and only if $T$ is an FFD and the group $U(T)/U(R)$ is finite.
		\smallskip
		
		\item If $T$ is quasilocal, then $R$ is an FFD if and only if $T$ is an FFD, $D$ is a field, and the group $K^*/D^*$ is finite.
	\end{enumerate}
\end{prop}

\begin{proof}
	(1) For the direct implication, suppose that $R$ is an FFD. Since $M$ is a nonzero ideal of $T$ that is contained in $R$, the nonempty set $M \setminus \{0\}$ is contained in $[R :_R T] = \{r \in R \mid rT \subseteq R\}$. As a result, it follows from Proposition~\ref{prop:FFD to ring extension} that $T$ is an FFD and $U(T)/U(R)$ is finite.
	
	Conversely, suppose that $T$ is an FFD and the group $U(T)/U(R)$ is finite. Let $F$ be the quotient field of $R$ inside $\text{qf}(T)$. Since $(U(T) \cap F^*)/U(R)$ is a subgroup of $U(T)/U(R)$, the former must be finite. Thus, $R$ is an FFD by Proposition~\ref{prop:FFD underrings}.
	\smallskip
	
	(2) Suppose that $R$ is an FFD. It follows from the previous part that $T$ is an FFD. In addition, it follows from Proposition~\ref{prop:pullback quasilocal BF} that $D$ is a field, and so $U(R) = U(T) \cap R$ by Lemma~\ref{lem:units in pullback}. Because~$T$ is quasilocal, the map $\varphi \colon U(T) \to K^*$ obtained by restricting $\varphi$ to $U(T)$ is a surjective group homomorphism. By composing this map with the natural projection $K^* \to K^*/D^*$, we obtain a surjective group homomorphism $U(T) \to K^*/D^*$, whose kernel is $U(R)$ because $U(R) = U(T) \cap R$. Hence $U(T)/U(R) \cong K^*/D^*$, and so the previous part ensures that $K^*/D^*$ is finite.
	
	For the reverse implication, assume that $T$ is an FFD, $D$ is a field, and $K^*/D^*$ is finite. As in the previous part, $U(T)/U(R) \cong K^*/D^*$. Hence $U(T)/U(R)$ is finite, and it also follows from the previous part that $R$ is an FFD.
\end{proof}

The condition of $T$ being quasilocal in Proposition~\ref{prop:pullback quasilocal BF} and in part~(2) of Proposition~\ref{prop:pullback quasilocal FF} is not superfluous, as we show in the following example, which is part of \cite[Example~6.6]{AeA99}.

\begin{example}
	(1) Set $T = \qq[\pi] + X \rr[X]$ and consider the ring homomorphism $\varphi \colon T \to \cc$ defined by $\varphi(f) = f(i)$. Since $\varphi$ is surjective, $\ker \varphi$ is a nonzero maximal ideal of $T$, and we can think of $\varphi$ as the natural projection $T \to T/M$, where $M = \ker \varphi$. Take $D = \qq$ and $R = \varphi^{-1}(D)$. Because $\rr[X]$ is a BFD and $U(\rr[X]) \cap R = \qq \setminus \{0\} = U(R)$, it follows from Proposition~\ref{prop:BFD underrings} that $R$ is a BFD. However, $\qq[\pi]$ is not a field. In addition, the fact that $\qq[\pi]$ is not a field, along with Remark~\ref{rem:atomicity/ACCP in D+M construction}, ensures that~$T$ is not even atomic. In particular, $T$ is not a BFD.
	\smallskip
	
	(2) Let $D$ be a subring of a field $K$. Consider a family of indeterminates indexed by~$K$, namely, $\{X_k \mid k \in K\}$, and set $T = \zz[\{X_k \mid k \in K\}]$. Now let $\varphi \colon T \to K$ be the ring homomorphism determined by the assignments $X_k \mapsto k$ for every $k \in K$. As $\varphi$ is surjective, we can assume that $\varphi$ is the natural projection $T \to T/M$, where $M = \ker \varphi$. Now take any subring $D$ and set $R = \varphi^{-1}(D)$. Because $T$ is a UFD and $U(R) = U(T) = \{\pm 1\}$, it follows from Proposition~\ref{prop:FFD underrings} that $R$ is an FFD regardless of our choice of $D$.
\end{example}

\bigskip
\section{Polynomial-Like Rings}
\label{sec:polynomial-like rings}

In this section, we study conditions under which the bounded and finite factorization properties transfer between an integral domain and its ``polynomial-like rings". We put special emphasis on integral domains of the form $A + XB[X]$ and $A + XB[[X]]$, where $A \subseteq B$ is an extension of integral domains, and the generalized case obtained by replacing the single extension $A \subseteq B$ by the possibly-infinite tower of integral domains $A_1 \subseteq A_2 \subseteq \cdots$.

\smallskip
\subsection{Bounded Factorization Subdomains of $R[X]$ and $R[[X]]$}

Let $A \subseteq B$ be an extension of integral domains. As in~\cite{AeA99}, we say that $B$ is a \emph{bounded factorization domain with respect to} $A$ or an $A$-\emph{BFD} if for every nonzero nonunit $x \in B$, there is an $n_0 \in \nn$ such that if $x = b_1 \cdots b_n$ for some nonunits $b_1, \dots, b_n \in B$, then at most $n_0$ of the $b_i$'s belong to $A$.

\begin{theorem} \label{thm:BF in polynomial and power series rings} \emph(\cite[Proposition 2.1]{AeA99}\emph)
	Let $A \subseteq B$ be an extension of integral domains. Then the following statements are equivalent.
	\begin{enumerate}
		\item[(a)] $A + XB[X]$ is a BFD.
		\smallskip
		
		\item[(b)] $A + XB[[X]]$ is a BFD.
		\smallskip
		
		\item[(c)] $B$ is an $A$-BFD and $U(A) = U(B) \cap A$.
	\end{enumerate}
	In addition, if $B$ is a BFD, then (c) can be replaced by the statement \\
		\vspace{4pt}
		\quad \emph{(c$^\prime$)} $U(A) = U(B) \cap A$, \\
	and if $\emph{qf}(A) \subseteq B$, then (c) can be replaced by the statement \\ 
		\vspace{4pt}
		\quad \emph{(c$^{\prime \prime}$)} $A$ is a field.
\end{theorem}

\begin{proof}
	Set $R = A + XB[X]$ and $T = A + XB[[X]]$.
	
	(a) $\Rightarrow$ (b): By Proposition~\ref{prop:BFD characterizations}, there is a length function $\ell_R \colon R^* \to \nn_0$. Now define the function $\ell_T \colon T^* \to \nn_0$ by $\ell_T\big(\sum_{i=n}^\infty a_iX^i \big) = \ell_R(a_nX^n) + n$ for every $\sum_{i=n}^\infty a_iX^i$ with $a_n \neq 0$. Clearly, $\ell_T\big(\sum_{i=n}^\infty a_iX^i \big) = 0$ if and only if $n=0$ and $a_0 \in U(A)$. In addition, for all $f = \sum_{i=n}^\infty a_iX^i$ and $g = \sum_{i=m}^\infty b_iX^i$ in $T^*$ with $a_n \neq 0$ and $b_m \neq 0$, the fact that $\ell_R$ is a length function guarantees that $\ell_T(fg) = \ell_R(a_n b_m X^{n+m}) + n+m \ge \ell_R(a_nX^n) + \ell_R(b_mX^m) + n +m = \ell_T(f) + \ell_T(g)$. Hence~$\ell_T$ is a length function, and so $T$ is a BFD by Proposition~\ref{prop:BFD characterizations}.
	\smallskip
	
	(b) $\Rightarrow$ (c): It is clear that $U(A) \subseteq U(B) \cap A$. For the reverse inclusion, take $u \in A$ such that $u^{-1} \in B$. Since $T$ is a BFD, it satisfies ACCP, and therefore, the ascending chain of principal ideals $(u^{-n}XT)_{n \in \nn}$ must stabilize. Then $u^{-n}XT = u^{-(n+1)}XT$ for some $n \in \nn$, from which we obtain $u \in U(T) \cap A = U(A)$. Thus, $U(A) = U(B) \cap A$. To show that $B$ is an $A$-BFD, let $b$ be a nonzero nonunit of~$B$. Since $T$ is a BFD, there is an $n_0 \in \nn$ such that $bX$ cannot be the product of more than~$n_0$ nonunits in $T$. Write $b = a_1 \cdots a_m b_1 \cdots b_n$, where $a_1, \dots, a_m$ are nonunits of $A$ and $b_1, \dots, b_n$ are nonunits in $B \setminus A$. Clearly, $a_1, \dots, a_m$ are nonunits in $T$. Then $bX = a_1 \cdots a_m(b_1 \cdots b_nX)$, and so $m \le n_0-1$. Hence $B$ is an $A$-BFD.
	\smallskip
	
	(c) $\Rightarrow$ (a): Assume now that $B$ is an $A$-BFD satisfying $U(A) = U(B) \cap A$. It immediately follows that $A$ is a BFD. Take $f = \sum_{i=0}^n b_i X^i$ with $b_n \neq 0$ to be a nonzero nonunit of $R$. If $n = 0$, then $f = b_0 \in A$ and so there is an $n_0 \in \nn$ such that $f$ cannot be the product of more than $n_0$ nonunits of~$R$. On the other hand, suppose that $n \ge 1$. As $B$ is an $A$-BFD, there is an upper bound $n_1 \in \nn$ for the number of nonunit factors in $A$ of a factorization of $b_n$ in $B$. Then a factorization of $f$ in $R$ has at most $n_1 + n$ nonunit factors. Thus, $R$ is a BFD.
	\smallskip
	
	(c) $\Leftrightarrow$ (c$^\prime$) when $B$ is a BFD: This is clear as $B$ is also an $A$-BFD.
	\smallskip
	
	(c) $\Leftrightarrow$ (c$^{\prime \prime}$) when $\text{qf}(A) \subseteq B$: For the direct implication, it suffices to note that $\text{qf}(A)^* \subseteq U(B)$ implies that $A^* \subseteq U(B) \cap A = U(A)$. The reverse implication follows from the fact that every extension of the field $A$ is an $A$-BFD.
\end{proof}

\begin{cor} \emph(\cite[Proposition~2.5]{AAZ90}, \cite[Corollary~2.2]{AAZ92}, and \cite[Corollary~3.1]{hK01}\emph) \label{cor:BFD for polynomial and power series rings}
	The following statements are equivalent for an integral domain $R$.
	\begin{enumerate}
		\item[(a)] $R$ is a BFD.
		\smallskip
		
		\item[(b)] $R[X]$ is a BFD.
		\smallskip
		
		\item[(c)] $R[[X]]$ is a BFD.
		\smallskip
		
		\item[(d)] $R[X,X^{-1}]$ is a BFD.
		\smallskip
		
		\item[(e)] The ring of formal Laurent series $R((X))$ is a BFD.
	\end{enumerate}
\end{cor}

\begin{proof}
	(a) $\Leftrightarrow$ (b) $\Leftrightarrow$ (c): These equivalences follow by taking $B=A=R$ in Theorem~\ref{thm:BF in polynomial and power series rings}.
	\smallskip
	
	(b) $\Leftrightarrow$ (d): Observe that the ring of Laurent polynomials $R[X,X^{-1}]$ is the localization of $R[X]$ at the multiplicative set $S = \{uX^n \mid u \in U(R) \text{ and } n \in \nn_0\}$ generated by the prime $X$. Then Lemma~\ref{lem:when multiplicative sets generated by primes are SMS} guarantees that $S$ is a splitting multiplicative set, while Lemma~\ref{lem:localization at prime-generated MS are inert} guarantees that the extension $R[X] \subseteq R[X]_S = R[X,X^{-1}]$ is inert. As a consequence, it follows from Theorems~\ref{thm:underring localization BFD/FFD} and~\ref{thm:overring localization BFD/FFD} that $R[X]$ is a BFD if and only if $R[X,X^{-1}]$ is a BFD.
	\smallskip
	
	(c) $\Leftrightarrow$ (e): After observing that $R((X)) = R[[X]]_S$, where $S = \{uX^n \mid u \in U(R) \text{ and } n \in \nn_0\}$, we can simply repeat the argument given in the previous paragraph.
\end{proof}

\begin{cor} \emph(\cite[Proposition~2.6]{AAZ90}\emph)\label{cor:subrings of polynomial rings are BFD}
	Let $R$ be a BFD, and let $\{X_i \mid i \in I\}$ be a family of indeterminates for some set $I$. Then every subring of $R[\{X_i \mid i \in I\}]$ containing $R$ is a BFD.
\end{cor}

\begin{proof}
	Set $R_I = R[\{X_i \mid i \in I\}]$, and let $T$ be a subring of $R_I$ containing $R$. Take $f$ to be a nonunit of $R_I$, and then take a finite subset $J$ of $I$ such that $f \in R_J = R[\{X_j \mid j \in J\}]$. As $R$ is a BFD and $|J| < \infty$, it follows from Corollary~\ref{cor:BFD for polynomial and power series rings} that $R_J$ is a BFD. The equality $Z_{R_I}(f) = Z_{R_J}(f)$ ensures that $|L_{R_I}(f)| = |L_{R_J}(f)| < \infty$. Hence $R_I$ is a BFD. Since $R \subseteq T \subseteq R_I$, it follows that $U(T) = U(R) = U(R_I)$. Thus, Proposition~\ref{prop:BFD underrings} guarantees that $T$ is a BFD.
\end{proof}

With the notation as in Theorem~\ref{thm:BF in polynomial and power series rings}, the integral domain $A$ is a BFD if $A + XB[X]$ is a BFD. However, the converse of this implication does not hold in general.

\begin{example}
	Consider the integral domain $R = \zz + X \qq[X]$. Clearly, $\zz$ is a BFD. Observe, on the other hand, that $R$ is a particular case of the $D+M$ construction, where $D = \zz$ is not a field. Thus,~$R$ is not a BFD by Proposition~\ref{prop:BFD and D+M construction}.
\end{example}

We would also like to emphasize that even if $A + XB[X]$ and $A + XB[[X]]$ are both BFDs, $B$ may not be a BFD. The following example is \cite[Example~2.7]{AAZ90}.

\begin{example}
	Let $\bar{\zz}$ be the ring of algebraic integers. Since the ascending chain of principal ideals $(2^{1/2^n}\bar{\zz})_{n \in \nn}$ does not stabilize, $\bar{\zz}$ does not satisfy ACCP, and so it is not a BFD. However, the integral domain $R = \zz + X \bar{\zz}[X]$ is a BFD. To verify this, let $\ell \colon \zz^* \to \nn_0$ be a length function, and define $\ell_R \colon R^* \to \nn_0$ by $\ell_R(f) = \ell(f(0)) + \deg f$. It is clear that $\ell_R(f) = 0$ if and only if $\ell(f(0)) = \deg f = 0$, which happens precisely when $f \in \{\pm 1\} = U(R)$. In addition, it is clear that $\ell_R(fg) \ge \ell_R(f) + \ell_R(g)$ when $f,g \in R^*$. Hence $\ell_R$ is a length function, and so Proposition~\ref{prop:BFD characterizations} guarantees that $R$ is a BFD. It follows from Theorem~\ref{thm:BF in polynomial and power series rings} that $\zz + X \bar{\zz}[[X]]$ is also a BFD.
\end{example}

With the notation as in Theorem~\ref{thm:BF in polynomial and power series rings}, if $B$ is taken to be the quotient field of $A$, then the property of being a BFD transfers from $A$ to any intermediate integral domain of the extension $A[X] \subseteq A + XB[X]$ if we impose a certain condition.

\begin{prop}  \emph(\cite[Theorem~7.5]{AAZ91}\emph)
	Let $R$ be an integral domain with quotient field $K$, and let $T$ be an integral domain such that $R[X] \subseteq T \subseteq R + XK[X]$. In addition, assume that for every $n \in \nn_0$, there is an $r_n \in R^*$ such that $r_n f \in R[X]$ for every $f \in T$ with $\deg f \le n$. Then $T$ is a BFD if and only if $R$ is a BFD.
\end{prop}

\begin{proof}
	For the direct implication, suppose that $T$ is a BFD. It is clear that $U(R) = U(T)$. Therefore~$R$ is a BFD by Proposition~\ref{prop:BFD underrings}.
	
	For the reverse implication, suppose that $R$ is a BFD. Take $f \in T$, and write $f = c_1 \cdots c_k g$, where $c_1, \dots, c_k \in R$ and $\deg g \ge 1$. Let $cX^n$ be the leading term of $f$. Then $r_n g \in R[X]$, and so the leading coefficient $r_n c$ of $r_n f = c_1 \cdots c_k (r_n g) \in R[X]$ must belong to $R$. As $r_n g \in R[X]$, the product $c_1 \cdots c_k$ divides $r_nc$ in $R$. Thus, $k \le \max L_R(r_nc)$, and so $L_T(f)$ is bounded by $n + \max L_{R}(r_nc)$. Hence $T$ is a BFD.
\end{proof}

\begin{cor} \emph(\cite[Corollary~7.6]{AAZ91}\emph)
	Let $R$ be an integral domain with quotient field $K$. Then the ring $I(K,R)$ of $R$-valued polynomials of $K[X]$ is a BFD if and only if $R$ is a BFD.
\end{cor}

The following example is \cite[Example~2.7(b)]{AAZ90}.

\begin{example}
	Since $\zz$ is a BFD, the ring $R$ of integer-valued polynomials of $\qq[X]$, often denoted by $\text{Int}(\zz)$, is a BFD. In fact, $R$ is a two-dimensional completely integrally closed Pr\"ufer domain that satisfies ACCP. In addition, if $M$ is a height-two maximal ideal of $R$, then the localization $R_M$ is a two-dimensional valuation domain that is not even atomic.
\end{example}
\smallskip

Next we turn our attention to certain integral domains that generalize subdomains of the form $A + B[X]$ and $A + B[[X]]$. For the rest of this subsection, we let $(A_n)_{n \ge 0}$ be an ascending chain of integral domains contained in a field $L$ (that is, $A_n$ is a subdomain of $A_{n+1}$ for every $n \in \nn_0$), and we set $A = \bigcup_{n \in \nn_0} A_n$. Observe that $A$ is a subring of $L$. In addition, we set
\begin{equation} \label{eq:tower domain rings}
	\aaa[X] = \bigoplus_{n \in \nn_0} A_n X^n \quad \text{ and } \quad  \aaa[[X]] = \prod_{n \in \nn_0} A_n X^n.
\end{equation}
It is clear that $\aaa[X]$ and $\aaa[[X]]$ are subrings of $A[X]$ and $A[[X]]$, respectively. Parallel to Theorem~\ref{thm:BF in polynomial and power series rings}, we will give a necessary and sufficient condition for the integral domains $\aaa[X]$ and $\aaa[[X]]$ to be BFDs. The results about $\aaa[X]$ and $\aaa[[X]]$ we have included in this section are from the unpublished Ph.D. dissertation of P.~L. Kiihne~\cite{pK99}.

Before proceeding, we emphasize that even if $A_n$ is a BFD for every $n \in \nn_0$, the integral domain~$A$ may not be a BFD; for this, see Example~\ref{ex:canonical conductive Puiseux algebra} below.

\begin{theorem}  \emph(\cite[Theorem~3.3.5]{pK99}\emph)\label{thm:BFD theorem for infinite tower of domains}
	The following statements are equivalent.
	\begin{enumerate}
		\item[(a)] $\aaa[X]$ is a BFD.
		\smallskip
		
		\item[(b)] $\aaa[[X]]$ is a BFD.
		\smallskip
		
		\item[(c)] $U(A_0) = U(A) \cap A_0$, and $A_n$ is an $A_0$-BFD for every $n \in \nn_0$.
	\end{enumerate}
	In addition, if $A[X]$ (or, equivalently, $A[[X]]$) is a BFD, then $\aaa[X]$ and $\aaa[[X]]$ are BFDs.
\end{theorem}

\begin{proof}
	(a) $\Rightarrow$ (c): Suppose that $\aaa[X]$ is a BFD. It is clear that $U(A_0) \subseteq U(A) \cap A_0$. For the reverse inclusion, take $u \in A_0$ such that $u^{-1} \in A$. Take $m \in \nn_0$ such that $u^{-1} \in A_m$. As $\aaa[X]$ satisfies ACCP by Corollary~\ref{cor:BFMs are ACCP monoids}, the ascending chain of principal ideals $(u^{-n}X^m \aaa[X])_{n \in \nn}$ of $\aaa[X]$ must stabilize. As a result, there is an $n \in \nn$ such that $u^{-n}X^m \aaa[X] = u^{-(n+1)}X^m \aaa[X]$, from which we deduce that $u \in U(\aaa[X]) \cap A_0 = U(A_0)$. Hence $U(A) \cap A_0 \subseteq U(A_0)$.
	
	To prove the second statement, fix $k \in \nn_0$, and then take a nonunit $b \in A_k^*$. Since $\aaa[X]$ is a BFD, there is an $n_0 \in \nn$ such that $bX^k$ cannot be the product of more than~$n_0$ nonunits in $\aaa[X]$. Write $b = a_1 \cdots a_m b_1 \cdots b_n$, where $a_1, \dots, a_m$ are nonunits of $A_0$ and $b_1, \dots, b_n$ are nonunits in $A_k \setminus A_0$. Then $bX^k = a_1 \cdots a_m(b_1 \cdots b_nX^k)$, and since $a_1, \dots, a_m$ are nonunits in $\aaa[X]$, the inequality $m \le n_0 - 1$ holds. Hence $A_k$ is an $A_0$-BFD.
	\smallskip
	
	(c) $\Rightarrow$ (a): Assume that $U(A_0) = U(A) \cap A_0$ and $A_n$ is an $A_0$-BFD for every $n \in \nn_0$. Since $U(A_0) = U(A_d) \cap A_0$, and $A_d$ is an $A_0$-BFD, Theorem~\ref{thm:BF in polynomial and power series rings} guarantees that $A_0 + XA_d[X]$ is a BFD. Since $f \in A_0 + XA_d[X]$ and every pair of non-associate divisors of $f$ in $\aaa[X]$ is also a pair of non-associate divisors of $f$ in $A_0 + XA_d[X]$, the fact that $A_0 + XA_d[X]$ is a BFD implies that $L_{\aaa[X]}(f)$ is finite. Hence $\aaa[X]$ is a BFD.
	\smallskip
	
	(b) $\Rightarrow$ (c): It follows mimicking the argument we use to prove (a) $\Rightarrow$ (c).
	\smallskip
	
	(c) $\Rightarrow$ (b): Assume that $U(A_0) = U(A) \cap A_0$ and $A_n$ is an $A_0$-BFD for every $n \in \nn_0$. Let $f = \sum_{i=m}^\infty b_iX^i \in \aaa[[X]]^*$ be a nonunit, and assume that $b_m \neq 0$. Since $A_m$ is an $A_0$-BFD, there is an $n_0 \in \nn$ such that any factorization of $b_m$ in $A_m$ involves at most $n_0$ factors in $A_0$. Now suppose that $f = f_1 \cdots f_k g_1 \cdots g_\ell$ in $\aaa[[X]]$, where $f_1, \dots, f_k$ are nonunits with order $0$ and $g_1, \dots, g_\ell$ are nonunits of order at least~$1$. It is clear that $\ell \le m$. On the other hand, comparing the coefficients of the degree $m$ monomials in both sides of the equality $f = f_1 \cdots f_k g_1 \cdots g_\ell$, we see that $b_m = c_1 \cdots c_k c$ in $A_m$, where $c_1, \dots, c_k$ are nonunits in $A_0$. Therefore $k \le n_0$, and so $\max L_{\aaa[[X]]}(f) \le m+n_0$. We conclude that $\aaa[[X]]$ is a BFD.
\end{proof}

\begin{cor} \emph(\cite[Corollary~3.3.9]{pK99}\emph) \label{cor:A_0 field implies aaa[X] and aaa[[X]] BFDs}
	If $A_0$ is a field, then $\aaa[X]$ and $\aaa[[X]]$ are BFDs.
\end{cor}

\begin{proof}
	It is an immediate consequence of Theorem~\ref{thm:BFD theorem for infinite tower of domains} since when $A_0$ is a field both statements of part~(c) of Theorem~\ref{thm:BFD theorem for infinite tower of domains} trivially hold.
\end{proof}

In the spirit of Theorem~\ref{thm:BFD theorem for infinite tower of domains}, observe that if there is an $N \in \nn_0$ such that $A_n = A_N$ for every $n \ge N$, then $\aaa[X]$ is a BFD (resp., $\aaa[[X]]$ is a BFD) if and only if $A_N$ is an $A_0$-BFD and $U(A_N) \cap A_0 = U(A_0)$. On the other hand, the reverse implication of the last statement of Theorem~\ref{thm:BFD theorem for infinite tower of domains} does not hold, as we proceed to illustrate using \cite[Example~3.3.12]{pK99}.

\begin{example} \label{ex:canonical conductive Puiseux algebra}
	Let $F$ be a field, and for every $n \in \nn$, let $M_n$ be the additive monoid $\{0\} \cup \qq_{\ge 1/n}$. Now set $A_0 = F$ and $A_n = F[M_n]$ for every $n \in \nn$. We have seen in Example~\ref{ex:BFD that is neither an HFD nor an FFD} that the integral domain $A_n$ is a BFD for every $n \in \nn$. However, $A = \bigcup_{n \in \nn_0} A_n = F[\qq_{\ge 0}]$ is not a BFD because it is not even atomic. As $A$ is not a BFD, it follows from Corollary~\ref{cor:BFD for polynomial and power series rings} that neither $A[X]$ nor $A[[X]]$ are BFDs. On the other hand, since $A_0$ is a field, both $\aaa[X]$ and $\aaa[[X]]$ are BFDs by Corollary~\ref{cor:A_0 field implies aaa[X] and aaa[[X]] BFDs}.
\end{example}

\smallskip
\subsection{Finite Factorization Subdomains of $R[X]$ and $R[[X]]$}

The main purpose of this subsection is to characterize when the integral domains $A + XB[X]$ and $A + XB[[X]]$ are FFDs for a given extension of integral domains $A \subseteq B$. 

Unfortunately, the equivalences in Theorem~\ref{thm:BF in polynomial and power series rings} do not hold if we replace BFD by FFD. However, some of the equivalences in Corollary~\ref{cor:BFD for polynomial and power series rings} are still true for the finite factorization property.

\begin{theorem}  \emph(\cite[Proposition~5.3]{AAZ90} and \cite[Corollary~4.2]{hK01}\emph) \label{thm:FFD for polynomial and Laurent ring}
	The following statements are equivalent for an integral domain $R$.
	\begin{enumerate}
		\item[(a)] $R$ is an FFD.
		\smallskip
		
		\item[(b)] $R[X]$ is an FFD.
		\smallskip
		
		\item[(c)] $R[X,X^{-1}]$ is an FFD.
	\end{enumerate}
\end{theorem}

\begin{proof}
	(a) $\Rightarrow$ (b):  Assume that $R$ is an FFD, and let $K$ be the quotient field of $R$. Suppose, by way of contradiction, that $R[X]$ is not an FFD. It follows from Proposition~\ref{prop:FFD characterizations} that there is a nonzero nonunit $f \in R[X]$ having infinitely many non-associate divisors in $R[X]$. Take $(f_n)_{n \in \nn}$ to be a sequence of non-associate divisors of $f$ in $R[X]$. Let $c$ be the leading coefficient of $f$, and let $c_n \in R$ be the leading coefficient of $f_n$ for every $n \in \nn$. Since $c_n$ is a divisor of $c$ for every $n \in \nn$, and $R$ is an FFD, after replacing $(f_n)_{n \in \nn}$ by a subsequence, we can assume that $c_1$ and $c_n$ are associates in $R$ for every $n \in \nn$. In addition, after replacing $f_n$ by $c_1 c_n^{-1} f_n$ for every $n \in \nn_{\ge 2}$, we can assume that all polynomials in the sequence $(f_n)_{n \in \nn}$ have the same leading coefficient, namely, $c_1$. Since each $f_n$ divides $f$ in $K[X]$, which is an FFD, there are distinct $i,j \in \nn$ such that $f_i$ and $f_j$ are associates in $K[X]$. As $f_i$ and $f_j$ have the same leading coefficient, they must be equal, which contradicts that they are non-associates in $R[X]$.
	\smallskip
	
	(b) $\Rightarrow$ (c) Suppose that $R[X]$ is an FFD. Since the extension $R[X] \subseteq R[X]_S = R[X,X^{-1}]$ is inert for the multiplicative set $S = \{uX^n \mid u \in U(R) \text{ and } n \in \nn_0\}$ (see the proof of Corollary~\ref{cor:BFD for polynomial and power series rings}), it follows from Theorem~\ref{thm:overring localization BFD/FFD} 
	that $R[X,X^{-1}]$ is an FFD.
	\smallskip
	
	(c) $\Rightarrow$ (a) Suppose that $R[X,X^{-1}]$ is an FFD, and let $K$ be the quotient field of $R$. Because $U(R[X,X^{-1}]) = \{uX^n \mid u \in U(R) \text{ and } n \in \zz\}$, the group $(U(R[X,X^{-1}]) \cap K^*)/U(R)$ is trivial. As a result, it follows from Proposition~\ref{prop:FFD underrings} that $R$ is an FFD.
\end{proof}

\begin{cor} \emph(\cite[Examples~2, 4, and~7]{AM96}\emph) \label{cor:when polynomials in a collection of indeterminates are FFDs}
	For an FFD $R$ and a set $\{X_i \mid i \in I\}$ of indeterminates, the following statements hold.
	\begin{enumerate}
		\item $R[\{X_i \mid i \in I\}]$ is an FFD.
		\smallskip
		
		\item $R[\{X_i, X_i^{-1} \mid i \in I\}]$ is an FFD.
		\smallskip
		
		\item If $R$ is either a finite field or $\zz$, then every subring of $R[\{X_i \mid i \in I\}]$ is an SFFD, and hence an FFD.
	\end{enumerate}
\end{cor}

\begin{proof}
	(1) Set $R_I = R[\{X_i \mid i \in I\}]$. Take a nonunit $f$ in $R_I^*$, and let $J$ be a finite subset of $I$ such that $f \in R_J = R[\{X_j \mid j \in J\}]$. The integral domain $R_J$ is an FFD by Theorem~\ref{thm:FFD for polynomial and Laurent ring}. In addition, every divisor of $f$ in $R_I$ is also a divisor of $f$ in $R_J$. As $U(R_J) = U(R) = U(R_I)$, the fact that $f$ has only finitely many non-associate divisors in $R_J$ implies that $f$ has only finitely many non-associate divisors in $R_I$. Hence $R_I$ is an FFD.
	\smallskip
	
	(2) The integral domain $R[\{X_i, X_i^{-1} \mid i \in I \}]$ is the localization of $R[\{X_i \mid i \in I\}]$ at the multiplicative set $S$ generated by $\{X_i \mid i \in I\}$. Since $X_i$ is a prime element in $R[\{X_i \mid i \in I\}]$ for every $i \in I$, it follows from Lemma~\ref{lem:localization at prime-generated MS are inert} that the extension $R[\{X_i \mid i \in I\}] \subseteq R[\{X_i, X_i^{-1} \mid i \in I\}]$ is inert. As $R[\{X_i \mid i \in I\}]$ is an FFD by part~(1), part~(2) of Theorem~\ref{thm:overring localization BFD/FFD} guarantees that $R[\{X_i, X_i^{-1} \mid i \in I\}]$ is an FFD.
	\smallskip
	
	(3) Let $R$ be either a finite field or $\zz$. As in part~(1), set $R_I = R[\{X_i \mid i \in I\}]$. Let $T$ be a subring of $R_I$, and then let $f$ be a nonunit in $T^*$. Take a finite subset $J$ of $I$ such that $f$ belongs to $R_J = R[\{X_j \mid j \in J\}]$. By part~(1), $R_J$ is an FFD. Moreover, since $|U(R_J)| = |U(R)| < \infty$, it follows from Proposition~\ref{prop:SFFD characterizations} that $R_J$ is an SFFD. As every divisor of $f$ in $T$ is also a divisor of $f$ in $R_J$, the element $f$ has only finitely many divisors in $T$. Thus, $T$ is an SFFD, and therefore, an FFD.
\end{proof}

In contrast to Theorem~\ref{thm:FFD for polynomial and Laurent ring}, the ring of power series $R[[X]]$ need not be an FFD when $R$ is an FFD. To illustrate this observation with an example, we need the following proposition.

\begin{prop}  \emph(\cite[Corollary~2]{AM96}\emph) \label{prop:if R[[X]] is FFD, R is completely integrally closed}
	Let $R$ be an integral domain. If $R[[X]]$ is an FFD, then $R$ is completely integrally closed. Therefore when $R$ is Noetherian, $R[[X]]$ is an FFD if and only if $R$ is integrally closed.
\end{prop}

\begin{proof}
	Suppose towards a contradiction that $R$ is not completely integrally closed, and then take an almost integral element $t \in \text{qf}(R) \setminus R$ over $R$. As a result, the ideal $[R :_R R[t]]$ is nonzero, and therefore, $\big[ R[[X]] :_{R[[X]]} R[t][[X]] \big] \neq \{0\}$. Since $R[[X]]$ is an FFD, it follows from Proposition~\ref{prop:FFD to ring extension} that the group $U(R[t][[X]])/U(R[[X]])$ is finite. Hence we can choose $m,n \in \nn$ so that $m \neq n$ and $(1 + tX^m) U(R[[X]]) = (1 + tX^n) U(R[[X]])$, which implies that $(1 + tX^m)(1 + tX^n)^{-1} \in R[[X]]$. However, this contradicts that $(1 + tX^m)(1 + tX^n)^{-1} = 1 -tX^m + \cdots$ and $-t \notin R$. Thus, $R$ is completely integrally closed.
	
	The direct implication of the second statement follows directly from the first statement because every completely integrally closed domain is integrally closed. The reverse implication is also immediate because every Noetherian integrally closed domain is a Krull domain, which is an FFD by Theorem~\ref{thm:Krull domains are FFDs}.
\end{proof}

We are now in a position to illustrate that $R[[X]]$ may not be an FFD even if $R$ is an FFD. The following example is \cite[Remark~2]{AM96}.

\begin{example} \label{ex:R being an FFD does not imply that R[[X]] is an FFD}
	Let $F_1 \subsetneq F_2$ be a field extension of finite fields, and consider the integral domain $R = F_1 + YF_2[[Y]]$. Since $YF_2[[Y]]$ is a nonzero maximal ideal of $F_2[[Y]]$, the ring $R$ has the form of a $D+M$ construction, where $T = F_2[[Y]]$. Since $F_2[[Y]]$ is an FFD and $F_2^*/F_1^*$ is a finite group, it follows from Proposition~\ref{prop:FFD and D+M construction} that $R$ is an FFD. On the other hand, note that every element of $F_2 \setminus F_1$ is an almost integral element over $R$. Hence $R$ is not completely integrally closed. Thus, Proposition~\ref{prop:if R[[X]] is FFD, R is completely integrally closed} guarantees that $R[[X]]$ is not an FFD.
\end{example}

Next we characterize when the construction $A + XB[X]$ of an extension $A \subseteq B$ of integral domains yields FFDs. To do this, we need the finiteness of the group $U(B)/U(A)$, which can be easily verified to be stronger than the condition $U(A) = U(B) \cap A$.

\begin{prop} \emph(\cite[Proposition~3.1]{AeA99}\emph) \label{prop:FFD for polynomial-like extensions}
	Let $A \subseteq B$ be an extension of integral domains. Then $A + XB[X]$ is an FFD if and only if $B$ is an FFD and $U(B)/U(A)$ is a finite group.
\end{prop}

\begin{proof}
	Set $R = A + XB[X]$. For the direct implication, suppose that $R$ is an FFD. Since $XB[X]$ is a nonzero common ideal of $R$ and $B[X]$, it follows that $[R :_R B[X]] < \infty$. Hence Proposition~\ref{prop:FFD to ring extension} guarantees that the group $U(B[X])/U(R) = U(B)/U(A)$ is finite and $B[X]$ is an FFD. As a consequence,~$B$ is an FFD.
	\smallskip
	
	For the reverse implication, suppose that $B$ is an FFD and $U(B)/U(A)$ is finite. Since $B$ is an FFD, so is $B[X]$ by Theorem~\ref{thm:FFD for polynomial and Laurent ring}. Since $B[X]$ is an FFD and $(U(B[X]) \cap \text{qf}(R))/U(R) = U(B)/U(A)$ is finite, it follows from Proposition~\ref{prop:FFD underrings} that $R$ is an FFD.
\end{proof}

We have characterized SFFDs in Section~\ref{sec:classes and examples of BFDs and FFDs}. We are now in a position to give two more characterizations.

\begin{prop}  \emph(\cite[Theorem~5]{AM96}\emph) \label{prop:SFFD characterizations II}
	The following statements are equivalent for an integral domain~$R$.
	\begin{enumerate}
		\item[(a)] $R$ is an SFFD.
		\smallskip
		
		\item[(b)] For any set of indeterminates $\{X_i \mid i \in I\}$ over $R$, every subring of the polynomial ring $R[\{X_i \mid i \in I\}]$ is an SFFD.
		\smallskip
		
		\item[(c)] For any set of indeterminates $\{X_i \mid i \in I\}$ over $R$, every subring of the polynomial ring $R[\{X_i \mid i \in I\}]$ is an FFD.
		\smallskip
		
		\item[(d)] Every subring of $R[X]$ is an FFD.
	\end{enumerate}
\end{prop}

\begin{proof}
	(a) $\Rightarrow$ (b): Let $\{X_i \mid i \in I\}$ be a nonempty set of indeterminates over $R$, and let $T$ be a subring of $R_I = R[\{X_i \mid i \in I\}]$. Take $f \in T^*$. Since $R_I$ is an FFD by Corollary~\ref{cor:when polynomials in a collection of indeterminates are FFDs} and $U(R_I) = U(R)$ is finite by Proposition~\ref{prop:SFFD characterizations}, there are only finitely many divisors of $f$ in $R_I$. Therefore $f$ has only finitely many divisors in $T$. Thus, $T$ is an SFFD.
	\smallskip
	
	(b) $\Rightarrow$ (c): This is clear.
	\smallskip
	
	(c) $\Rightarrow$ (d): This is clear.
	\smallskip
	
	(d) $\Rightarrow$ (a): Since $R$ is a subring of $R[X]$, it is an FFD. Set $S = R_0 + XR[X]$, where $R_0$ is the prime subring of $R$. As $S$ is a subring of $R[X]$, it is an FFD. In addition, because $X \in [S :_S R[X]]$, it follows from Proposition~\ref{prop:FFD to ring extension} that $U(R[X])/U(S)$ is finite. Since $U(R[X]) = U(R)$ and $U(S) = U(R_0)$, the group $U(R)/U(R_0)$ is finite. Now the fact that $U(R_0)$ is finite immediately implies that $U(R)$ is finite. Thus, $R$ is an SFFD by Proposition~\ref{prop:SFFD characterizations}.
\end{proof}

The power series analog of Proposition~\ref{prop:FFD for polynomial-like extensions} does not hold, as the following example illustrates.

\begin{example}
	Let $F_1 \subsetneq F_2$ be an extension of finite fields. We have seen in Example~\ref{ex:R being an FFD does not imply that R[[X]] is an FFD} that $F_1 + YF_2[[Y]]$ is an FFD. Take $A = B = F_1 + YF_2[[Y]]$. Although $B$ is an FFD and the group $U(B)/U(A)$ is finite, $A + XB[[X]] = B[[X]]$ is not an FFD, as shown in Example~\ref{ex:R being an FFD does not imply that R[[X]] is an FFD}.
\end{example}

However, we can characterize when $A + XB[[X]]$ is an FFD using the condition that $B[[X]]$ is an FFD, which is stronger than $B$ being an FFD.

\begin{prop}  \emph(\cite[Proposition~3.3]{AeA99}\emph) \label{prop:FFD for power-series-like extensions}
	Let $A \subseteq B$ be an extension of integral domains. Then $A + XB[[X]]$ is an FFD if and only if $B[[X]]$ is an FFD and $U(B)/U(A)$ is a finite group.
\end{prop}

\begin{proof}
	For the direct implication, suppose that $R = A + XB[[X]]$ is an FFD. Since $XB[[X]]$ is a nonzero ideal of both $B[[X]]$ and $R$, it follows that $\big[ R :_R B[[X]] \big] \neq \{0\}$. Therefore $B[[X]]$ is an FFD and the group $U(B)/U(A) \cong U(B[[X]])/U(R)$ is finite by Proposition~\ref{prop:FFD to ring extension}.
	\smallskip
	
	For the reverse implication, suppose that $B[[X]]$ is an FFD and the group $U(B)/U(A)$ is finite. Since $B[[X]]$ is an FFD and the group $(U(B[[X]] \cap \text{qf}(R))/U(R) = U(B[[X]])/U(R) \cong U(B)/U(A)$ is finite, it follows from Proposition~\ref{prop:FFD underrings} that $R$ is an FFD.
\end{proof}

Let $A \subseteq B$ be an extension of integral domains. If $R = A + XB[X]$ is an FFD, then it follows from Proposition~\ref{prop:FFD for polynomial-like extensions} that $U(B)/U(A)$ is finite, and so $U(A) = U(B) \cap A$. Then Proposition~\ref{prop:FFD underrings} guarantees that $A$ is also an FFD because $U(A) = U(R) \cap \text{qf}(A)$. Similarly, $A$ is an FFD provided that $A + XB[[X]]$ is an FFD. We record this observation as a corollary.

\begin{cor} \emph(\cite[Remark~3.5]{AeA99}\emph)
	Let $A \subseteq B$ be an extension of integral domains. If either $A + XB[X]$ or $A + XB[[X]]$ is an FFD, then $A$ is an FFD.
\end{cor}

\smallskip
Now we return to study the integral domains $\aaa[X]$ and $\aaa[[X]]$ introduced in~\eqref{eq:tower domain rings}. This time, we focus our attention on the finite factorization property. To begin with, we give two sufficient conditions and one necessary condition for $\aaa[X]$ to be an FFD.

\begin{prop} \emph(\cite[Theorem~3.4.6 and Proposition~3.4.7]{pK99}\emph)  \label{prop:tower of domain; sufficient condition for FFD}
	The following statements hold.
	\begin{enumerate}
		\item The integral domain $A_0 + XA_1 + \dots + X^{n-1}A_{n-1} + X^nA_n[X]$ is an FFD for every $n \in \nn_0$ if and only if $\aaa[X]$ is an FFD.
		\smallskip
		
		\item If $U(A)/U(A_0)$ is finite and $A[X]$ is an FFD, then $\aaa[X]$ is an FFD.
	\end{enumerate}
\end{prop}

\begin{proof}
	(1) Set $R_n = A_0 + XA_i + \dots + X^{n-1}A_{n-1} + X^nA_n[X]$ for every $n \in \nn_0$. For the direct implication, assume that $R_n$ is an FFD for every $n \in \nn_0$. Take $f \in \aaa[X]^*$, and let $d$ be the degree of $f$. Clearly, every divisor of $f$ in $\aaa[X]$ belongs to $R_d$. Since $R_d$ is an FFD and $U(R_d) = U(A_0) = U(\aaa[X])$, the element $f$ has only finitely many non-associate divisors in $\aaa[X]$. Thus, $\aaa[X]$ is an FFD by Proposition~\ref{prop:FFD characterizations}.
	
	For the reverse implication, fix $m \in \nn_0$ and take $f \in R_m$. As in the previous paragraph, two polynomials are non-associate divisors of $f$ in $R_m$ if and only if they are non-associate divisors of $f$ in $\aaa[X]$. Because $\aaa[X]$ is an FFD, so is $R_m$ by Proposition~\ref{prop:FFD characterizations}.
	\smallskip
	
	(2) Observe that $\aaa[X]$ is a subring of $A[X]$, and $U(A[X])$ is contained in $\text{qf}(\aaa[X])$. Therefore $(U(A[X]) \cap \text{qf}(\aaa[X]))/U(\aaa[X]) = U(A)/U(A_0)$ is finite. As a result, it follows from Proposition~\ref{prop:FFD underrings} that $\aaa[X]$ is an FFD.
\end{proof}

If the chain of integral domains $(A_n)_{n \ge 0}$ stabilizes, then we can characterize when $\aaa[X]$ (or $\aaa[[X]]$) is an FFD.

\begin{prop}  \emph(\cite[Theorem~3.4.5 and Proposition~3.4.8]{pK99}\emph)
	If there is an $N \in \nn$ such that $A_n = A_N$ for every $n \ge N$, then the following statements hold.
	\begin{enumerate}
		\item $\aaa[X]$ is an FFD if and only if $A_N$ is an FFD and the group $U(A_N)/U(A_0)$ is finite.
		\smallskip
		
		\item $\aaa[[X]]$ is an FFD if and only if $A_N[[X]]$ is an FFD and the group $U(A_N[[X]])/U(\aaa[[X]])$ is finite.
	\end{enumerate}
\end{prop}

\begin{proof}
	(1) For the direct implication, assume that $\aaa[X]$ is an FFD. Because $\aaa[X] \subseteq A_N[X]$ and $X^N \in \big[ \aaa[X] :_{\aaa[X]} A_N[X] \big]$, Proposition~\ref{prop:FFD to ring extension} guarantees that the integral domain $A_N[X]$ is an FFD and the group $U(A_N[X])/U(\aaa[X]) = U(A_N)/U(A_0)$ is finite. Since $A_N[X]$ is an FFD, so is~$A_N$.
	
	Conversely, suppose that $A_N$ is an FFD and $U(A_N)/U(A_0)$ is finite. The ring of polynomials $A_N[X]$ is an FFD by Theorem~\ref{thm:FFD for polynomial and Laurent ring}. On the other hand, $U(A_N[X]) \subseteq \text{qf}(\aaa[X])$, and therefore, $(U(A_N[X]) \cap \text{qf}(\aaa[X]))/U(\aaa[X]) = U(A_N)/U(A_0)$ is finite. Thus, it follows from Proposition~\ref{prop:FFD underrings} that $\aaa[X]$ is an FFD.
	\smallskip
	
	(2) Assume first that $\aaa[[X]]$ is an FFD. Because $\aaa[[X]]$ is an FFD contained in $A_N[X]$ and $X^N \in \big[ \aaa[[X]] :_{\aaa[[X]]} A_N[[X]] \big]$, it follows from Proposition~\ref{prop:FFD to ring extension} that $A_N[[X]]$ is an FFD and also that the group $U(A_N[[X]])/U(\aaa[[X]])$ is finite.
	
	For the reverse implication, suppose that $A_N[[X]]$ is an FFD and $U(A_N[[X]])/U(\aaa[[X]])$ is finite. It is easy to verify that the quotient field of $\aaa[[X]]$ is $\text{qf}(A_N[[X]])$. As a result, we obtain that the group $(U(A_N[[X]]) \cap \text{qf}(\aaa[[X]]))/U(\aaa[[X]]) = U(A_N[[X]])/U(\aaa[[X]])$ is finite. Since $A_N[[X]]$ is an FFD, Proposition~\ref{prop:FFD underrings} guarantees that $\aaa[[X]]$ is an FFD.
\end{proof}

\smallskip
\subsection{Monoid Domains}
\label{subsec:monoid domains}

Let $R$ be an integral domain, and let $M$ be a torsion-free monoid. Since monoids here are assumed to be cancellative and commutative, it follows from \cite[Corollary~3.4]{rG84} that $M$ admits a compatible total order (indeed, every compatible partial order on $M$ extends to a compatible total order on $\gp(M)$ \cite[Theorem~3.1]{rG84}). Hence we tacitly assume that $M$ is a totally ordered monoid. We say that $f = \sum_{i=1}^n c_i X^{m_i} \in R[M]^*$ is represented in \emph{canonical form} if $c_i \neq 0$ for every $i \in  \ldb 1, n \rdb$ and $m_1 > \dots > m_n$. Observe that any element of $R[M]^*$ has a unique representation in canonical form. In this case, $\deg f = m_1$ is called the \emph{degree} of $f$, while $c_1$ and $c_1X^{m_1}$ are called the \emph{leading coefficient} and the \emph{leading term} of $f$, respectively. As it is customary for polynomials, we say that $f$ is a \emph{monomial} if $n = 1$.

Most of the results presented in this subsection were established by H. Kim in \cite{hK98,hK01}, where the interested reader can also find similar results concerning atomicity, the ACCP, and the unique factorization property. We start by discussing the bounded factorization property in the context of monoid domains.

\begin{prop}  \emph(\cite[Propositions~1.4 and~1.5]{hK01}\emph) \label{prop:BF in monoid domains}
	Let $R$ be an integral domain with quotient field~$K$, and let $M$ be a torsion-free monoid. Then the following statements hold.
	\begin{enumerate}
		\item If $R[M]$ is a BFD, then $R$ is a BFD and $M$ is a BFM.
		\smallskip
		
		\item If $R$ and $K[M]$ are both BFDs, then $R[M]$ is a BFD.
	\end{enumerate}
\end{prop}

\begin{proof}
	(1) Suppose that the monoid domain $R[M]$ is a BFD. It follows from \cite[Theorem~11.1]{rG84} that $U(R[M]) = \{uX^m \mid u \in U(R) \ \text{and} \ m \in U(M) \}$. Therefore $U(R) = U(R[M]) \cap R$, and it follows from Proposition~\ref{prop:BFD underrings} that $R$ is a BFD. To verify that $M$ is a BFM, first note that by virtue of \cite[Theorem~11.1]{rG84}, $a \in \ii(M)$ if and only if $X^a \in \ii(R[M])$. As a result, for every $b \in M \setminus U(M)$, the set $L_M(b)$ is finite if and only if the set $L_{R[M]}(X^b)$ is finite. This, together with the fact that $R[M]$ is a BFD, implies that $M$ is a BFM.
	\smallskip
	
	(2) Assume that $R$ and $T = K[M]$ are both BFDs. Proposition~\ref{prop:BFD characterizations} guarantees the existence of length functions $\ell_R \colon R^* \to \nn_0$ and $\ell_T \colon T^* \to \nn_0$ of $R^*$ and $T^*$, respectively. Now define the function $\ell \colon R[M]^* \to \nn_0$ by setting $\ell(f) = \ell_T(f)+ \ell_R(c)$, where $c$ is the leading coefficient of $f$. It is clear that every unit $uX^m$ of $R[M]$ is a unit of $T$ with $u \in U(R)$, and so $\ell (uX^m) = \ell_T(uX^m) + \ell_R(u) = 0$. Also, for polynomial expressions $f_1$ and $f_2$ in $R[M]^*$ with leading coefficients $c_1$ and $c_2$, respectively, $\ell(f_1 f_2) = \ell_T(f_1 f_2) + \ell_R(c_1 c_2) \ge (\ell_T(f_1) + \ell_R(c_1)) + (\ell_T(f_2) + \ell_R(c_2)) = \ell(f_1) + \ell (f_2)$. Thus, $\ell$ is a length function, and it follows from Proposition~\ref{prop:BFD characterizations} that $R[M]$ is a BFD.
\end{proof}
\smallskip

We have just seen that for an integral domain $R$ and a torsion-free monoid~$M$, the fact that $R[M]$ is a BFD guarantees that both $R$ and $M$ satisfy the corresponding property. If every nonzero element of $\gp(M)$ has type $(0, 0, \dots)$, then the reverse implication also holds, as part~(3) of the next theorem shows. A nonzero element $b$ of an abelian group $G$ has \emph{type} $(0,0, \dots)$ if there is a largest $n \in \nn$ such that the equation $nx = b$ is solvable in $G$.

\begin{theorem} \emph(\cite[Theorems~3.12, Proposition~3.14, and Theorem~3.15]{hK98}\emph) \label{thm:BF in monoid domains}
	Let $R$ be an integral domain, $F$ a field, $G$ a torsion-free abelian group whose nonzero elements have type $(0,0, \dots)$, and~$M$ a torsion-free monoid whose nonzero elements have type $(0,0, \dots)$ in $\emph{\gp}(M)$. Then the following statements hold.
	\begin{enumerate}
		\item $R[G]$ is a BFD if and only if $R$ is a BFD.
		\smallskip
		
		\item $F[M]$ is a BFD if and only if $M$ is a BFM.
		\smallskip
		
		\item $R[M]$ is a BFD if and only if $R$ is a BFD and $M$ is a BFM.
	\end{enumerate} 
\end{theorem}

\begin{proof}
	(1) If $R[G]$ is a BFD, then it follows from part~(1) of Proposition~\ref{prop:BF in monoid domains} that $R$ is a BFD. For the reverse implication, suppose that $R$ is a BFD, and let $K$ be the quotient field of $R$. The monoid domain $K[G]$ is a UFD by \cite[Theorem~7.12]{GP74}. In particular, $K[G]$ is a BFD, and so part~(2) of Proposition~\ref{prop:BF in monoid domains} ensures that $R[G]$ is a BFD.
	\smallskip
	
	(2) If $F[M]$ is a BFD, then it follows from part~(1) of Proposition~\ref{prop:BF in monoid domains} that $M$ is a BFM. Conversely, suppose that $M$ is a BFM. As every nonzero element of $\gp(M)$ has type $(0,0, \dots)$, the monoid domain $T = F[\gp(M)]$ is a UFD, and so a BFD, by \cite[Theorem~12]{GP74}. Since $M$ is a BFM and $T$ is a BFD, Propositions~\ref{prop:BFM characterization via length functions} and~\ref{prop:BFD characterizations} guarantee the existence of length functions $\ell_M \colon M \to \nn_0$ and $\ell_T \colon T^* \to \nn_0$, respectively. Define $\ell \colon F[M]^* \to \nn_0$ by $\ell(f) = \ell_T(f) + \ell_M(\deg f)$. One can easily verify that $\ell$ is a length function of $F[M]^*$ (see the proof of part~(2) of Proposition~\ref{prop:BF in monoid domains}). Hence $F[M]$ is a BFD by Proposition~\ref{prop:BFD characterizations}.
	\smallskip
	
	(3) The direct implication follows from part~(1) of Proposition~\ref{prop:BF in monoid domains}. For the reverse implication, suppose that $R$ is a BFD and $M$ is a BFM. The monoid domain $R[\gp(M)]$ is a BFD by part~(1), while the monoid domain $\text{qf}(R)[M]$ is a BFD by part~(2). Then it follows from Proposition~ \ref{prop:a locally finite intersection of BFDs is a BFD} that $R[M] = R[\gp(M)] \cap \text{qf}(R)[M]$ is a BFD.
\end{proof}

\begin{cor} \emph(\cite[Corollary~3.17]{hK98}\emph) \label{cor:BFD on monoid domains of fg monoids}
	Let $R$ be an integral domain, and let $M$ be a finitely generated torsion-free monoid. Then $R[M]$ is a BFD if and only if $R$ is a BFD.
\end{cor}

\begin{proof}
	Since $M$ is torsion-free and finitely generated, $\gp(M)$ is a torsion-free finitely generated abelian group, and so a free abelian group. Hence every nonzero element of $\gp(M)$ has type $(0,0,\dots)$. In addition, it follows from Corollary~\ref{cor:finitely generated monoids are FFMs} that $M$ is an FFM, and so a BFM. Hence the corollary is a consequence of part~(3) of Theorem~\ref{thm:BF in monoid domains}.
\end{proof}

In \cite{AJ15}, D.~D. Anderson and J.~R. Juett proved a version of part~(3) of Theorem~\ref{thm:BF in monoid domains}, where they assume that $M$ is reduced, but not that all nonzero elements of $\gp(M)$ have type $(0, 0,  \dots)$.

\begin{theorem} \emph(\cite[Theorem~13]{AJ15}\emph) \label{thm:BF from M and R to R[M] when M is reduced}
	Let $R$ be an integral domain, and let $M$ be a reduced torsion-free monoid. Then $R[M]$ is a BFD if and only if $R$ is a BFD and $M$ is a BFM.
\end{theorem}

\begin{proof}
	The direct implication follows from part~(1) of Proposition~\ref{prop:BF in monoid domains}. To argue the reverse implication, suppose that $R$ is a BFD and $M$ is a BFM. Propositions~\ref{prop:BFM characterization via length functions} and~\ref{prop:BFD characterizations} guarantee the existence of length functions $\ell_M \colon M \to \nn_0$ and $\ell_R \colon R^* \to \nn_0$, respectively. Define $\ell \colon R[M]^* \to \nn_0$ by $\ell(f) = \ell_M(\deg f) + \ell_R(c)$, where $c$ is the leading coefficient of~$f$. As $M$ is reduced, $U(R[M]) = U(R)$ and so $\ell(f) = \ell_R(f) = 0$ when $f \in U(R[M])$. In addition, if $f_1$ and $f_2$ in $R[M]^*$ have leading coefficients $c_1$ and $c_2$, respectively, then $\ell(f_1 f_2) = \ell_M(\deg f_1 + \deg f_2) + \ell_R(c_1 c_2) \ge \ell(f_1) + \ell(f_2)$. Hence $\ell$ is a length function, and $R[M]$ is a BFD by Proposition~\ref{prop:BFD characterizations}.
\end{proof}

With the notation as in Theorem~\ref{thm:BF from M and R to R[M] when M is reduced}, the monoid domain $R[M]$ may be a BFD (in fact, an SFFD) even when not every nonzero element of $\gp(M)$ has type $(0,0, \dots)$; see, for instance, Example~\ref{ex:FFD with a non-atomic localization} and \cite[Example~5.4]{AAZ90}.
\bigskip

Now we turn to discuss the finite factorization property in the context of monoid domains. The following result is parallel to Proposition~\ref{prop:BF in monoid domains}.

\begin{prop}  \emph(\cite[Propositions~1.4 and~1.5]{hK01}\emph) \label{prop:FFD monoid domains}
	Let $R$ be an integral domain with quotient field~$K$, and let $M$ be a torsion-free monoid. Then the following statements hold.
	\begin{enumerate}
		\item If $R[M]$ is an FFD, then $R$ is an FFD and $M$ is an FFM.
		\smallskip
		
		\item If $R$ and $K[M]$ are both FFDs, then $R[M]$ is an FFD.
	\end{enumerate}
\end{prop}

\begin{proof}
	(1) Suppose that $R[M]$ is an FFD. Since $U(R[M]) \cap K^* = U(R)$, it follows from Proposition~\ref{prop:FFD underrings} that $R$ is an FFD. On the other hand, since $a \in \ii(M)$ if and only if $X^a \in \ii(R[M])$ by \cite[Theorem~11.1]{rG84}, we find that $|Z_M(m)| = |Z_{R[M]}(X^m)| < \infty$ for every $m \in M$. As a consequence, $M$ is an FFM.
	\smallskip
	
	(2) Now assume that $R$ and $K[M]$ are both FFDs. Suppose, by way of contradiction, that there is an $f \in R[M]^*$ with infinitely many non-associate divisors in $R[M]$. Let $cX^m$ be the leading term of $f$. Since every divisor of $f$ in $R[M]$ is also a divisor of $f$ in $K[M]$ and $f$ has only finitely many non-associate divisors in $K[M]$ by Proposition~\ref{prop:FFD characterizations}, there must be a sequence $(f_n)_{n \in \nn}$ consisting of non-associate divisors of~$f$ in $R[M]$ such that $f_i K[M] = f_j K[M]$ for all $i, j \in \nn$. For every $n \in \nn$, let $c_n X^{m_n}$ be the leading term of~$f_n$. As $K[M]$ is an FFD, it follows from part~(1) that $M$ is an FFM. Because $m \in m_n + M$ for every $n \in \nn$ and~$M$ is an FFM, after replacing $(f_n)_{n \in \nn}$ by a suitable subsequence, one can assume that $\deg f_i + M = \deg f_j + M$ for all $i,j \in \nn$. Furthermore, after replacing $f_n$ by $X^{u_n}f_n$, where $u_n = \deg f_1 - \deg f_n$, one can assume that for every $n \in \nn$ there is a $k_n \in K$ such that $f_n = k_n f_1$. Clearly, $c_n$ divides $c$ in $R$ for every $n \in \nn$. In addition, if $c_i$ and $c_j$ are associates in~$R$, then  $k_i/k_j \in U(R)$ and so $f_i$ and $f_j$ are associates in $R[M]$, which implies that $i = j$. Thus, $c$ has infinitely many non-associate divisors in $R$, contradicting that $R$ is an FFD.
\end{proof}

As for the bounded factorization property, the converse of part~(1) of Proposition~\ref{prop:FFD monoid domains} holds provided that every nonzero element of $\gp(M)$ has type $(0, 0, \dots)$.

\begin{theorem} \emph(\cite[Theorems~3.21, Proposition~3.24, and Theorem~3.25]{hK98}\emph) \label{thm:FF in monoid domains}
	Let $R$ be an integral domain, $F$ a field, $G$ a torsion-free abelian group whose nonzero elements have type $(0,0, \dots)$, and~$M$ a torsion-free monoid whose nonzero elements have type $(0,0, \dots)$ in $\emph{\gp}(M)$. Then the following statements hold.
	\begin{enumerate}
		\item $R[G]$ is an FFD if and only if $R$ is an FFD.
		\smallskip
		
		\item $F[M]$ is an FFD if and only if $M$ is an FFM.
		\smallskip
		
		\item $R[M]$ is an FFD if and only if $R$ is an FFD and $M$ is an FFM.
	\end{enumerate} 
\end{theorem}

\begin{proof}
	(1) It follows from part~(1) of Proposition~\ref{prop:FFD monoid domains} that $R$ is an FFD when the monoid domain $R[G]$ is an FFD. Conversely, assume that $R$ is an FFD, and let $K$ be the quotient field of $R$. Since the monoid domain $K[G]$ is a UFD by \cite[Theorem~7.12]{GP74}, it is an FFD. As a result, $R[G]$ is an FFD by part~(2) of Proposition~\ref{prop:FFD monoid domains}.
	\smallskip
	
	(2) By part~(1) of Proposition~\ref{prop:FFD monoid domains}, $M$ is an FFM provided that $F[M]$ is an FFD. For the reverse implication, suppose that $M$ is an FFM and assume, by way of contradiction, that $F[M]$ is not an FFD. Take an $f \in F[M]^*$ having infinitely many non-associate divisors, and let $(f_n)_{n \in \nn}$ be a sequence of non-associate divisors of $f$ in $F[M]$. Since $M$ is an FFM and $\deg f_n$ is a divisor of $\deg f$ in $M$ for every $n \in \nn$, by virtue of Proposition~\ref{prop:FFM characterization via idf-monoids} we can assume that $\deg f_n = \deg f_1$ for every $n \in \nn$. The monoid domain $F[\gp(M)]$ is an FFM by \cite[Theorem~7.12]{GP74}. As $f_n$ is a divisor of $f$ in $F[\gp(M)]$ for every $n \in \nn$, Proposition~\ref{prop:FFD characterizations} guarantees the existence of distinct $i,j \in \nn$ such that $f_iF[\gp(M)] = f_j F[\gp(M)]$. Since $\deg f_i = \deg f_j$, it follows that $f_j = \alpha f_i$ for some $\alpha \in F$. Hence $f_i$ and $f_j$ are associates in $F[M]$, which is a contradiction.
	\smallskip
	
	(3) In light of part~(1) of Proposition~\ref{prop:FFD monoid domains}, $R$ is an FFD and $M$ is an FFM provided that $R[M]$ is an FFD. To argue the reverse implication, suppose that $R$ is an FFD and $M$ is an FFM. Note that $R[\gp(M)]$ is an FFD by part~(1) and $\text{qf}(R)[M]$ is an FFD by part~(2). Therefore Proposition~ \ref{prop:a locally finite intersection of FFDs is an FFD} guarantees that $R[M] = R[\gp(M)] \cap \text{qf}(R)[M]$ is an FFD.
\end{proof}

Parallel to Corollary~\ref{cor:BFD on monoid domains of fg monoids}, we obtain the following corollary, whose proof follows similarly.

\begin{cor} 
	Let $R$ be an integral domain, and let $M$ be a finitely generated torsion-free monoid. Then $R[M]$ is an FFD if and only if $R$ is an FFD.
\end{cor}

One can naturally generalize the notion of an SFFD to monoids. A monoid $M$ is called a \emph{strong finite factorization monoid} (or an \emph{SFFM}) if every element of $M$ has only finitely many divisors. Clearly, a reduced monoid is an SFFM if and only if it is an FFM.

\begin{prop}  \emph(\cite[Propositions~1.4 and~1.5]{hK01}\emph) \label{prop:SFFD monoid domains}
	Let $R$ be an integral domain with quotient field~$K$, and let $M$ be a torsion-free monoid. Then the following statements hold.
	\begin{enumerate}
		\item If $R[M]$ is an SFFD, then $R$ is an SFFD and $M$ is an SFFM.
		\smallskip
		
		\item If $R$ is an SFFD, $M$ is an SFFM, and $K[M]$ is an FFD, then $R[M]$ is an SFFD.
	\end{enumerate}
\end{prop}

\begin{proof}
	(1) Assume that $R[M]$ is an SFFD. By Proposition~\ref{prop:SFFD characterizations}, $U(R[M])$ is finite, so $U(R) \subseteq U(R[M])$ implies that $U(R)$ is also finite. On the other hand, $R$ is an FFD by Proposition~\ref{prop:FFD monoid domains}. Hence $R$ is an SFFD by Proposition~\ref{prop:SFFD characterizations}. As $R[M]$ is an SFFD, to verify that every element $m \in M$ has finitely many divisors in $M$, it suffices to observe that $m \in d + M$ if and only if $X^m \in X^d R[M]$.
	\smallskip
	
	(2) Assume that $R$ is an SFFD, $M$ is an SFFM, and $K[M]$ is an FFD. In particular, $R$ and $K[M]$ are FFDs, and so it follows from part~(2) of Proposition~\ref{prop:FFD monoid domains} that $R[M]$ is an FFD. On the other hand, $U(M)$ is finite because $M$ is an SFFM, and $U(R)$ is finite by Proposition~\ref{prop:SFFD characterizations}. Hence $U(R[M])$ must be finite. Thus, Proposition~\ref{prop:SFFD characterizations} ensures that $R[M]$ is an SFFD.
\end{proof}

As in the case of the bounded and finite factorization properties, we have the following result.

\begin{theorem} \emph(\cite[Propositions~3.28 and~3.30]{hK98}\emph) \label{thm:SFF in monoid domains}
	Let $R$ be an integral domain, $F$ a field, $G$ a torsion-free abelian group, and~$M$ a torsion-free monoid whose nonzero elements have type $(0,0, \dots)$ in $\emph{\gp}(M)$. Then the following statements hold.
	\begin{enumerate}
		\item $R[G]$ is an SFFD if and only if $R$ is an SFFD and $G$ is the trivial group.
		\smallskip
		
		\item $F[M]$ is an SFFD if and only if $F$ is a finite field and $M$ is an SFFM.
		\smallskip
		
		\item $R[M]$ is an SFFD if and only if $R$ is an SFFD and $M$ is an SFFM.
	\end{enumerate} 
\end{theorem}

\begin{proof}
	(1) The reverse implication follows immediately. For the direct implication, assume that $R[G]$ is an SFFD. By Proposition~\ref{prop:SFFD characterizations}, the set $U(R[G])$ is finite, and so $G$ must be a finite group. This, along with the fact that $G$ is torsion-free, ensures that $G$ is the trivial group. Hence $R = R[G]$ is an SFFD.
	\smallskip
	
	(2) This is an immediate consequence of part~(3) below.
	\smallskip 
	
	(3) It follows from part~(1) of Proposition~\ref{prop:SFFD monoid domains} that if $R[M]$ is an SFFD, then $R$ is an SFFD and~$M$ is an SFFM. For the reverse implication, suppose that $R$ is an SFFD and $M$ is an SFFM. Since $M$ is an SFFM, $U(M)$ must be finite. On the other hand, $R$ is an FFD and $U(R)$ is finite by Proposition~\ref{prop:SFFD characterizations}. Therefore $R[M]$ is an FFD by part~(3) of Theorem~\ref{thm:FF in monoid domains}. In addition, as $U(R)$ and $U(M)$ are finite, so is $U(R[M])$. Thus, $R[M]$ is an SFFD by virtue of Proposition~\ref{prop:SFFD characterizations}.
\end{proof}
\smallskip

In general, there seems to be no characterization (in terms of $R$ and $M$) for the monoid domains $R[M]$ that are BFDs (FFDs or SFFDs). In the same direction, the question of whether $R[M]$ satisfies ACCP provided that both $R$ and $M$ satisfy the same condition seems to remain open, although it has been positively answered in~\cite[Theorem~13]{AJ15} for the case when $M$ is reduced (a result parallel to Theorem~\ref{thm:BF from M and R to R[M] when M is reduced}). By contrast, it is known that $R[M]$ need not be atomic provided that both $R$ and~$M$ are atomic, even if $R$ is a field or $M = \nn_0$ (i.e., $R[M] = R[X]$); for more details about this last observation, see~\cite{CG19} and~\cite{mR93}.

\smallskip
A partially ordered set is \emph{Artinian} if it satisfies the descending chain condition, and it is \emph{narrow} if it does not contain infinitely many incomparable elements. For a ring $R$, a monoid $M$, and a partial order~$\le$ compatible with $M$, the \emph{generalized power series ring} $R[[X;M^{\le}]]$ is the ring comprising all formal sums $f = \sum_{m \in M} c_mX^m$ whose support $\{m \in M \mid c_m \neq 0\}$ is Artinian and narrow. D.~D. Anderson and J.~R. Juett have also investigated in \cite{AJ15} when the generalized power series ring $R[[X;M^{\le}]]$ is a BFD (or satisfies ACCP), obtaining in \cite[Theorem~17]{AJ15} a result analogous to Theorem~\ref{thm:BF from M and R to R[M] when M is reduced} but in the context of generalized power series rings.

\smallskip
\subsection{Graded Integral Domains}

We conclude this section by saying a few words about the bounded and finite factorization properties in graded integral domains.

Recall that an integral domain $R$ is $M$-graded for a torsion-free monoid $M$ provided that for every $m \in M$, there is a subgroup $R_m$ of the underlying additive group of $R$ such that the following conditions hold:
\begin{enumerate}
	\item $R = \bigoplus_{m \in M} R_m$ is a direct sum of abelian groups, and
	\smallskip
	
	\item $R_m R_n \subseteq R_{m+n}$ for all $m,n \in M$.
\end{enumerate}

The following proposition generalizes parts~(1) of Propositions~\ref{prop:BF in monoid domains} and~\ref{prop:FFD monoid domains} and can be proved in a similar manner.

\begin{prop} \emph(\cite[Proposition~2.1]{KKP04}\emph)
	Let $M$ be a torsion-free monoid and $R = \bigoplus_{m \in M} R_m$ be an $M$-graded integral domain. Then $R_0$ is a BFD (resp., an FFD, an SFFD) if $R$ is a BFD (resp., an FFD, an SFFD).
\end{prop}

Let $D$ be an integral domain with quotient field $K$, and let $I$ be a proper ideal of $D$. If $t$ is transcendental over $D$, then $R = D[It, t^{-1}]$ is called the (\emph{generalized}) \emph{Rees ring} of $D$ with respect to~$I$. Observe that the (generalized) Rees ring $R$ is a $\zz$-graded integral domain with quotient field $K(t)$. Various factorization properties of $R$ when $I$ is principal were studied by D.~D. Anderson and the first author in~\cite{AA95}. In order to generalize some of the results obtained in~\cite{AA95}, H. Kim, T. I. Keon, and Y. S. Park introduced in~\cite{KKP04} the notions of graded atomic domain, graded BFD, and graded FFD.

\begin{definition}
	Let $R$ be a graded integral domain.
		\begin{enumerate}
		\item $R$ is \emph{graded atomic} if every nonunit homogeneous element of $R^*$ is a product of finitely many homogeneous irreducibles in $R$.
		\smallskip
		
		\item $R$ is a \emph{graded BFD} if $R$ is graded atomic, and for every nonunit homogeneous element of $R^*$, there is a bound on the length of factorizations into homogeneous irreducibles.
		\smallskip
		
		\item $R$ is a \emph{graded FFD} if every nonunit homogeneous element of $R^*$ has only finitely many non-associate homogeneous irreducible divisors.
	\end{enumerate}
\end{definition}

We are in a position to characterize when a (generalized) Rees ring is a BFD (or an FFD).

\begin{prop} \emph(\cite[Proposition~2.5]{KKP04}\emph) \label{prop:BF and FF for generalized Rees rings}
	For an integral domain $D$ with a proper ideal $I$, assume that the (generalized) Rees ring $R = D[It, t^{-1}]$ is atomic and $t^{-1} \in \mathcal{P}(R)$. Then the following statements are equivalent.
	\begin{enumerate}
		\item[(a)] $R$ is a BFD (resp., an FFD).
		\smallskip
		
		\item[(b)] $R$ is a graded BFD (resp., a graded FFD).
		\smallskip
		
		\item[(c)] $D$ is a BFD (resp., an FFD).
	\end{enumerate}
\end{prop}

\begin{proof}
	(a) $\Rightarrow$ (b) $\Rightarrow$ (c): These implications follow immediately.
	\smallskip
	
	(c) $\Rightarrow$ (a): We will only prove the BFD part, as the FFD part follows similarly. Assume that~$D$ is a BFD. It follows from Corollary~\ref{cor:BFD for polynomial and power series rings} that $D[t,t^{-1}]$ is also a BFD. Because $t^{-1}$ is a prime in $R$, the multiplicative set $S$ it generates in $R$ is a splitting multiplicative set by Lemma~\ref{lem:when multiplicative sets generated by primes are SMS}. It is clear that $R_S = D[t,t^{-1}]$. Since $D[t,t^{-1}]$ is a BFD, part~(1) of Theorem~\ref{thm:underring localization BFD/FFD} guarantees that $R$ is also a BFD.
\end{proof}

\begin{remark}
	The statement of Proposition~\ref{prop:BF and FF for generalized Rees rings} still holds if one replaces being a BFD by satisfying ACCP and being a graded BFD by satisfying ACC on homogeneous principal ideals (see \cite[Proposition~2.5]{KKP04}).
\end{remark}

\bigskip
\section{Generalized Bounded and Finite Factorization Domains}
\label{sec:generalized BFDs and FFDs}

In this section, we present an abstraction of the unique and finite factorization properties based on an extended notion of a factorization. These ideas were introduced by D. D. Anderson and the first author in \cite{AA10}. In the same paper, they considered a similar abstraction for half-factoriality and other-half-factoriality (called quasi-factoriality in~\cite{AA10}) that we will not consider here.

Let $R$ be an integral domain, and let $r$ be a nonunit of $R^*$. An \emph{atomic factorization} of $r$ in~$R$ is an element $a_1 \cdots a_n$ of the free commutative monoid on $\ii(R)$ (i.e., a formal product of irreducibles up to order) such that $a_1 \cdots a_n = r$ in $R$. Note that, by definition, two atomic factorizations are \emph{not} identified up to associates. 

\begin{definition} \label{def:generalized FFDs}
	Let $R$ be an integral domain, and let $\approx$ be an equivalence relation on $\ii(R)$. Then we say that two atomic factorizations $a_1 \cdots a_m$ and $b_1 \cdots b_n$ in $R$ are $\approx$-\emph{equivalent} if $m=n$ and there is a permutation $\sigma$ of $\ldb 1,m \rdb$ such that $b_i \approx a_{\sigma(i)}$ for every $i \in \ldb 1,m \rdb$. 
	\begin{enumerate}
		\item $R$ is a $\approx$-\emph{CKD} if $R$ is atomic and has only finitely many irreducible elements up to $\approx$-equivalence.
		\smallskip
		
		\item $R$ is a $\approx$-\emph{FFD} if $R$ is atomic and every nonunit $r \in R^*$ has only finitely many factorizations in $R$ up to $\approx$-equivalence.
		\smallskip
		
		\item $R$ is a $\approx$-\emph{UFD} if $R$ is atomic and for every nonunit $r \in R^*$, any two factorizations of $r$ in $R$ are $\approx$-equivalent.
	\end{enumerate}
\end{definition}

With the notation as in Definition~\ref{def:generalized FFDs}, observe that when $\approx$ is the associate relation on $\ii(R)$, we recover the standard definitions of a CKD, an FFD, and a UFD from those of a $\approx$-CKD, a $\approx$-FFD, and a $\approx$-UFD, respectively. The following example is \cite[Example~2.6(a)]{AA10}.

\begin{example} \label{ex:initial example of a generalized FFD}
	Let $R$ be the ring of power series $\qq[[X]]$, and define the equivalence relation $\approx$ on $\ii(R) = \{\sum_{i=1}^\infty b_iX^i \in R \mid b_1 \neq 0 \}$ by setting $\sum_{i=1}^\infty b_iX^i \approx \sum_{i=1}^\infty c_iX^i$ whenever $b_1 c_1 > 0$. It can be readily verified that $R$ is a $\approx$-FFD and a $\approx$-CKD. In addition, $R$ is a UFD that is not a $\approx$-UFD. Note that the relation $\approx$ is strictly contained in the associate relation on $\ii(R)$.
\end{example}

 It is clear that if $R$ is a $\approx$-FFD, then $R$ is a BFD. We record this observation for future reference.

\begin{remark} \label{rem:generalized FFDs are BFDs}
		Let $R$ be an integral domain, and let $\approx$ be an equivalence relation on $\ii(R)$. If $R$ is a $\approx$-FFD, then $R$ is a BFD.
\end{remark}

Although when $\approx$ is the associate relation, the definitions of a $\approx$-FFD and a BFD are not equivalent, they may be equivalent for other choices of $\approx$. The next example is \cite[Example~2.1(c)]{AA10}.

\begin{example} \label{ex:generalized FFD for full relation}
		Let $R$ be an integral domain, and let $\approx$ be the full equivalence relation on $\ii(R)$, that is, $r \approx s$ for all $r,s \in \ii(R)$. Observe that two atomic factorizations of a nonunit in $R^*$ are $\approx$-equivalent if and only if they involve the same number of irreducibles. As a consequence, $R$ is a $\approx$-FFD if and only if $R$ is a BFD, and $R$ is a $\approx$-CKD if and only if $R$ is atomic.
\end{example}

A CKD (resp., an FFD, a UFD) may not be a $\approx$-CKD (resp., $\approx$-FFD, $\approx$-UFD). To illustrate this, we use \cite[Example~2.1(b)]{AA10}.

\begin{example} \label{ex:diagonal relation on I(R)}
		Let $R$ be an integral domain, and let $\approx$ be the diagonal relation on $\ii(R)$, that is, $a \approx b$ if and only if $a = b$ for all $a,b \in \ii(R)$.
	\begin{enumerate}
		\item Suppose that $R$ is a $\approx$-CKD. Because $R$ is atomic and $\ii(R)$ is finite, the multiplicative monoid~$R^*$ is finitely generated, and it follows from \cite{jI59} that $R^*$ is finite. In this case, $R$ is a field. Thus, a CKD containing an irreducible cannot be a $\approx$-CKD.
		\smallskip
		
		\item Suppose now that $R$ contains at least one irreducible. Then it is clear that $R$ is a $\approx$-UFD if and only if $R$ is a UFD and $U(R) = \{1\}$. Similarly, $R$ is a $\approx$-FFD if and only if $R$ is an FFD and $U(R)$ is finite (i.e., $R$ is an SFFD).
	\end{enumerate}
\end{example}

If $R$ is an integral domain and $\approx$ is an equivalence relation on $\ii(R)$, then every implication in Diagram~\eqref{diag:generalized UFD, FFD, and CKD} holds.

\begin{equation} \label{diag:generalized UFD, FFD, and CKD}
	\begin{tikzcd}[cramped]
															  & \approx \! \textbf{-UFD } \arrow[r, Rightarrow] & \ \approx \! \textbf{-FFD } \arrow[d, Rightarrow]  & \approx \! \textbf{-CKD } \arrow[d, Rightarrow]  \\
		\textbf{UFD } \arrow[r, Rightarrow] &\textbf{ \ FFD } \ \arrow[r, Rightarrow] 				 & \ \textbf{ BFD } \arrow[r, Rightarrow]   & \textbf{ atomic domain}
	\end{tikzcd}
\end{equation}
\smallskip

For an integral domain $R$, we let $\sim$ be the associate relation on $\ii(R)$.

\begin{prop}  \emph(\cite[Theorem~2.5]{AA10}\emph) \label{prop:generalized FFD sufficient conditions}
	Let $R$ be an integral domain, and let $\approx$, $\approx_1$, and $\approx_2$ be equivalence relations on $\ii(R)$. Then the following statements hold.
	\begin{enumerate}
		\item If $\approx_1 \, \subseteq \, \approx_2$ and $R$ is a $\approx_1$-FFD, then $R$ is a $\approx_2$-FFD. In particular, if $R$ is an FFD and $\sim \, \subseteq \, \approx$, then $R$ is a $\approx$-FFD.
		\smallskip
		
		\item If $R$ is a $\approx$-CKD and a BFD, then $R$ is a $\approx$-FFD.
	\end{enumerate}
\end{prop}

\begin{proof}
	(1) The first statement is a direct consequence of part~(2) of Definition~\ref{def:generalized FFDs}, while the second statement is a special case of the first statement.
	\smallskip
	
	(2) Let $r$ be a nonunit of $R^*$. Since $R$ is a $\approx$-CKD, for every $\ell \in \nn$, the element $r$ has only finitely many atomic factorizations that are non-equivalent with respect to $\approx$ and involve exactly $\ell$ irreducibles. This, along with the fact that $R$ is a BFD, implies that $r$ has only finitely many atomic factorizations up to $\approx$-equivalence. Thus, $R$ is a $\approx$-FFD. 
\end{proof}

In part~(1) of Proposition~\ref{prop:generalized FFD sufficient conditions}, we observe that an FFD $R$ can also be a $\approx$-FFD for an equivalence relation on $\ii(R)$ satisfying $\approx \, \subsetneq \, \sim$ (see, for instance, Example~\ref{ex:initial example of a generalized FFD}). On the other hand, none of the conditions in the hypothesis of part~(2) of Proposition~\ref{prop:generalized FFD sufficient conditions} is superfluous. In addition, although every $\approx$-FFD is a BFD (for any relation $\approx$ on the set of irreducibles), the reverse implication of part~(2) of Proposition~\ref{prop:generalized FFD sufficient conditions} does not hold. The following examples, which are part of \cite[Example~2.6]{AA10}, illustrate these observations.

\begin{example} \hfill 
	\begin{enumerate}
		\item Since every CKD is an FFD, it follows from part~(1) of Example~\ref{ex:diagonal relation on I(R)} that any CKD $R$ containing at least one irreducible is a BFD that is not a $\approx$-CKD when $\approx$ is taken to be the diagonal relation on $\ii(R)$.
		\smallskip
		
		\item Consider the additive submonoid $M = \langle 1/p \mid p \in \pp \rangle$ of $\qq_{\ge 0}$, and let $R$ be the monoid domain $\qq[M]$. We have seen in Example~\ref{ex:ACCP domain that is not a BFD} that $R$ satisfies ACCP but is not a BFD. In addition, we have seen in Example~\ref{ex:generalized FFD for full relation} that when $\approx$ is the full relation $\ii(R)^2$, the integral domain $R$ is a BFD if and only if it is a $\approx$-FFD and also that $R$ is atomic if and only if it is a $\approx$-CKD. As a result, $R$ is a $\approx$-CKD that is not a $\approx$-FFD.
		\smallskip
		
		\item To see that the converse of part~(2) of Proposition~\ref{prop:generalized FFD sufficient conditions} does not hold, it suffices to take an FFD that is not a CKD, for instance, the ring of integers $\zz$.
	\end{enumerate}
\end{example}

The following theorem describes how the extended notion of a $\approx$-FFD behaves with respect to the $D+M$ construction.

\begin{theorem}  \emph(\cite[Theorem~2.10]{AA10}\emph)
	Let $T$ be an integral domain, and let $K$ and $M$ be a subfield of~$T$ and a nonzero maximal ideal of $T$, respectively, such that $T = K + M$. For a proper subfield $k$ of~$K$, set $R = k + M$. Let $\approx$ be an equivalence relation on $\ii(T)$, and set $\approx' \, = \, \approx \cap \ii(R)^2$. Then the following statements hold.
	\begin{enumerate}
		\item If $T$ is quasilocal, then $R$ is a $\approx$-FFD if and only if $T$ is a $\approx$-FFD.
		\smallskip
		
		\item If $T$ is not quasilocal, then $R$ is a $\approx'$-FFD if $T$ is a $\approx$-FFD.
	\end{enumerate}
\end{theorem}

\begin{proof}
	(1) Since $T$ is quasilocal, $R$ is quasilocal and $\ii(R) = \ii(T) \subseteq M$ by Lemma~\ref{lem:irreducibles of the D+M construction}, and then one can easily see that $\mathcal{P}(R)$ is empty. In addition, we have seen in the proof of Proposition~\ref{prop:pullback quasilocal BF} that~$R$ is atomic if and only if $T$ is atomic. As a consequence, $R$ is a $\approx$-FFD if and only if~$T$ is a $\approx$-FFD. This, together with the fact that $\ii(R) = \ii(T)$, guarantees that $R$ is $\approx'$-FFD if and only if~$T$ is a $\approx$-FFD. 
	\smallskip
	
	(2) Suppose now that $T$ is not quasilocal. In this case, $R$ is not quasilocal. Once again, it follows from Lemma~\ref{lem:irreducibles of the D+M construction} that $\ii(R) = \ii(T) \cap R$, and one can check that $\mathcal{P}(R) = (\mathcal{P}(T) \cap R) \setminus M$ (in this case, $\mathcal{P}(R)$ may be nonempty). Then $R$ is a $\approx'$-FFD if $T$ is a $\approx$-FFD.
\end{proof}

There are integral domains $R$ with a relation $\approx$ on $\ii(R)$ such that $R$ is a $\approx$-FFD, but $R$ is neither a $\approx$-UFD nor an FFD.

\begin{example}
	Consider the monoid domain $R = \qq[M]$, where $M$ is the additive monoid $\{0\} \cup \qq_{\ge 1}$. We have already seen in Example~\ref{ex:BFD that is neither an HFD nor an FFD} that $R$ is a BFD that is neither an FFD nor an HFD. Observe that the monoid domain $R[Y]$ is a BFD by Corollary~\ref{cor:BFD for polynomial and power series rings} and that $R[Y]$ is not an FFD (resp., an HFD) because $R$ is not an FFD (resp., an HFD). Finally, note that if $T$ is the DVR we obtain by localizing $R[Y]$ at the maximal ideal $YR[Y]$ and $\approx$ denotes the equivalence relation on $R[Y]$ defined by being associates in $T$, then $R[Y]$ is a $\approx$-FFD that is not a $\approx$-UFD.
\end{example}

Lastly, we determine when the polynomial ring $R[X]$ is a $\sim_{K[X]}$-FFD, where two elements of $R[X]$ are related with respect to $\sim_{K[X]}$ whenever they are associates in $K[X]$ (here $K$ is the quotient field of $R$).

\begin{theorem}  \emph(\cite[Theorem~2.14]{AA10}\emph)
	Let $R$ be an atomic integral domain with quotient field $K$. Then $R[X]$ is a $\sim_{K[X]}$-FFD if and only if $R$ is a BFD.
\end{theorem}

\begin{proof}
	Let $\approx$ denote $\sim_{K[X]}$. For the direct implication, it suffices to note that if $R[X]$ is a $\approx$-FFD, then it is a BFD, and so $R$ must be a BFD.
	
	Conversely, suppose that $R$ is a BFD. It follows from Theorem~\ref{thm:BF in polynomial and power series rings} that $R[X]$ is also a BFD. Take a nonunit $f \in R[X]$, and take $\ell \in \nn$ such that $\max L_{R[X]}(f) < \ell$. Observe that if two atomic factorizations of $f$ are $\approx$-equivalent, then they must contain the same number of irreducibles in $R$ and the same number of irreducibles in $R[X] \setminus R$. For $m,n \in \nn_0$ such that $m+n \le \ell$, suppose that $c_1 \dots c_m f_1 \dots f_n$ and $d_1 \dots d_m g_1 \dots g_n$ with $c_i, d_i \in R$ and $f_j, g_j \in R[X] \setminus R$, are two atomic factorizations of $f$ in $R[X]$. If these factorizations are $\approx$-equivalent, then, after a possible reordering, $f_i K[X] = g_i K[X]$. Since both $f_i$ and $g_i$ divide $f$ for every $i \in \ldb 1, n \rdb$ and the set $\{h K[X] \mid f \in h R[X] \}$ is finite, we can conclude that $f$ has only finitely many factorizations up to $\approx$-equivalence. Thus, $R[X]$ is a $\approx$-FFD.
\end{proof}

\bigskip
\section*{Acknowledgments}

While working on this paper, the second author was supported by the NSF award DMS-1903069.

\bigskip


\begin{thebibliography}{20}
	
	\bibitem{AA92} D.~D. Anderson and D.~F. Anderson: \emph{Elasticity of factorizations in integral domains}, J. Pure Appl. Algebra \textbf{80} (1992) 217--235.
	
	\bibitem{AA10} D.~D. Anderson and D.~F. Anderson: \emph{Factorization in integral domains IV}, Comm. Algebra {\bf 38} (2010) 4501--4513.
	
	\bibitem{AA95} D.~D. Anderson and D.~F. Anderson: \emph{The ring $R[X, \frac{r}{X}]$}. In: \emph{Zero-dimensional Commutative Rings} (Eds. D. F. Anderson and D. E. Dobbs) pp. 95--113, Lecture Notes in Pure and Applied Mathematics, vol. 171, Marcel Dekker, New York 1995.
	
	\bibitem{AAZ90} D.~D. Anderson, D.~F. Anderson, and M.~Zafrullah: \emph{Factorization in integral domains}, J. Pure Appl. Algebra {\bf 69} (1990) 1--19.
	
	\bibitem{AAZ92} D.~D. Anderson, D.~F. Anderson, and M.~Zafrullah: \emph{Factorization in integral domains II}, J. Algebra {\bf 152} (1992) 78--93.
	
	\bibitem{AAZ91} D.~D. Anderson, D.~F. Anderson, and M.~Zafrullah: \emph{Rings between $D[X]$ and $K[X]$}, Houston J. Math. {\bf 17} (1991) 109--129.
	
	\bibitem{AJ17} D.~D. Anderson and J.~R. Juett: \emph{Length functions in commutative rings with zero divisors}, Comm. Algebra \textbf{45} (2017) 1584--1600.
	
	\bibitem{AJ15} D.~D. Anderson and J.~R. Juett: \emph{Long length functions}, J. Algebra \textbf{426} (2015) 327--343.
	
	\bibitem{AM92} D.~D. Anderson and J.~L Mott: \emph{Cohen-Kaplansky domains: Integral domains with a finite number of irreducible elements}, J. Algebra \textbf{148} (1992) 17--41.
	
	\bibitem{AM96} D.~D. Anderson and B. Mullins: \emph{Finite factorization domains}, Proc. Amer. Math. Soc. \textbf{124} (1996) 389--396.
	
	\bibitem{dfA97} D.~F. Anderson: \emph{Elasticity of factorizations in integral domains: A survey}. In: \emph{Factorization in Integral Domains} (Ed. D. D. Anderson) pp. 1--29, Lecture Notes in Pure and Applied Mathematics, vol. 189, Marcel Dekker, New York 1997.
	
	\bibitem{AeA99} D.~F. Anderson and D.~N. El~Abidine: \emph{Factorization in integral domains III}, J. Pure Appl. Algebra {\bf 135} (1999) 107--127.
	
	\bibitem{BIK97} V. Barucci, L. Izelgue, and S. Kabbaj: \emph{Some factorization properties of $A + XB[X]$ domains}. In: \emph{Commutative Ring Theory}  (Eds. P. Cahen, M. Fontana, E. Houston, S. Kabbaj) pp. 227--241, Lecture Notes in Pure and Applied Mathematics, vol. 185, Marcel Dekker, New York 1997.
	
	\bibitem{aB65} A. Brandis: \emph{\"Uber die multiplikative struktur von K\"orpererweiterungen}, Math. Z. \textbf{87} (1965) 71--73.
	
	\bibitem{BR76} J. Brewer and E. A. Rutter: \emph{$D+M$ constructions with general overrings}, Michigan Math. J. \textbf{23} (1976) 33--42.
	
	
	\bibitem{CC00} S.~T. Chapman and J. Coykendall: \emph{Half-factorial domains, a survey}. In: \emph{Non-Noetherian Commutative Ring Theory} (Eds. S.~T. Chapman and S. Glaz) pp. 97--115, Mathematics and Its Applications, vol. 520, Springer, Boston 2000.	
	
	\bibitem{CGG20a} S. T. Chapman, F. Gotti, and M.  Gotti: \emph{When is a Puiseux monoid atomic?}, To appear in Amer. Math. Monthly. Preprint on arXiv: https://arxiv.org/pdf/1908.09227.pdf
	
	\bibitem{lC66} L. Claborn: \emph{Every abelian group is a class group}, Pacific J. Math. {\bf 18} (1966) 219--222.
	
	\bibitem{CK46} I. S. Cohen and I. Kaplansky: \emph{Rings with a finite number of primes, I}, Trans. Amer. Math. Soc. {\bf 60} (1946) 468--477.
	
	\bibitem{pC68} P. M. Cohn: \emph{Bezout rings and and their subrings}, Proc. Cambridge Philos. Soc. {\bf 64} (1968) 251--264.
	
	\bibitem{CMZ86} D. Costa, J. L. Mott, and M. Zafrullah: \emph{Overrings and dimensions of general $D + M$ constructions}, J. Natur. Sci. Math. 26 (1986) 7--14.
	
	\bibitem{CG19} J. Coykendall and F. Gotti: \emph{On the atomicity of monoid algebras}, J. Algebra \textbf{539} (2019) 138--151.
	
	\bibitem{GGT19} A. Geroldinger, F. Gotti, and S. Tringali: \emph{On strongly primary monoids, with a focus on Puiseux monoids}, J. Algebra \textbf{567} (2021) 310--345. 
	
	\bibitem{GH06} A. Geroldinger and F. Halter-Koch: \emph{Non-unique Factorizations: Algebraic, Combinatorial and Analytic Theory}, Pure and Applied Mathematics Vol. 278, Chapman \& Hall/CRC, Boca Raton, 2006.
	
	\bibitem{rG84} R. Gilmer: \emph{Commutative Semigroup Rings}, The University of Chicago Press, Chicago, 1984.
	
	\bibitem{rG68} R. Gilmer: \emph{Multiplicative Ideal Theory}, Queen's Papers in Pure and Applied Mathematics, No. 12, Queen's Univ. Press, Kingston, Ontario, 1968.
	
	\bibitem{rG72} R. Gilmer: \emph{Multiplicative Ideal Theory}, Pure and Applied Mathematics, Vol. 12, Marcel Dekker, New York, 1972.
	
	\bibitem{GP74} R.~Gilmer and T.~Parker: \emph{Divisibility properties of semigroup rings}, Michigan Math. J. \textbf{21} (1974) 65--86.
	
	\bibitem{fG20a} F. Gotti: \emph{Atomic and antimatter semigroup algebras with rational exponents}. Preprint on arXiv: https://arxiv.org/pdf/1801.06779v3.pdf
	
	\bibitem{fG19} F. Gotti: \emph{Increasing positive monoids of ordered fields are FF-monoids}, J. Algebra \textbf{518} (2019) 40--56.
	
	\bibitem{fG20} F. Gotti: \emph{Irreducibility and factorizations in monoid rings}. In: \emph{Numerical Semigroups} (Eds. V. Barucci, S. T. Chapman, M. D'Anna, and R. Fr\"oberg) pp. 129--139, Springer INdAM Series, vol. \textbf{40}, Cham 2020.
	
	\bibitem{aG74} A.~Grams: \emph{Atomic rings and the ascending chain condition for principal ideals}, Math. Proc. Cambridge Philos. Soc. \textbf{75} (1974) 321--329.
	
	\bibitem{GW75} A.~Grams and H.~Warner: \emph{Irreducible divisors in domains of finite character}, Duke Math. J. \textbf{ 42} (1975) 271--284.
	
	\bibitem{fHK92} F. Halter-Koch: \emph{Finiteness theorems for factorizations}, Semigroup Forum \textbf{44} (1992) 112--117.
	
	\bibitem{jI59} J. R. Isbell: \emph{On the multiplicative semigroup of a commutative ring}, Proc. Amer. Math. Soc. \textbf{10} (1959) 908--909.
	
	\bibitem{pK99} P.~L. Kiihne: \emph{Factorization and ring-theoretic properties of subrings of $K[X]$ and $K[[X]]$}. PhD Dissertation, The University of Tennessee, Knoxville, 1999.
	
	\bibitem{hK98} H. Kim: \emph{Factorization in monoid domains}. PhD Dissertation, The University of Tennessee, Knoxville, 1998.
	
	\bibitem{hK01} H. Kim: \emph{Factorization in monoid domains}, Comm. Algebra \textbf{29} (2001) 1853--1869.
	
	\bibitem{KKP04} H. Kim, T.~I. Keon, and Y.~S. Park: \emph{Factorization and divisibility in generalized Rees rings}, Bull. Korean Math. Soc. \textbf{41} (2004) 473--482.
	
	\bibitem{jL19} J. D. LaGrange: \emph{Divisor graphs of a commutative ring}. In: \emph{Advances in Commutative Algebra} (Eds. A. Badawi and J. Coykendall) pp. 217--244, Springer Trends in Mathematics, Birkh\"auser, Singapore, 2019.
	
	\bibitem{LM72} K. B. Levitz and J. L. Mott: \emph{Rings with finite norm property}, Can. J. Math. \textbf{24} (1972) 557--565.
	
	\bibitem{mR93} M. Roitman: \emph{Polynomial extensions of atomic domains}, J. Pure Appl. Algebra \textbf{87} (1993) 187--199.
	
	\bibitem{jlS86} J.L. Steffan: \emph{Longueurs des d\'ecompositions en produits d'\'el\'ements irr\'eductibles dans un anneau de Dedekind}, J. Algebra \textbf{102} (1986) 229--236.
	
	\bibitem{rV90} R. Valenza: \emph{Elasticity of factorization in number fields}, J. Number Theory \textbf{36} (1990) 212--218.
	
	\bibitem{aZ76} A. Zaks: \emph{Half-factorial domains}, Bull. Amer. Math. Soc. {\bf 82} (1976) 721--723.
	
	\bibitem{aZ80} A.~Zaks: \emph{Half-factorial domains}, Israel J. Math. {\bf 37} (1980) 281--302.

\end{thebibliography}
\end{document}